\documentclass[12pt,a4paper]{article}

\usepackage{geometry}
\usepackage[utf8x]{inputenc}
\usepackage{amsmath}
\usepackage{amsthm}
\usepackage{float}
\usepackage{amssymb}
\usepackage{amsfonts}
\usepackage[english]{babel}
\usepackage{enumerate}
\usepackage[fixlanguage]{babelbib}
\usepackage{graphicx}
\usepackage{tikz}
\usepackage{mathtools}
\usepackage{graphics}
\usepackage{mathrsfs,hyperref}
\usepackage[shortlabels]{enumitem}
\usepackage{epstopdf}
\usepackage{bigints}
\usepackage{textgreek}
\usepackage{gfsporson}
\usepackage{amssymb}
\usepackage{amsmath}
\usetikzlibrary{matrix,arrows,decorations.pathmorphing}

\DeclareMathOperator{\supp}{supp}
\DeclareMathOperator{\curl}{curl}

\DeclareMathOperator{\dist}{dist}
\DeclareMathOperator{\Div}{div}
\DeclareMathOperator{\diam}{diam}

\newcommand{\res}{\mathop{\hbox{\vrule height 7pt width 0.5pt depth 0pt
\vrule height 0.5pt width 6pt depth 0pt}}\nolimits}
\newcommand{\black}{\color{black}}

\newcommand{\weakstar}{\stackrel{*}{\rightharpoonup}}
\newcommand{\weak}{\rightharpoonup}

\newcommand{\R}{\mathbb{R}}
\newcommand{\N}{\mathbb{N}}
\newcommand{\Om}{\Omega}
\newcommand{\eps}{\varepsilon}

\definecolor{Aquamarine}{cmyk}{0.82,0,0.30,0}
\definecolor{Orange}{cmyk}{0,0.61,0.87,0}
\definecolor{LimeGreen}{cmyk}{0.50, 0.5, 1, 0}

\theoremstyle{definition}
\newtheorem{teo}{Theorem}[section]
\newtheorem{lem}[teo]{Lemma}

\newtheorem{prop}[teo]{Proposition}

\newtheorem{cor}[teo]{Corollary}
\theoremstyle{plain} 
\newtheorem{defin}[teo]{Definition}

\newtheorem{oss}[teo]{Remark}

\numberwithin{equation}{section}
\title{Nonlinear three-dimensional derivation of line tension for dislocations: quadratic growth}
\author{Adriana Garroni, Roberta Marziani, Riccardo Scala}

\begin{document}

\maketitle
\begin{abstract}
	In this paper we derive a line tension model for dislocations in 3d starting from a geometrically nonlinear elastic energy with quadratic growth.
	In the asymptotic analysis,  as the amplitude of the Burgers vectors (proportional to the lattice spacing) tends to zero, we show that the elastic energy linearises and the line tension energy density, up to an overall constant rotation, is identified by the linearised cell problem  formula given in  \cite{C.G.O.}. 
\end{abstract}


\vspace{2mm}

\noindent {\bf AMS (MOS) subject clas\-si\-fi\-ca\-tion:}  49J45, 58K45, 74C05.

\section{Introduction}
\label{intro}
Dislocations are line defects in crystals originated by plastic slips. Their presence, motion, and interaction are considered the key ingredients in order to understand plastic behaviour of metals, as well as other important effects (e.g. the interface energy at grain boundaries). We refer to \cite{BBS,HL} for a general introduction to the subject.

In the last decades the mathematical community has shown an increasing interest in the analysis of models  for dislocations involving several different approaches and frameworks.
Here we focus on a three dimensional  semi-discrete variational model where dislocations can be seen as topological singularities of a continuum strain field. 
More precisely, while a deformed elastic body can be described by a deformation whose gradient represents locally the distortion of an undeformed reference configuration, in the presence of defects the relevant continuum variable is a field $\beta\in L^1(\Om;\R^{3\times3})$ (the strain field), which may be represented by a gradient only locally. Therefore the defects may be identified with the set in which the $\curl\beta$ is concentrated.   Precisely a distribution of dislocations in an elastic body $\Om\subset\R^3$ is represented by a matrix valued measure $\mu$ of the form
\begin{equation}\label{measureintro}
\mu=b\otimes t\mathcal{H}^1\res\gamma,	
\end{equation}
where $b\in\mathcal{B}$ is the (normalised) Burgers vector (a so to say vector valued multiplicity which describes the kinematics of the line defect) and $\mathcal{B}\subset\R^3$ is a discrete lattice generated by the set of admissible Burgers vectors  (which depends on the underline crystalline structure of the body), $\gamma$ is a closed curve in $\Om$ and $t$ its unit tangent vector. Therefore the strain in the presence of a distribution of dislocations $\mu$ satisfies
\begin{equation}\label{rotore}
\curl\beta=\eps\mu\quad\text{in $\Om$},
\end{equation}
in the sense of distributions.  
The small parameter $\eps$ represents here a length scale which is comparable with the lattice parameter and reveals the discrete nature of the model. The discreteness is indeed highlighted by the presence of the line $\gamma$ around which crystal defects occur at a microscopic scale, inside the so called core of the dislocation, while the continuum variable $\beta$ is actually an approximation of the discrete deformation far from the core. The coexistence of the continuum variable and the discrete one, which makes this model so to say semi-discrete,  is common feature of models showing topological defects 
(this is also the case of Ginzburg Landau models for superconductors or Landau De Gennes models for liquid crystals, see e.g. \cite{ABO,B.B.H.,B.B.O.,Je,SS,SS1}).

The validity of the continuum approximation though is limited to regimes in which the density of dislocations is not too high and in a region sufficiently far from the dislocations. It is known indeed that  an incompatible strain $\beta$ satisfying \eqref{rotore} diverges close to the dislocation line as
$$
|\beta(x)|\simeq\frac{1}{\dist(x,\supp\mu)},
$$
and in particular is not squared integrable.
In order to work with the continuum strain field it is then common to perform a regularisation of the elastic energy:  either removing the core of the dislocation from the energy by considering 
\begin{equation}\label{energyintro}
\int_{\Om_\varepsilon(\mu)} W(\beta)dx,
\end{equation}
with $\Om_\varepsilon(\mu):=\{x\in\Om:\dist(x,\supp\mu)>\varepsilon\}$ or by regularising directly the strain field enforcing the following alternative constraint
\begin{equation}
\curl\beta=\varepsilon \mu*\varphi_\varepsilon,
\end{equation}
with $\varphi_\varepsilon$ a  mollifier at scale $\eps$. We will use the latter which we refer to as the regularisation by mollification while  the former is the so called core cut-off regularisation.

In the framework of linear elasticity it is well known that the energy stored by a straight dislocation with Burgers vector $b$ and direction $t$ in a hollow cylinder $T_\varepsilon$ with inner  and outer radii  respectively $0<\varepsilon<R$, and height $h$  is given by 
$$ 
\int_{T_\varepsilon} \frac12\mathbb{C}\beta:\beta\ dx\simeq \Psi_0(b,t)h\log\frac R\varepsilon,
$$ 
where $\mathbb{C}$ denotes the elastic tensor and $\curl\beta=b\otimes  t\mathcal{H}^1\res\R t$ in $T_\varepsilon$. 

The function $\Psi_0$ is the so called self-energy per unit length of a straight infinity dislocation and, for any Burgers vector $b\in\R^3$ and any direction $t\in S^1$,  it is obtained classically by solving the elastic problem in the whole of $\R^3$ or it can be characterised, as in \eqref{sself}, by a suitable variational formula  (see \cite[Lemma 5.1]{C.G.O.}).
  
In \cite{C.G.O.} Conti, Garroni and Ortiz  have shown that for  general dislocation distributions $\mu$ satisfying an appropriate diluteness conditions (see Definition \ref{def_dilute2}) the rescaled energy
$$
 \frac{1}{|\log\varepsilon|}\int_{\Om_\varepsilon(\mu)}\frac12\mathbb{C}\beta:\beta\ dx\quad\text{for $\curl\beta=\mu$ in $\Om$},$$
$\Gamma$-converges to the line-tension energy
\begin{equation}\label{intro3}
\int_\gamma\tilde\Psi_0(b,t)\,d\mathcal{H}^1,
\end{equation}
where $\tilde\Psi_0$ is the $\mathcal{H}^1$-elliptic envelope of $\Psi_0$ and is obtained by a relaxation procedure in \cite{C.G.M.}. In particular the relaxation process may produce microstructures at mesoscopic scales.

The above mentioned result is one of the first rigorous asymptotic analysis of the elastic energy induced by dislocations in a quite general three dimensional framework (see also \cite{Hu}). Previous results were indeed confined to special geometries where the models could be treated in a two-dimensional framework and studied by means of $\Gamma$-convergence in different relevant energetic regimes. In these reduced models dislocations can be seen either as points in the cross section of a cylindrical domain (see \cite{CL, GLP, DLGP}) with a strong similarity with the case of filaments of currents in superconductors (\cite{SS,SS1,Je}) or as line confined to a single slip plane and their energy described by nonlocal phase field models (generalising the Peierls Nabarro model \cite{K.C.O.,K.O.,GM,CG,C.G.M.,C.G.MU.1}). 

In the context of two dimensional models with point singularities several authors have also studied fully discrete models for screw dislocations (\cite{Pons, ACP}) deriving the same asymptotic obtained by means of the semi-discrete models (which then turns out to be a robust approximation of the discrete framework) and considering further asymptotic expansions able to capture the interaction and to drive the evolution of systems of dislocations (\cite{ADLGP}, \cite{HuO1}, \cite{HuO2}, see also \cite{MPS} and \cite{MRS}).

The classical semi-discrete models of dislocations, including the ones mentioned above, are based on the assumption of the elastic far field be small and therefore of the corresponding elastic energy be linear, which then confines the analysis of systems with dilute dislocations in single grain bodies. A step forward in the direction of having more flexible models that in principle may be suitable for the description of grain boundaries effects or of large deformations in thin materials is to consider nonlinear energies which may incorporate invariance under rigid rotations. 
This was first done for non-coherent interfaces in rods by \cite{MP1,MP2,FPP1} where the nonlinearity of the energy, with $p$ growth and $p<2$, is also used as an alternative regularisation of the core region (see also  \cite{SV1,SV2}). 

In a two dimensional setting the use of a geometric nonlinear energy
has shown already quite interesting features (see \cite{s.z.,MSZ}). The prototype  energy is the following
$$
\int_\omega \dist^2(\beta, SO(2))dx,
$$
with $\omega$ representing the two dimensional cross section of the cylindrical body and the field $\beta \in L^1(\omega;\R^{2\times2})$ an incompatible field, describing the local distortion of the body in the context of plane elasticity,  with curl concentrated on a sum of Dirac masses representing the distribution of dislocations ($\curl\beta=\eps\sum_{i=1}^N b_i\delta_{x_i}$). The analysis performed in \cite{s.z.,MSZ} required an assumption of separation of scales (then removed in \cite{JJ}) that prevents dislocations for being too close in the scale of $\eps$.
A crucial ingredient is a rigidity estimate for incompatible fields proved in \cite{MSZ} which is the nonlinear counterpart of the Korn's inequality for incompatible fields proved in \cite{GLP}. In dilute regimes  (in particular in the logarithmic scale) the rigidity estimate shows that the strain is close to a given rotation and therefore in the asymptotics as the lattice spacing tends to zero the rescaled energy linearises around such a rotation
$\Gamma$-converging as $\varepsilon$ goes to zero, to 
\begin{equation}
    \label{intro5}
    \int_\omega\frac{1}{2}\mathbb{C}\beta:\beta dx+\sum_i\psi(Q^Tb_i)\quad\text{for}\quad\curl\beta=0,\,\,Q\in SO(2),
\end{equation}
where $\mathbb{C}=\frac{\partial^2W}{\partial F^2}(I)$ and $\psi$ is the same self-energy which is found in the asymptotic of the linear, above mentioned, semi-discrete two dimensional models. 
We stress here that the result in the limit is still single grain and, up to a rigid rotation, coincides with the one obtained in the linear case. Nevertheless the analysis can be  pushed to different regimes and in principle this models can allow for multi-grain structures (see  \cite{LL} for the derivation of the Shockley Read formula for small angles grain boundaries).

In the present paper we will combine  two important features: the geometric nonlinearity and a full three dimensional geometry. We will assume, as in \cite{C.G.O.}, that dislocations are separated in the sense of Definition \ref{def_dilute2} and derive via $\Gamma$-convergence the line-tension energy for a general 3d dimensional distribution of dislocations. 
We consider an energy density $W$ with quadratic growth and invariant by rigid rotations, then behaving  as the $\dist^2(\cdot,SO(3))$, and to any distribution of dislocations $\mu$ of the form \eqref{measureintro} in $\Om$ we associate the rescaled energy
\begin{equation}\label{model-3d}
\mathcal{E}_\varepsilon(\beta):=\frac{1}{\varepsilon^2|\log\varepsilon|}\int_\Om W(\beta)dx,
\end{equation}
where the incompatible deformation field $\beta\colon\Om\rightarrow\R^{3\times3}$ satisfies
$
\curl\beta=\varepsilon \mu*\varphi_\varepsilon.
$
The $\Gamma$-convergence analysis shows that the limit functional takes the form
\begin{align}
\int_\Omega\frac{1}{2}\mathbb{C}\beta:\beta\ dx+\int_\gamma\tilde{\Psi}_0(Q^Tb,t)d\mathcal{H}^1,
    \label{functionallimit}
\end{align}
where  $(\beta,Q)\in L^2(\Omega;\mathbb{R}^{3\times3})\times SO(3)$, $\curl\beta=0$ in $\Omega$, $\tilde\Psi_0$ is the self energy of the linear case (see \eqref{intro3}), and $\gamma$ is the support of the limit dislocation density $\mu=b\otimes t\mathcal{H}^1\res\gamma$. 
The precise statements of the results are given in Section \ref{sec1}..

We stress that in the general three dimensional model the geometry of line dislocations makes the problem of removing the separation scales from the definition of admissible configuration, as it was done in the two dimensional case \cite{JJ}, substantially more complex (as well as for the linear case in \cite{C.G.O.}). In particular the crucial use of the rigidity estimate here (and of Korn's inequality in the linear context), does not permit to adopt directly the strategy used in the context of Ginzburg Landau (\cite{ABO}). Therefore how to obtain a general compactness result and a sharp lower bound for the energy in \eqref{model-3d} is still an open problem.

Finally we also point out that the combination of the three-dimensional framework with the geometric nonlinearity is far from being a straightforward adaptation of the techniques used in \cite{C.G.O.} and in \cite{s.z.}. In particular the proof of the lower bound were linearisation, concentration, and relaxation must be dealt all at once, requires to show precise quantitative estimates for the optimal energy of straight dislocations. A key step is then the introduction in Section \ref{sec2} of an auxiliary cell problem formula which introduces an extra parameter and allows to separate the linearisation (guaranteed by the rigidity estimate and the subsequent compactness) from the relaxation process (see Subsection \ref{aux-cell}).

\section{The model and the results}\label{sec1}      
In what follows we denote by $\Omega$ the material reference configuration, that is a simply connected, bounded domain in $\mathbb{R}^3$ {with boundary of class $C^2$}. \\
We identify a distribution of dislocations as a matrix valued measure supported on a one dimensional subset of $\Omega$. In this framework the topological nature of these defects can be easily translated in the property for these measures to be divergence-free in the sense of distributions.
Precisely the set of \textit{admissible dislocation densities} is 
 the set of all divergence-free bounded measures $\mu\in\mathcal{M}(\Omega;\mathbb{R}^{3\times3})$ of the form
\begin{equation*}
    \mu=b\otimes t\mathcal{H}^1\res\gamma,
\end{equation*}
with $\gamma$ a 1-rectifiable subset of $\Omega$,  $t\colon\gamma\to S^2$ its tangent vector, and $b\in L^1(\gamma;\mathcal{B};\mathcal{H}^1\res\gamma)$ the Burgers vector field. Here $\mathcal{B}\subset\mathbb{R}^3$ is a discrete lattice and represents the space of \textit{admissible renormalised Burgers vectors} (e.g. in the case of the cubic crystal $\mathcal{B}=\mathbb{Z}^3$). In particular without loss of generality we will assume that
\begin{equation}\label{normal}
\min\{|b|:\ b\in \mathcal{B}\}=1.
\end{equation}

The divergence-free conditions reads as
\begin{equation*}
    \int_\gamma b\cdot(D\phi)td\mathcal{H}^1=0,
\end{equation*}
for all $\phi\in C_0^\infty(\Omega;\mathbb{R}^3)$.
We will denote this set of admissible dislocation densities as $\mathcal{M}_\mathcal{B}(\Omega)$, i.e.,
\begin{align}
\mathcal{M}_\mathcal{B}(\Omega):=\Bigl\{\mu\in\mathcal{M}(\Omega;\R^{3\times3}):\ & \mu=b\otimes t\mathcal{H}^1\res\gamma,\ \notag\\ &\Div\mu=0,\ b\in \mathcal{B},\ \gamma \ 1-\hbox{rectifiable} \Bigr\}.
\end{align}

\begin{oss} It can be seen (see for instance \cite[Theorem 2.5]{C.G.M.}) that for each $\mu\in\mathcal{M}_\mathcal{B}(\Omega)$ $\gamma$ must be the union of countable number of Lipschitz curves with no endpoints in $\Omega$, $b$ must be constant on each connected component of $\gamma$ away from branching points, and in each branching point the oriented sum of Burgers vector must be zero.
\end{oss}
In order to associate a semi-discrete elastic energy to a given distribution of dislocations, as already mentioned, we need to regularise the problem inside the core, i.e. at scale $\varepsilon$ proportional to the lattice spacing. Among different types of regularisations (that in our analysis give rise to the same asymptotics, see \cite{C.G.O.}) we decide to introduce a mollification kernel which has the effect of spreading the mass of  $\mu$ in a neighbourhood of order $\varepsilon$ of its support and therefore smearing out the singularity of the corresponding strain. Then we define the class
of admissible strains associated with any $\mu\in\mathcal{M}_\mathcal{B}(\Omega)$ as
\begin{equation}\label{admiss_measures}
    \mathcal{AS}_\varepsilon(\mu):=\bigl\{\beta\in L^2(\Omega;\mathbb{R}^{3\times3})\,:\,\curl\beta=\varepsilon\tilde{\mu}*\varphi_\varepsilon\quad\text{in }\quad\Omega
      \bigr\},
\end{equation}
where $\tilde \mu\in \mathcal{M}_\mathcal{B}(\R^3)$ is an extension of $\mu$, that is $\tilde\mu\res\Om=\mu$,  $\varphi_\varepsilon(x):=\varepsilon^{-3}\varphi(x/\varepsilon)$ is a mollifier, and $\curl\beta$ is intended in a distributional sense. 
To simplify the arguments we will assume that $\varphi\le C \chi_{B_{1}(0)}$, so that $\varphi_\varepsilon(x)\leq c\frac{\chi_{B_{\varepsilon}(0)}}{|B_{\varepsilon}(0)|}$.

Note that the scaling $\varepsilon$ of the dislocation density $\tilde\mu*\varphi_\varepsilon$ reflects the fact that dislocations are defects at the atomic scale, and the support of $\tilde\mu*\varphi_\eps$ represents the dislocation core. \\

To any admissible pair $(\mu,\beta)\in\mathcal{M}_\mathcal{B}(\Omega)\times\mathcal{AS}_\varepsilon(\mu)$, we associate an energy of the form 
\begin{equation}
\label{energy1}
\mathcal{E}_\varepsilon(\mu,\beta):=\int_\Omega W(\beta)dx,
\end{equation}
where $W\colon\mathbb{R}^{3\times3}\to[0,+\infty]$ satisfies the classical assumptions for the geometrically nonlinear elastic setting, that is
\begin{enumerate}[(i)]
    \item $W$ is $C^1$ and $C^2$ in a neighbourhood of $SO(3)$;
    \item $W(I)=0$ (stress-free reference configuration);
    \item $W(RF)=W(F)$ for every $R\in SO(3)$ and $F\in\mathbb{R}^{3\times3}$ (frame indifference);
    \item there exist constants $C_1$, $C_2>0$ such that for every $F\in\mathbb{R}^{3\times3}$
    \begin{equation}
        \label{quadraticgrowth}
        C_1\dist^2(F,SO(3))\le W(F)\le C_2\dist^2(F,SO(3));
    \end{equation}
    \item there exists a constant $C>0$ such that for every $F\in\mathbb{R}^{3\times3}$
    \begin{equation}
        \label{jacob}
        |\frac{\partial W}{\partial F} (F)|\le C\dist(F,SO(3)),
    \end{equation}
    where $\frac{\partial W}{\partial F} (F)\in\mathbb{R}^{3\times3}$ is the Jacobian matrix of $W$ in $F$. 
\end{enumerate} 
The main goal of this paper is to study the asymptotic behaviour of the energy in a mesoscopic scale, i.e., a scale at which lines are still visible and in the asymptotics we recover a line tension. Under our assumption for the energy the natural rescaled functional is given by
$$ \frac{1}{\varepsilon^2|\log\varepsilon|}\mathcal{E}_\varepsilon(\mu,\beta)$$
(see also \cite{s.z.}). As in \cite{C.G.O.}, in order to perform the analysis, we need to assume a diluteness condition for the admissible dislocation densities (the analogous of this condition in 2d was also considered in \cite{s.z.} and recently removed in \cite{JJ}).

\begin{defin}\label{def_dilute2} Given two positive parameters $\alpha, h>0$, a dislocation measure $\mu\in\mathcal{M}_\mathcal{B}(\Omega)$, with $\Omega\subset\mathbb{R}^3$ open, is said to be $(h,\alpha)$-dilute if there are finitely many closed segments $\gamma_j\subset\Omega$ and vectors $b_j\in\mathcal{B}$, $t_j\in S^2$ (with $t_j$ tangent to $\gamma_j$) such that
\begin{equation*}
    \mu=\sum_jb_j\otimes t_j\mathcal{H}^1\res\gamma_j,
\end{equation*}
where the closed segments $\gamma_j$ satisfy the properties:
\begin{enumerate}[(a)]
    \item each $\gamma_j$ has length at least $h$;
    \item if $\gamma_j$ and $\gamma_k$ are disjoint then their distance is at least $\alpha h$;
    \item if the segments $\gamma_j$ and $\gamma_k$ are not disjoint then they share an endpoint, and the angle between them is at least $\alpha$.
\end{enumerate}
The set of $(h,\alpha)$-dilute measures is denoted by $\mathcal{M}_\mathcal{B}^{h,\alpha}(\Omega)$. \\
Moreover  we say that a measure $\mu\in\mathcal{M}_\mathcal{B}^{h,\alpha}(\Omega)$ is a $(h,\alpha)$-dilute  in $\bar\Om$ if also the following condition holds:

\begin{itemize}
	\item[(d)] If $\gamma_j\cap \partial\Om\neq\emptyset$, then this intersection consists of a single point $x_j$ and the angle between $\gamma_j$ and the tangent plane to $\partial\Om$ at $x_j$ is at least $\alpha$. Moreover, if  $\gamma_j\cap \partial\Om=\emptyset$ the distance between $\gamma_j$ and $\partial\Om$ is at least $h\alpha$.
\end{itemize}
We will denote this space by ${\mathcal M}_\mathcal{B}^{h,\alpha}(\bar\Om)$.
\end{defin}
Condition (d) in Definition \ref{def_dilute2} essentially guarantees that any $\mu\in\mathcal{M}^{h,\alpha}_\mathcal{B}(\Om)$ can be extended to $\tilde\mu\in\mathcal{M}_\mathcal{B}^{h,\alpha}(\R^3)$ still dilute. This is clearly the case when $\Om$ is half space and the extension is obtained by reflection (the general case is more delicate but can be obtained with some suitable generalisation stated in Lemma \ref{extension_2}). Nevertheless, we point out that this assumption is not restrictive. A dilute extension if $\Omega$ is of class $C^1$ can be always obtained with ad hoc construction that we do not give here.\\

In addition we choose the diluteness parameters $h$ and $\alpha$ much larger than the core radius $\varepsilon$, namely 
\begin{equation}
    \lim_{\varepsilon\to0}\frac{\log(1/(\alpha_\varepsilon h_\varepsilon))}{|\log\varepsilon|}= \lim_{\varepsilon\to0}\alpha_\varepsilon=\lim_{\varepsilon\to0}h_\varepsilon=0.\label{diluteness1}
\end{equation}
For technical reasons (see Proposition \ref{stimaL2}) we also require the following stronger diluteness condition
\begin{equation}\label{diluteness2}
	\alpha_\varepsilon^4h_\varepsilon^6|\log\varepsilon|>1.
\end{equation}


With this choice of the diluteness parameters we will show that the rescaled functionals 
\begin{equation}\label{functional1}
    \mathcal{F}_\varepsilon(\mu,\beta):=\begin{cases}
  \displaystyle \frac{1}{\varepsilon^2|\log\varepsilon|}\mathcal{E}_\varepsilon(\mu,\beta)&\text{if $(\mu,\beta)\in \mathcal{M}_\mathcal{B}^{h_\varepsilon,\alpha_\varepsilon}(\bar\Omega)\times\mathcal{AS}_\varepsilon(\mu)$}\\
   +\infty&\text{otherwise,}
    \end{cases}
\end{equation}
 $\Gamma$-converges to  the following functional
\begin{equation}    \label{functional2}
    \mathcal{F}_0(\mu,\beta,Q):= \begin{cases}
  \displaystyle \int_\Omega\frac{1}{2}\mathbb{C}\beta:\beta\ dx+\int_\gamma\tilde{\Psi}_0(Q^Tb,t)d\mathcal{H}^1&\text{if $(\mu,\beta,Q)\in \mathcal{AS}$}
   \\ +\infty&\text{otherwise,}
    \end{cases}
\end{equation}
where 
\begin{equation*}
	\mathcal{AS}:=\{(\mu,\beta,Q)\in\mathcal{M}_\mathcal{B}(\Om)\times L^2(\Om;\R^{3\times3})\times SO(3):\,\mu=b\otimes t\mathcal{H}^1\res\gamma,\,\curl\beta=0\,\text{in $\Om$}\}.
\end{equation*}
Here $\mathbb{C}:=\frac{\partial^2W}{\partial F^2}(I)$ is the Hessian of $W$ at the identity, $A:B=\sum_{i,j}A_{ij}B_{ij}$ denotes the Euclidean scalar product of matrices,
while $\tilde{\Psi}_0$ is the $\mathcal{H}^1$-elliptic envelope of $\Psi_0$, defined as
\begin{multline}
    \tilde{\Psi}_0(b',t):=\inf\biggl\{\int_\gamma\Psi_0(\theta(x),\tau(x))d\mathcal{H}^1(x):\,\nu=\theta\otimes\tau\mathcal{H}^1\res\gamma\in\mathcal{M}_{\mathcal{B}}(B_{1/2}(0)),\\
    \supp(\nu-b'\otimes t\mathcal{H}^1\res(\mathbb{R}t\cap B_{1/2}(0)))\subset\subset B_{1/2}(0)\biggr\}.
    \label{relaxedenergy}
\end{multline}
The function
$\Psi_0$ is defined in \eqref{sself} and \eqref{selfen1} below and  represents the self-energy per unit length of a straight infinite dislocation. Note that the dependence of the limit functional on a rotation $Q$ is due to the geometric nonlinear nature of the energy, namely to its frame indifference, and comes out from a Taylor expansion of the energy density $W$ near a constant rotation $Q$.\\
 
 The first important result concerns a compactness property for the dislocation measures and associated fields with equibounded energies.

\begin{teo}[Compactness]
\label{compactness}
 Let $\varepsilon_j\to0$ and $(h_{\varepsilon_j},\alpha_{\varepsilon_j})$  be as in \eqref{diluteness1}. If  $(\mu_j,\beta_j)\in{\mathcal{M}}_\mathcal{B}^{h_{\varepsilon_j},\alpha_{\varepsilon_j}}(\bar\Om)\times\mathcal{AS}_{\varepsilon_j}(\mu_j)$ is a sequence such that $\mathcal{F}_{\varepsilon_j}(\mu_j,\beta_j)\le C$, for some $C>0$, then the following hold:
\begin{enumerate}[(i)]
\item
There exists a measure $\mu\in\mathcal{M}_\mathcal{B}(\Om)$ such that, up to subsequence,
$$\mu_j\stackrel{*}{\rightharpoonup}\mu\quad\text{in}\quad \mathcal{M}_\mathcal{B}(\Om);$$

\item There exist a sequence $\{Q_j\}\subset SO(3)$ and $\beta\in L^2(\Om;\R^{3\times3})$ with $\curl\beta=0$ in $\Om$ such that, up to subsequence,
$$ \frac{Q_j^T\beta_j-I}{\varepsilon_j\sqrt{|\log\varepsilon_j|}}\chi_{\Om_j}\rightharpoonup\beta\quad\text{in}\quad L^2(\Om;\R^{3\times3})\quad\text{and}\quad Q_j\to Q\in SO(3),$$
where $\chi_{\Om_j}$ denotes the characteristic function of $\Om_j\colon=\{x\in\Omega:\dist(x,\partial\Omega)>\varepsilon_j\}$.
  
\end{enumerate}
\end{teo}


Since in the definition of admissible strains we do not specify how do we extend the measure outside $\Omega$ in order to define the regularised density, we cannot expect a control on the total mass of $\tilde\mu_j*\varphi_{\varepsilon_j}$. This is the reason why in the compactness of the strains (property (ii) of the above result) we need to remove a neighbourhood of the boundary. Whether the control on the energy of the $\beta$'s is enough to deduce compactness without any further assumption on the extension is not clear to us. 

In order to recover compactness in the all of $\Om$ we can fix a specific extension operator $ T\colon\mathcal{M}_\mathcal{B}(\Om)\to\mathcal{M}_\mathcal{B}(\R^3)$, and replace the class of admissible strains $\mathcal{AS}_\varepsilon(\mu)$ defined in \eqref{admiss_measures} with the class
\begin{equation}\label{ASstar}
\mathcal{AS}^\star_\varepsilon(\mu):=\{\beta\in L^2(\Om;\R^{3\times3})\,:\,\curl\beta=\varepsilon T\mu*\varphi_\varepsilon\,\text{in $\Om$}\}.
\end{equation} There are many possible choices of the extension operator $T$, one that will work is defined by Lemma \ref{extension_2}.
\begin{oss}
	If in Theorem \ref{compactness} we assume that $\beta_j$ belongs to $\mathcal{AS}^\star_\varepsilon(\mu_j)$ property (ii)  in Theorem \ref{compactness} can be replaced by the following 
	\begin{itemize}
		\item[(ii')] There exist a sequence $\{Q_j\}\subset SO(3)$ and $\beta\in L^2(\Om;\R^{3\times3})$ with $\curl\beta=0$ such that, up to subsequence,
		\begin{equation}\label{compactness2}
		\frac{Q_j^T\beta_j-I}{\varepsilon_j\sqrt{|\log\varepsilon_j|}}\rightharpoonup\beta\quad\text{in $L^2(\Om;\R^{3\times3})$}\quad\text{and}\quad Q_j\to Q\in SO(3). 
		\end{equation}
	\end{itemize}
\end{oss}
\begin{defin}\label{def_conv}
We say that $(\mu_j,\beta_j)\in {\mathcal M}^{h_{\varepsilon_j},\alpha_{\varepsilon_j}}_{\mathcal B}(\bar\Om)\times \mathcal {AS}_{\varepsilon_j}(\mu_j)$ converges to $(\mu,\beta,Q)\in \mathcal M_{\mathcal B}(\Om)\times L^2(\Om;\R^{3\times3})\times SO(3),$ if
(i) and (ii) of Theorem \ref{compactness} hold. 
\end{defin}
\begin{teo}[$\Gamma$-convergence]
\label{gammalimit}
The energy functional $\mathcal{F}_\varepsilon$ $\Gamma$-converges to $\mathcal{F}_0$ in the following sense.
\begin{enumerate}[(i)]
\item (Lower Bound). For any sequence $\varepsilon_j\to0$ and any $(\mu_j,\beta_j)\in {\mathcal{M}}_\mathcal{B}^{h_{\varepsilon_j},\alpha_{\varepsilon_j}}(\bar\Omega)\times\mathcal{AS}_{\varepsilon_j}(\mu_{\varepsilon_j})$ converging to $(\mu,\beta,Q)$ in the sense of Definition \ref{def_conv}, one has  
    \begin{equation*}
        \mathcal{F}_0(\mu,\beta,Q)\le\liminf_{j\to\infty}\mathcal{F}_{\varepsilon_j}(\mu_j,\beta_j);
    \end{equation*}
    \item (Upper Bound). For any $(\mu,\beta,Q)\in\mathcal{M}_\mathcal{B}(\Omega)\times L^2(\Omega;\mathbb{R}^{3\times3})\times SO(3)$ with $\curl\beta=0$ and any sequence $\varepsilon_j\to0$ there exists a sequence $(\mu_j,\beta_j)\in {\mathcal{M}}_\mathcal{B}^{h_{\varepsilon_j},\alpha_{\varepsilon_j}}(\bar\Omega)\times\mathcal{AS}_{\varepsilon_j}(\mu_{\varepsilon_j})$ converging to $(\mu,\beta,Q)$ in the sense of Definition \ref{def_conv} such that
 \begin{equation*}
     \limsup_{j\to\infty}\mathcal{F}_{\varepsilon_j}(\mu_j,\beta_j)\le\mathcal{F}_0(\mu,\beta,Q).
 \end{equation*}
\end{enumerate}
\end{teo}

\section{Asymptotics for straight dislocations }\label{sec2}
The elastic energy $\mathcal{F}_\varepsilon$ of a given polyhedral measure $\mu\in{\mathcal M}_\mathcal{B}^{h,\alpha}(\bar\Om)$ is asymptotically equivalent to the sum of the energy contributions of each segment. The latter is obtained by studying a cell problem, which provides the energy per unit length of a straight infinite dislocation. In the following we analyse a three-dimensional cell problem in a nonlinear framework combining and developing techniques used by \cite{C.G.O.} for the linear case and by \cite{s.z.} for the two-dimensional nonlinear case.
\subsection{Linear Cell Problem}
In the three-dimensional linear framework, corresponding to a given elasticity tensor $\mathbb{C}\colon\R^{3\times3}\to\R^{3\times3}$ linear, symmetric and positive definite, the line tension density was characterised in \cite{C.G.O.}. For completeness we give here the main results.\\
For $t\in S^2$ we fix a matrix 
\begin{equation}
    Q_t\in SO(3)\quad\text{such that $Q_te_3=t$},
    \label{notation1}
\end{equation}
and let 
\begin{equation}
\label{notation2}
\Phi_t(r,\theta,z):=Q_t(r\cos\theta,r\sin\theta,z),
\end{equation}
be the change of variables to cylindrical coordinates with axis $t$. The local basis in cylindrical coordinates is
\begin{equation}
    e_r:=(\cos\theta,\sin\theta,0),\quad e_\theta:=(-\sin\theta,\cos\theta,0),\quad e_3:=(0,0,1).\label{notation3}
\end{equation}
We denote by $B'_R$ (respectively $B_R$) the ball of radius $R$ in $\mathbb{R}^2$ (respectively $\mathbb{R}^3$) centered in the origin.\\

For any $b\in\R^3$ and $t\in S^2$ we define $\mu_{b,t}:=b\otimes t\mathcal{H}^1\res\mathbb{R}t$ and denote by $\eta_{b,t}\in L^1(\mathbb{R}^3;\mathbb{R}^{3\times3})$ the distributional solution to 
\begin{equation}
    \begin{cases}
    \Div\mathbb{C}\xi=0&\text{in $\mathbb{R}^3$}\\
    \curl\xi=\mu_{b,t}&\text{in $\mathbb{R}^3$}.
    \end{cases}\label{eul1}
\end{equation}
The function $\eta_{b,t}$ is of the form
\begin{equation}
     \eta_{b,t}(\Phi_t(r,\theta,z))=\frac{1}{r}(f(\theta)\otimes Q_te_\theta+g\otimes Q_te_r),\label{eul2}
\end{equation}
where $(f,g)\in L^2((0,2\pi);\mathbb{R}^3)\times\mathbb{R}^3$, with $\int_0^{2\pi} f(s)ds=b$, are solutions to the following minimum problem
\begin{equation}\label{sself}
\Psi_0(b,t):=\min\left\{\int_0^{2\pi}\frac12\mathbb{C}G(\theta):G(\theta)\ d\theta\right\},
\end{equation}
the minimum being taken over  all functions $G\colon(0,2\pi)\to\R^{3\times3}$ of the form $G(\theta)=f(\theta)\otimes Q_te_3+g\otimes Q_te_r$ as in \eqref{eul2} 
(see \cite[Lemma 5.1]{C.G.O.} ).
In particular
\begin{equation}
    \label{eul3}
    |\eta_{b,t}|(x)\le c\frac{|b|}{\dist(x,\R t)},
\end{equation}
for a constant $c>0$ depending only on $\mathbb C$. The line-tension energy density associated to the measure $\mu_{b,t}$ given  in \eqref{sself} can be rewritten as
\begin{equation}
    \Psi_0(b,t):=\int_0^{2\pi}\frac{1}{2}\mathbb{C}\eta_{b,t}(\Phi_t(1,\theta,0)):\eta_{b,t}(\Phi_t(1,\theta,0))d\theta.
    \label{selfen1}
\end{equation}

\begin{oss}
The function $\Psi_0$ is continuous and satisfies
\begin{itemize}
    \item there exist $c_0,c_1>0$ such that
    \begin{equation}
    c_0|b|^2\le\Psi_0(b,t)\le c_1|b|^2;\label{selfen2}
\end{equation}
\item for any $t\in S^2$ the map $b\mapsto\Psi_0(b,t)$ is quadratic;
\item there exists $c>0$ such that for any $t,t'\in S^2$
\begin{equation}
    \Psi_0(b,t)\le(1+c|t-t'|)\Psi_0(b,t').\label{selfen3}
\end{equation}
\end{itemize}
\label{remark1}
\end{oss}

This line-tension energy density for straight infinite dislocations is the starting point in order to characterise the asymptotics in \cite{C.G.O.}. Indeed it can be shown that the line-tension energy 
\begin{equation}
\int_\gamma  \Psi_0(b,t)d\mathcal H^1,    
\end{equation}
may be not lower semicontinuous, therefore the limiting energy in the linear framework requires a relaxation procedure. 
More precisely in \cite{C.G.M.} the authors show that the relaxed energy is given by 
\begin{equation*}
    \int_\gamma\tilde{\Psi}_0(b,t)d\mathcal{H}^1,
\end{equation*}
where $\tilde{\Psi}_0$ is the $\mathcal H^1$-elliptic envelope of $\Psi_0$ as defined in \eqref{relaxedenergy}. We stress that $\Psi_0(b,t)\ge\tilde\Psi_0(b,t)$ and, in particular, it has linear growth with respect to $b$, namely, there exist $\tilde{c}_0,\tilde{c}_1>0$ such that
\begin{equation}
    \tilde{c}_0|b|\le\tilde\Psi_0(b,t)\le\tilde{c}_1|b|.\label{selfen4}
\end{equation}

Next we characterise the elastic energy induced by $\mu_{b,t}$ in a finite cylinder.
To this aim fix $b\in\mathbb{R}^3$, $t\in S^2$, $h,r,R\in(0,\infty)$ with $r<R\le h$, and denote $T_h^r:=Q_t((B'_R\setminus B'_r)\times(0,h))$ with $Q_t$ defined in (\ref{notation1}).
Then consider the three-dimensional linear cell problem 
\begin{align}
    \Psi(b,t,h,r,R)&:=\frac{1}{h\log\frac{R}{r}}\min\biggl\{\int_{T_h^r}\frac{1}{2}\mathbb{C}\eta:\eta\ dx,\,\eta\in L^1_{loc}(\mathbb{R}^3;\mathbb{R}^{3\times3}),\curl\eta=\mu_{b,t}\biggr\}\notag\\ 
   &=\frac{1}{h\log\frac{R}{r}}\min\biggl\{\int_{T_h^r}\frac{1}{2}\mathbb{C}\eta:\eta\ dx,\,\eta\in L^2(T_h^r;\mathbb{R}^{3\times3}),\,\curl\eta=0\,\text{in $T_h^r$},\notag \\
       &\;\;\;\;\int_0^{2\pi}\eta(\Phi_t(\rho,\theta,z))Q_te_\theta\rho\ d\theta=b\ \text{for }\,(\rho,z)\in(r,R)\times(0,h)\biggr\}
       \label{linearcell}.
    \end{align}
 
\begin{oss} The condition
\begin{equation*}
   \int_0^{2\pi}\eta(\Phi_t(\rho,\theta,z))Q_te_\theta\rho\ d\theta=b\quad\text{for }\,(\rho,z)\in(r,R)\times(0,h),
\end{equation*}
is intended in the  following integral sense 
\begin{equation}\label{add3}
    \int_0^h\int_r^R\varphi(\rho,z) \int_0^{2\pi}\eta(\Phi_t(\rho,\theta,z))Q_te_\theta\rho\ d\theta d\rho dz=\int_0^h\int_r^R\varphi(\rho,z)b\ d\rho dz,
\end{equation}
for all $\varphi\in L^2((r,R)\times(0,h))$.
\end{oss}
The following asymptotic analysis is proved in \cite[Lemma 5.6, 5.10 and 5.11]{C.G.O.}
\begin{lem} There is a constant $c>0$ such that for every $M\ge1$, there is a function $\omega_M\colon(0,\infty)\to(0,\infty)$ with 
\begin{equation*}
    \lim_{r\to0}\omega_M(r)=0,
\end{equation*}
and
\begin{equation*}
    \left(1-\frac{c}{M}-\omega_M\Big(\frac{r}{R}\Big)\right)\Psi_0(b,t)\le\Psi(b,t,h,r,R)\le\Psi_0(b,t),
\end{equation*}
for all $b\in\mathbb{R}^3$, $t\in S^2$, $r,R,h>0$ such that $MR\le h$. In particular 
\begin{equation} \label{cgolem1}
  \lim_{h\to\infty}\lim_{r\to0}\Psi(b,t,h,r,R)=\Psi_0(b,t).
\end{equation}
Furthermore, there exists $c_*>0$ such that, for all $b,t,h,r,R$ with $2r\le R\le h$,
\begin{equation}
    c_*\left|b\right|^2\le\Psi(b,t,h,r,R).
\end{equation}\label{cgolem}
\end{lem}

\begin{lem}\label{estimatefromabove}
Let $h,R>0$ with $R\le h$, $t\in S^2$, $b\in\R^3$, $T:=Q_t(B'_R\times(0,h))$. Let $\eta\in L^1(T;\R^{3\times3})$ be such that 
\begin{equation}\label{estimatefromabove1}
    |\eta|(x)\le\frac{c^*|b|}{\dist(x,\R t)}\quad\text{for all}\quad x\in T,
\end{equation}
for some constant $c^*>0$, and
\begin{equation}\label{estimatefromabove2}
    \curl\eta=b\otimes t\mathcal{H}^1\res\R t\quad\text{in $T$}.
\end{equation}
Then, there exists $\tilde\eta_r\in L^1(T;\R^{3\times3})$ such that $\curl\tilde\eta_r=\curl\eta$ in $T$, $\tilde\eta_r=\eta$ in a neighbourhood of $\partial T$ and, for all $r\in(0,R/3)$,
\begin{equation*}
    \int_{T_h^r}\frac12{\mathbb C}\tilde\eta_r:\tilde\eta_r dx\le \Psi_0(b,t)h\log\frac Rr+C_0(1+c^*)^2|b|^2\left(h\left(\log\frac Rr\right)^{1/2}+\frac{h^3}{R^2}\right),
\end{equation*}
where $T_h^r:=Q_t((B'_R\setminus B'_r)\times(0,h))$ and $C_0$ is an universal constant (thus independent of $c^*$). 
\end{lem}
\begin{lem}\label{estimatefromabovemoll} 
	Let $h,R,r>0$, with $3r\le R\le h$, $t\in S^2$, $b\in\R^3$, $T:=Q_t(B'_R\times(0,h))$. Let $\eta\in L^1(T;\R^{3\times3})$ be such that 
	\begin{equation}\label{estimatefromabove1moll}
	|\eta|(x)\le\frac{c^*|b|}{r+\dist(x,\R t)}\quad\text{for all}\quad x\in T,
	\end{equation}
	for some constant $c^*>0$, and
	\begin{equation}\label{estimatefromabove2moll}
	\curl\eta=(b\otimes t\mathcal{H}^1\res\R t)*\varphi_r\quad\text{in $T$},
	\end{equation}
	where $\varphi_r(x)=r^{-3}\varphi(x/r)$ is a mollifier.
	Then, there exists $\hat\eta_r\in L^1(T;\R^{3\times3})$ such that $\curl\hat\eta_r=\curl\eta$ in $T$, $\hat\eta_r=\eta$ in a neighbourhood of $\partial T$ and, for all $r\in(0,R/3)$,
	\begin{equation*}
	\int_{T}\frac12{\mathbb C}\hat\eta_r:\hat\eta_r dx\le \Psi_0(b,t)h\log\frac Rr+C_0(1+c^*)^2|b|^2\left(h\left(\log\frac Rr\right)^{1/2}+\frac{h^3}{R^2}\right),
	\end{equation*}
	with $C_0$ an universal constant.
\end{lem}
We have stated Lemma \ref{estimatefromabove} and Lemma \ref{estimatefromabovemoll} by explicitly keeping track of the dependence on the constant $c^*$ in the energy estimates, which can be easily obtained by inspecting the proof of \cite[Lemma 5.10 and 5.11]{C.G.O.}.
\subsection{Rigidity} In this section we state and prove a rigidity estimate (see \cite{F.J.M.02}) for a hollow cylinder with a constant which does not depend on the radius of the central hole. This result combines the two-dimensional nonlinear version proved in \cite{s.z.} with the three-dimensional Korn's inequality in \cite[Lemma 5.9]{C.G.O.}.
\begin{lem}[Rigidity with a hole] For every $h>0$ there exists $c=c(h)>0$ with the following property: let $R>0$, $\varepsilon\in(0,\frac{\min\{h,1\}}{2}R]$, $T_\varepsilon:=(B'_R\setminus B'_\varepsilon)\times(0,hR)$ and $u\in W^{1,2}(T_\varepsilon; \mathbb{R}^3)$, then there exists a rotation $Q\in SO(3)$  such that
\begin{equation*}
    \left\|\nabla u-Q\right\|_{L^2(T_\varepsilon)}\le c\left\|\dist(\nabla u,SO(3))\right\|_{L^2(T_\varepsilon)}. 
    \end{equation*}\label{rigidity}
\end{lem}
\begin{proof} In order to obtain the estimate we need to extend the function $u$ in the inner cylinder. This should be done in two steps.
	By scaling we can take $R=1$. We set $N:=\lfloor h/\varepsilon\rfloor-1$ and $z_j:=j\varepsilon$ for $j=0,...,N-1$, $z_N:=h-2\varepsilon$, so that $z_N-z_{N-1}\in[0,\varepsilon)$. For every $j$, we now apply the rigidity estimate \cite[Theorem 3.1]{F.J.M.02} on the domain $F_j:=(B'_{2\varepsilon}\setminus B'_\varepsilon)\times(z_j,z_j+2\varepsilon)$ we get
	\begin{equation*}
	\left\|\nabla u-Q_j\right\|_{L^2(F_j)}\le c\left\|\dist(\nabla u,SO(3))\right\|_{L^2(F_j)},
	\end{equation*}
	for a rotation $ Q_j\in SO(3)$, with a constant that, by scaling, does not depend on $\varepsilon$. 
	Moreover applying Poincare's inequality we find $q_j\in\R^3$ such that 
	\begin{equation*}
	\frac{1}{\varepsilon}\|u-Q_jx-q_j\|_{L^2( F_j)}\le c\|\dist(\nabla u,SO(3))\|_{L^2(F_j)}.
	\end{equation*}
	Therefore by applying the triangular inequality we derive for $k=j-1$ and $k=j+1$ 
	\begin{equation}\label{Roby}
	\frac{1}{\varepsilon}\left\|Q_kx+q_k-Q_jx-q_j\right\|_{L^2(F_j\cap F_k)}\le c\left\|\dist(\nabla u,SO(3))\right\|_{L^2(F_j\cup F_k)}.
	\end{equation}
	Now, using Lemma 2.6 in \cite{extension}, for any $j$ we can extend $u$ to a function $u_j$ on $\hat F_j:=B'_{2\varepsilon}\times (z_j,z_j+2\varepsilon)$ satisfying
	\begin{equation}
	\frac{1}{\varepsilon}\|u_j-Q_jx-q_j\|_{L^2(\hat F_j)}+  \left\|\nabla u_j-Q_j\right\|_{L^2(\hat{F}_j)}\le c\left\|\dist(\nabla u,SO(3))\right\|_{L^2(F_j)}.\label{extension*}
	\end{equation} 
	Again, by scaling, the constant does not depend on $\varepsilon$. 
	 Finally we interpolate the different extensions by choosing $\varphi_j\in C_c^\infty(\R)$ such that
	$$\sum_j\varphi_j=1\quad\text{on $(0,h)$},\quad\varphi_j=0\quad\text{on $(0,h)\setminus(z_j,z_j+2\varepsilon)$},\quad |\nabla\varphi_j|\le\frac{c}{\varepsilon}.$$
	In particular $\varphi_0(0)=\varphi_N(h)=1$.
	We define $\tilde u:=\sum_j\varphi_ju_j$ in $B'_\varepsilon\times(0,h)$ and write 
	\begin{equation}\label{extension***}
	\nabla\tilde u=\sum_j\varphi_j\nabla u_j+ \sum_j(u_j-Q_jx-q_j)\otimes\nabla\varphi_j+\sum_j(Q_jx+q_j)\otimes\nabla\varphi_j.
	\end{equation}
	Observe that, since $\sum_j\nabla\varphi_j=0$, for every $k$ we have
	\begin{equation}\label{pieceofext}
	\sum_j(Q_jx+q_j)\otimes\nabla\varphi_j=\sum_j(Q_jx+q_j-Q_kx-q_k)\otimes\nabla\varphi_j.
	\end{equation}
We can now write, by triangular inequality
	\begin{align}
	\left\|\dist(\nabla \tilde u,SO(3))\right\|_{L^2(B'_\varepsilon\times(0,h))}&\le\sum_j\|\dist(\nabla \tilde u,SO(3))\|_{L^2(\hat F_j)}\notag\\
	&\le \sum_j\|\nabla \tilde u-Q_j\|_{L^2(\hat F_j)}.\label{triangineq}
	\end{align}
	In turn, by \eqref{extension***}, \eqref{pieceofext} and recalling that $\varphi_k=0$ on $(z_j,z_j+2\eps)$ for $k\notin\{j-1,j,j+1\}$ we have
	\begin{align}
	&	\|\nabla\tilde u-Q_j\|_{L^2(\hat F_j)}\le \|\sum_{k=j-1}^{j+1}\varphi_k(\nabla u_k-Q_j)\|_{L^2(\hat F_j)}\notag
	\\&	\,+ \|\sum_{k=j-1}^{j+1}(u_k-Q_kx-q_k)\otimes\nabla\varphi_k\|_{L^2(\hat F_j)}	+\|\sum_{k=j-1}^{j+1}(Q_jx+q_j-Q_kx+q_k)\otimes\nabla\varphi_k\|_{L^2(\hat F_j)} \notag\\
	&	\le \sum_{k=j-1}^{j+1}\|\varphi_k(\nabla u_k-Q_k)\|_{L^2(\hat F_j)}+ \sum_{k=j-1}^{j+1}\|\varphi_k(Q_k-Q_j)\|_{L^2(\hat F_j)}\notag\\
		&\,+\sum_{k=j-1}^{j+1}\|(u_k-Q_kx-q_k)\otimes\nabla\varphi_k\|_{L^2(\hat F_j)}+ \sum_{k=j-1}^{j+1}\|(Q_jx+q_j-Q_kx+q_k)\otimes\nabla\varphi_k\|_{L^2(\hat F_j)}.\label{triangg}
	\end{align}
Using \eqref{extension*} in the first term on the right hand side of \eqref{triangg} we get
	\begin{align*}
		 \sum_{k=j-1}^{j+1}\|\varphi_k(\nabla u_k-Q_k)\|_{L^2(\hat F_j)}& \le \sum_{k=j-1}^{j+1}\|(\nabla u_k-Q_k)\|_{L^2(\hat F_j\cap\hat F_k)}\\
		 &\le c\sum_{k=j-1}^{j+1} \|\dist(\nabla u, SO(3))\|_{L^2(F_k)}.
	\end{align*}
	Similarly we estimate the third term 
		\begin{align*}
\sum_{k=j-1}^{j+1}\|(u_k-Q_kx-q_k)\otimes\nabla\varphi_k\|_{L^2(\hat F_j)}&\le \frac{c}{\eps}\sum_{k=j-1}^{j+1}\|(u_k-Q_kx-q_k)\|_{L^2(\hat F_j\cap\hat F_k)}
	\\
	&\le c\sum_{k=j-1}^{j+1} \|\dist(\nabla u, SO(3))\|_{L^2(F_k)}.
	\end{align*}
	The fourth term is instead bounded by
	\begin{align*}
	&	\sum_{k=j-1}^{j+1}\|(Q_jx+q_j-Q_kx+q_k)\otimes\nabla\varphi_k\|_{L^2(\hat F_j)}\\&\,\,\le \frac c\eps \|Q_jx+q_j-Q_{j+1}x+q_{j+1}\|_{L^2(\hat F_j\cap\hat F_{j+1})}	+ \frac c\eps \|Q_jx+q_j-Q_{j-1}x+q_{j-1}\|_{L^2(\hat F_j\cap\hat F_{j-1})}\\
		&\,\,\le c\|\dist(\nabla u,SO(3))\|_{L^2(F_j\cup F_{j+1})}+ c\|\dist(\nabla u,SO(3))\|_{L^2(F_j\cup F_{j-1})},
	\end{align*}
	where the second inequality follows by \eqref{Roby} and by the fact that
		\begin{equation}\label{extension**}
		\frac{1}{\varepsilon}\|Q_jx+q_j-Q_kx-q_k\|_{L^2(\hat F_j\cap \hat F_k)}\le 	\frac{c}{\varepsilon}\|Q_jx+q_j-Q_kx-q_k\|_{L^2(F_j\cap F_k)}
		\end{equation}
		for $k=j-1,j+1$. It remains to estimate the second term on the right hand side of \eqref{triangg}, to this end we observe that 
	\begin{align*}
		\sum_{k=j-1}^{j+1}\|\varphi_k(Q_k-Q_j)\|_{L^2(\hat F_j)}&\le 
		|Q_{j+1}-Q_j||\hat F_j|^{1/2}+	|Q_{j-1}-Q_j||\hat F_j|^{1/2}\\
		&\le c\eps^{3/2}|Q_{j+1}-Q_j|+c\eps^{3/2}|Q_{j-1}-Q_j|\\
		&\le \frac c\eps \|Q_jx+q_j-Q_{j+1}x-q_{j+1}\|_{L^2(\hat F_j\cap \hat F_{j+1})}\\
		&\,\,\,\,\,+\frac c\eps \|Q_jx+q_j-Q_{j-1}x-q_{j-1}\|_{L^2(\hat F_j\cap \hat F_{j+1})}\\
		&\le c\|\dist(\nabla u,SO(3))\|_{L^2(F_j\cup F_{j+1})}\\&\,\,\,\,\,+ c\|\dist(\nabla u,SO(3))\|_{L^2(F_j\cup F_{j-1})}.
	\end{align*}
	Going back to \eqref{triangineq} we have that 
	\begin{align*}
			\left\|\dist(\nabla \tilde u,SO(3))\right\|_{L^2(B'_\varepsilon\times(0,h))} &\le c\sum_j\|\dist(\nabla u, SO(3))\|_{L^2(F_j)}\\
			&\le  c\|\dist(\nabla u,SO(3))\|_{L^2((B'_1\setminus B'_\varepsilon)\times(0,h))}.
	\end{align*}
	Therefore we infer
	\begin{align*}
&	\|\dist(\nabla\tilde u,SO(3))\|_{L^2(B'_1\times(0,h))}\le \|\dist(\nabla u,SO(3))\|_{L^2((B'_1\setminus B'_\varepsilon)\times(0,h))}\\
&\,\,\,	+ \|\dist(\nabla \tilde u,SO(3))\|_{L^2(B'_\varepsilon\times(0,h))}\le c\|\dist(\nabla u,SO(3))\|_{L^2(B'_1\setminus B'_\varepsilon\times(0,h))},\label{rigidity2*}
	\end{align*}
	for a constant $c$ independent of $\varepsilon$. We can conclude the proof by applying the rigidity estimate on the fixed domain $B'_1\times(0,h)$ (see \cite[Theorem 6]{F.J.M.}  and \cite{A.B.P.}).
\end{proof}
\begin{oss}\label{circularsector}
	The estimate in Lemma \ref{rigidity} still holds if we replace $T_\varepsilon$ with a set $S_\varepsilon(\vartheta)\times(0,hR)$ where $S_\varepsilon(\vartheta):=\{(r,\theta):\,\varepsilon<r<R,\,0<\theta<\vartheta\}$ for $0<\vartheta<2\pi$ with a constant depending on $\vartheta$.
\end{oss}
Lemma \ref{rigidity} cannot be directly applied to strains $\beta$ whose curl is concentrated on a line, as they are not globally gradients in the hollow cylinder. They become gradients if we cut the domain to let it be simply connected. Therefore as in \cite{s.z.} we need a variant of Lemma \ref{rigidity} for a domain with a hole and a cut.\\
We fix the segment $L_\varepsilon:=\{(x,0):\,\varepsilon<x<R\}\subset\mathbb{R}^2$ and denote
\begin{equation}
    \tilde{T}_\varepsilon:=[(B'_R\setminus B'_\varepsilon)\setminus L_\varepsilon]\times(0,hR).\label{cutcylinder}
\end{equation}
\begin{cor}[Rigidity with a "hole" and a "cut"]
For every $h,R>0$, $\varepsilon\in(0,\frac{\min\{h,1\}}{2}R]$ let $\tilde{T}_\varepsilon$ be as in (\ref{cutcylinder}). Then there exists a constant $c=c(h)>0$ such that for every $u\in W^{1,2}(\tilde{T}_\varepsilon;\R^3)$ there exists a rotation $Q\in SO(3)$ that satisfies
\begin{equation}
    \left\|\nabla u-Q\right\|_{L^2(\tilde{T}_\varepsilon)}\le c\left\|\dist(\nabla u,SO(3))\right\|_{L^2(\tilde{T}_\varepsilon)}. \label{cut_rigidity}
    \end{equation}\label{rigiditybis}
\end{cor}
\begin{proof}
The proof is analogous to the two dimensional case (see \cite[Proposition 3.3]{s.z.}), and consists in partitioning the domain with four cylindrical subdomains of the form $S_\varepsilon(\pi/2)\times(0,hR)$ (as in Remark \ref{circularsector}) and consequently applying the rigidity estimate to each of them and find four constant rotations. The thesis follows by showing that for these rotations the estimate also holds in the whole domain.   
\end{proof}
\begin{cor}
\label{rigiditytris}
For every $h>0$ there exists $c=c(h)>0$ with the following property: let $R>0$, $\varepsilon\in(0,\frac{\min\{h,1\}}{2}R]$ and $T_\varepsilon$ as in Lemma \ref{rigidity} and let $b\in\R^3$ be fixed; then for every $\beta\in L^2(T_\varepsilon;\mathbb{R}^{3\times3})$ with $\curl\beta=0$ in $T_\varepsilon$ and $\int_0^{2\pi}\beta(\rho,\theta,z)e_\theta \rho \ d\theta=\varepsilon b$ for every $\rho\in(\varepsilon,R)$, $z\in(0,h)$, there exists a rotation $Q\in SO(3)$ such that
\begin{equation*}
    \|\beta-Q\|_{L^2({T}_\varepsilon)}\le c\| \dist(\beta,SO(3))\|_{L^2({T}_\varepsilon)}.
\end{equation*}
\end{cor}
\begin{proof} Let $\tilde T_\varepsilon$ be the cut cylinder as defined in \eqref{cutcylinder} and let $\beta$ be as in the statement. Then $\curl\beta=0$ in the simply connected domain $\tilde T_\varepsilon$, therefore there exists $u\in W^{1,2}(\tilde T_\varepsilon;\R^3)$ such that $\beta=\nabla u$ in $\tilde T_\varepsilon$. We conclude by applying Corollary \ref{rigiditybis} to the function $u$. 
\end{proof}

Analogously to the 2-dimensional case derived in \cite{s.z.}, Corollary \ref{rigiditytris} points out that any elastic strain $\beta$ whose curl is concentrated on the vertical axis of a cylinder with multiplicity $\varepsilon b$ and whose energy tends to zero, is close to a constant rotation $Q\in SO(3)$ in $L^2(T_\varepsilon;\R^{3\times 3})$. This suggests that in the limit as $\varepsilon\to0$ the energy density $W(\beta)$ linearises near such rotation, or, equivalently by frame indifference, $W(Q^T\beta)$ linearises near the identity $I$. This will imply that the nonlinear energy of a straight dislocation $\mu_{b,t}$ on a hollow cylinder is asymptotically comparable (as the inner radius goes to zero) to the self-energy \eqref{selfen1} where the normalised Burgers vector is rotated by $Q^T$.\\
In order to deal with this additional variable that appears in the asymptotic it is convenient to define an auxiliary cell problem formula which describes the nonlinear elastic energy of a straight dislocation in a cylinder with the constraint that the elastic strain is close to a fixed rotation. 

\subsection{Nonlinear Auxiliary Cell Problem Formula}\label{aux-cell}

For every $Q\in SO(3)$, $b\in\mathbb{R}^3$, $t\in S^2$, $R>r>0$, $h>0$ and $\lambda>0$ we consider the variational problem

\begin{align}
    &\Psi_\lambda^{nl}(Q,b,t,h,r,R):=\frac{1}{hr^2\log\frac{R}{r}}\inf\biggl\{ \int_{T_h^r}W(\beta)+\lambda|\beta-Q|^2dx:\beta\in L^2(T_h^r;\mathbb{R}^{3\times3}),\notag\\&\,\,\curl\beta=0\,\text{in $T_h^r$},\,
     \int_0^{2\pi}\beta(\Phi_t(\rho,\theta,z))Q_te_\theta\rho \ d\theta=rb\,\text{for all}\,\rho\in(r,R),\,z\in(0,h)
    \biggr\},
    \label{nonlinearcell}
\end{align}
where we recall $T_h^r:=Q_t((B'_R\setminus B'_r)\times(0,h))$.
The dependence on $Q\in SO(3)$ introduced in this auxiliary cell problem formula is needed since, a priori, due to the frame indifference of the energy, the latter can linearise close to any rotation giving rise to different optimal energies. This dependence which helps to control the convergence will not appear in the $\Gamma$-limit and will be removed by taking $\lambda\to0$.
\begin{oss} By scaling, it holds
\begin{equation}\label{scalingg}
    \Psi_\lambda^{nl}(Q,b,t,k h,k r,k R)=\Psi_\lambda^{nl}(Q,b,t,h,r,R)\quad\text{for all $k>0$}.
\end{equation}
Indeed, for every $\beta$ admissible competitor for $\Psi_\lambda^{nl}(Q,b,t,h,r,R)$, the function $\tilde{\beta}(x):=\beta(x/k)$ is admissible for $\Psi_\lambda^{nl}(Q,b,t,k h,k r,k R)$, since $\int_0^{2\pi}\tilde{\beta}(\Phi_t(\rho,\theta,z))Q_te_\theta\ d\theta=krb$.\\
Furthermore, by the frame indifference of $W$, we have
\begin{equation}
    \Psi_\lambda^{nl}(Q,b,t,h,r,R)= \Psi_\lambda^{nl}(I,Q^Tb,t,h,r,R)\quad \forall Q\in SO(3).\label{frameindif}
\end{equation}
\end{oss} 
The main result of this section regards the asymptotic analysis of the cell problem, more precisely, we show that $\Psi_\lambda^{nl}(Q,b,t,h,r,R)$ converges to $\Psi_0(Q^Tb,t)$ as $r\to0$, $h\to\infty$, $\lambda\to0$. 
\begin{lem}[Lower Bound]
\label{lowerbound}
There exists a constant $C>0$ such that, for every $Q\in SO(3)$, $b\in\mathbb{R}^3$, $t\in S^2$, $h\ge R>0$, and $\lambda>0$, it holds
\begin{equation*}
    \liminf_{r\to0}\Psi_\lambda^{nl}(Q,b,t,h,r,R)\ge\Psi_0(Q^Tb,t)-C|b|^2\frac{R}{h}.
\end{equation*}
\end{lem}
\begin{proof}
Fix $Q$, $b$ and $t$.
To simplify the notation we assume $t=e_3$ (the general case being similar) so that $T_h^r=(B'_R\setminus B'_r)\times(0,h)$.\\
For every $0<r<R$, let $\beta_r$ be admissible for the definition of $\Psi_\lambda^{nl}$ (see \eqref{nonlinearcell}), and be such that
\begin{equation}
    \frac{1}{hr^2\log\frac{R}{r}}\int_{T_h^r}W(\beta_r)+\lambda|\beta_r-Q|^2dx
    \le\Psi_\lambda^{nl}(Q,b,t,h,r,R)+|b|^2\frac{R}{h}.\label{stima1}
\end{equation}
Next fix $\delta\in(0,1/2)$ and divide $T_h^r$ into dyadic cylinders
 \begin{equation*}
    C_i^k:=(B'_{R\delta^{k-1}}\setminus B'_{R\delta^k})\times((i-1)h\delta^{k-1},ih\delta^{k-1}),
\end{equation*}
with $k=1,...,\tilde{k}_r$ where $\tilde{k}_r:=\lfloor k_r\rfloor+1$,
\begin{equation}
    k_r:=s\frac{\log \frac{R}{r}}{|\log\delta|} \quad\text{for some fixed $s\in(0,1)$,}
    \label{kerre}
\end{equation}
and $i=1,...,i_k:=\lfloor1/\delta^{k-1}\rfloor$. 
Since $\delta$, $R$, and $s$, are fixed, the smallest inner radius of the dyadic annuli, namely $R\delta^{\tilde{k}_r}$, is much bigger than $r$ as $r\to0$, indeed 
\begin{equation}\label{ordinedigrandezza}
    R\delta^{\tilde{k}_r}\ge R\delta^{k_r+1}=\delta R^{1-s}r^s\gg r.
\end{equation}
In particular
\begin{equation}
  \frac{1}{hr^2\log\frac{R}{r}}\int_{T_h^r}W(\beta_r)+\lambda|\beta_r-Q|^2dx\ge\frac{1}{\log\frac{R}{r}}\sum_{k=1}^{\tilde{k}_r}\sum_{i=1}^{i_k}\frac{1}{h}\int_{C_i^k}\frac{W(\beta_r)+\lambda|\beta_r-Q|^2}{r^2}dx.  \label{dim1}
\end{equation}
We stress that the cylinders $C_i^k$ (and the corresponding indices $i$ and $k$) depend on $r$. We will denote by $I(r)$ the set of such pairs of indices, i.e.,
\begin{equation}
    I(r)=\{(i,k)\,:\,i=1,...,i_k,\,k=1,...,\tilde{k}_r\}.\label{indici}
\end{equation}
For every $0<r<R$ and $k=1,...,\tilde{k}_r$ we denote by $i(k,r)$ an index $i$ such that
\begin{equation}
    \label{dim1bis}
    \sum_{i=1}^{i_k}\int_{C_i^k}W(\beta_r)+\lambda|Q -\beta_r|^2dx\ge i_k
    \int_{C^k_{i(k,r)}}W(\beta_r)+\lambda|Q-\beta_r|^2dx.
\end{equation}
Let $M:= h/R$ and let $\omega_M\colon(0,\infty)\to(0,\infty)$ be the function given in the statement of Lemma \ref{cgolem}. 

We will show that there exist a sequence of positive numbers $\sigma_r$, infinitesimal for $r\to0$, and a positive constant $C>0$, such that
\begin{equation}\begin{split}
    \frac{1}{h\delta^{k-1}}\int_{C^k_{i(k,r)}}&\frac{W(\beta_r)+\lambda|\beta_r-Q|^2}{r^2}dx\\
    &\ge\left|\log\delta\right|\Psi_0(Q^Tb,t)\Big(1-\frac{C}{M}-\omega_M(\delta)\Big)-\sigma_r,
    \end{split}
     \label{claim}
\end{equation}
for every $r>0$ and for every $k=1,...,\tilde{k}_r$. Equivalently we will show that
\begin{align}
    \label{claimbis}
    \lim_{r\to0}\bigg(\left|\log\delta\right|&\Psi_0(Q^Tb,t)\Big(1-\frac{C}{M}-\omega_M(\delta)\Big)\nonumber\\
    &-\frac{1}{h\delta^{k-1}}\int_{C_{i(k,r)}^k}\frac{W(\beta_r)+\lambda|\beta_r-Q|^2}{r^2}dx\bigg)_+=0,
\end{align}
uniformly in $k$, where $(a)_+$ denotes the positive part of $a$. We establish \eqref{claimbis} arguing by contradiction. The argument is based on the following claim which is a variation of the one used in \cite{s.z.} for the proof of the lower bound.\\
\noindent {\it Claim:
Given a sequence $r_j\to 0$, assume that there exists a constant $\tilde C>0$ such that
\begin{equation}\label{ipotesiclaim}
    \frac{1}{h\delta^{{k_j}-1}}\int_{C^{k_j}_{i_j}}\frac{W(\beta_j)+\lambda|\beta_j-Q|^2}{r_j^2}dx< \tilde C<+\infty,
\end{equation}
for some sequence $(i_j,k_j)\in I(r_j)$ as defined in \eqref{indici},  $\beta_j\in L^2(C_{i_j}^{k_j};\mathbb{R}^{3\times3})$ with 
$\curl\beta_j=0$ in $C_{i_j}^{k_j}$, and 
\begin{equation}
    \int_0^{2\pi}\beta_j(\rho,\theta,z)e_\theta\rho\ d\theta=r_jb,\label{ipotesiclaim2}
\end{equation}
for $(\rho,z)\in(R\delta^{k_j},R\delta^{k_j-1})\times((i-1)h\delta^{k_j-1},ih\delta^{k_j-1})$. Then it holds
\begin{multline}
  \liminf_{j\to +\infty}  \frac{1}{h\delta^{{k_j}-1}}\int_{C^{k_j}_{i_j}}\frac{W(\beta_j)+\lambda|\beta_j-Q|^2}{r_j^2}dx\\\geq \left|\log\delta\right|\Psi_0(Q^Tb,t)\Big(1-\frac{C}{M}-\omega_M(\delta)\Big).\label{tesiclaim}
\end{multline}
}
\bigskip

Before proving the Claim, we show that it implies \eqref{claimbis} and how this concludes the proof of the Lemma. If, by contradiction, \eqref{claimbis} does not hold, then there exists an $\varepsilon_0>0$ such that for all $\sigma>0$ there exists $r_\sigma\le\sigma$ and $k_\sigma$ such that 
\begin{align}\nonumber
    \bigg(\left|\log\delta\right|\Psi_0(Q^Tb,t)&\Big(1-\frac{C}{M}-\omega_M(\delta)\Big)\\
    &-\frac{1}{h\delta^{k_\sigma-1}}\int_{C_{i(k_\sigma,r_\sigma)}^{k_\sigma}}\frac{W(\beta_{r_\sigma})+\lambda|\beta_{r_\sigma}-Q|^2}{r_\sigma^2}dx\bigg)_+>\varepsilon_0,
    \label{contraddizione}
\end{align}
and therefore we can assume that there exist a sequence $r_j\to0$ and $k_j$ such that for the corresponding $i_j:=i(k_j,r_j)$, \eqref{ipotesiclaim} holds with
\begin{equation*}
    \tilde{C}=\left|\log\delta\right|\Psi_0(Q^Tb,t)\Big(1-\frac{C}{M}-\omega_M(\delta)\Big)-\varepsilon_0,
\end{equation*}
and $\beta_j:=\beta_{r_j}$.
Namely 
\begin{equation}
    \frac{1}{h\delta^{k-1}}\int_{C^{k_j}_{i_j}}\frac{W(\beta_j)+\lambda|\beta_j-Q|^2}{r_j^2}dx\le|\log\delta|\Psi_0(Q^Tb,t)\Big(1-\frac{C}{M}-\omega_M(\delta)\Big)-\varepsilon_0.\label{contr2}
\end{equation}
Combining this with \eqref{tesiclaim} we get the contradiction and we then prove \eqref{claim}.

Therefore by (\ref{dim1}), \eqref{dim1bis} and (\ref{claim}) we have
\begin{align*}
\liminf_{r\to0}\frac{1}{hr^2\log\frac{R}{r}}
\int_{T_h^r}&W(\beta_r)+\lambda|\beta_r-Q|^2dx\\
&\ge\liminf_{r\to0}\frac{1}{\log\frac{R}{r}}\sum_{k=1}^{\tilde{k}_r}\sum_{i=1}^{i_k}\frac{1}{h}\int_{C_i^k}\frac{W(\beta_r)+\lambda|\beta_r-Q|^2}{r^2}dx\\
&\ge \liminf_{r\to0}\frac{(1-\delta)}{\log\frac{R}{r}}\sum_{k=1}^{\tilde{k}_r}\frac{1}{h\delta^{k-1}}\int_{C^k_{i(k,r)}}\frac{W(\beta_r)+\lambda|\beta_r-Q|^2}{r^2}dx\\
&\ge \liminf_{r\to0} s(1-\delta)\Psi_0(Q^Tb,t)(1-\frac{C}{M}-\omega_M(\delta))-\sigma_r\\
&=s(1-\delta)\Psi_0(Q^Tb,t)(1-\frac{C}{M}-\omega_M(\delta)),
\end{align*}
where the quantity $(1-\delta)$ comes out from the fact that $\delta^{k-1}\cdot i_k=\delta^{k-1}\cdot\lfloor1/\delta^{k-1}\rfloor\ge(1-\delta)$.
Since $\delta\in(0,1/2)$ and $s\in(0,1)$ are arbitrary we can pass to the limit as $\delta\to0$ and $s\to1$ to get
\begin{equation*}
\liminf_{r\to0}\Psi_\lambda^{nl}(Q,b,t,h,r,R)\ge\Psi_0(Q^Tb,t)\Big(1-C\frac{R}{h}\Big)\ge \Psi_0(Q^Tb,t)-C|b|^2\frac{R}{h},
\end{equation*}
where the last inequality follows by (\ref{selfen2}) and the statement is proved.

It remains to prove the Claim. To simplify the notation we drop the dependence of $i$ and $k$ on $j$. We proceed in steps. First we scale the problem to a fixed cylinder (Step 1), then we estimate separately the quadratic (Step 2) and the nonlinear (Step 3) terms. \\

\textit{Step 1: Scaling and Compactness.} 
Here we reduce the analysis to a fixed cylinder which, with a little abuse of notation, we denote by
\[T_{h/R}^\delta:=(B'_1\setminus B'_\delta)\times(0,h/R)\]
and we consider the scaled function
\[\tilde{\eta}_j(x):=R\delta^{k-1}\eta_j\bigl(R\delta^{k-1}(x+(i-1)\frac{h}{R}e_3)\bigr)\quad \text{for $x\in T^\delta_{h/R}$},\]
where
\begin{equation}
\eta_j:=\frac{Q^T\beta_j-I}{r_j}.
\label{etaj}
\end{equation}

In this Step we show that up to a subsequence
\begin{equation}\label{add1}
\tilde{\eta}_j\rightharpoonup \tilde{\eta}\quad\text{in $L^2(T_{h/R}^\delta;\mathbb{R}^{3\times3})$,}
\end{equation}
and that $\tilde\eta$ is admissible for the cell problem defining $\Psi(Q^Tb,t,h/R,\delta,1)$ in \eqref{linearcell}.

Indeed from \eqref{ipotesiclaim} we have
\begin{equation}
    \frac{\lambda}{h\delta^{k-1}}\int_{C^k_i}\frac{|\beta_j-Q|^2}{r_j^2}dx< \tilde{C},\label{claim2}
\end{equation}
and then by a change of variable we have
\begin{equation*}
\int_{T^\delta_{h/R}}|\tilde{\eta}_j|^2dx=\frac{1}{R\delta^{k-1}}\int_{C_i^k}|\eta_j|^2dx\le\frac{h}{R\lambda}\tilde{C},
\end{equation*}
which implies the weak convergence in \eqref{add1} up to a subsequence. To verify the admissibility of the limiting function we recall that by \eqref{ipotesiclaim2} it holds
\begin{equation*}
\int_0^h\int_r^R\varphi(\rho,z)  \int_0^{2\pi}\eta_j(\rho,\theta,z)e_\theta\rho\ d\theta d\rho dz=\int_0^h\int_r^R\varphi(\rho,z)Q^Tb\ d\rho dz,
\end{equation*}
for all $\varphi\in L^2((R\delta^{k},R\delta^{k-1})\times((i-1)h\delta^{k-1},ih\delta^{k-1}))$, and therefore
\begin{equation}\label{add2}
\int_0^{h/R}\int_\delta^1\varphi(\rho,z)  \int_0^{2\pi}\tilde\eta_j(\rho,\theta,z)e_\theta\rho\ d\theta d\rho dz=\int_0^h\int_r^R\varphi(\rho,z)Q^Tb\ d\rho dz,
\end{equation}
for all $\varphi\in L^2((\delta,1)\times(0,h/R))$. The conclusion follows by taking the limit as $j\to\infty$, obtaining 
\begin{equation}\label{claim3}
\int_0^{h/R}\int_\delta^1\varphi(\rho,z)  \int_0^{2\pi}\tilde\eta(\rho,\theta,z)e_\theta\rho\ d\theta d\rho dz=\int_0^h\int_r^R\varphi(\rho,z)Q^Tb\ d\rho dz,
\end{equation}
for all $\varphi\in L^2((\delta,1)\times(0,h/R))$, and recalling  \eqref{add3}.
\bigskip

\textit{Step 2: Estimate of the quadratic term.} 
\\
By the lower semicontinuity of the $L^2$ norm it follows
\begin{equation}
    \label{aggiunta1}
  \liminf_{j\to\infty}  \int_{{T}^\delta_{h/R}}|\tilde{\eta}_j|^2dx\ge \int_{{T}^\delta_{h/R}}|\tilde{\eta}|^2dx.
\end{equation}
Using Jensen's inequality, \eqref{claim3}, and \eqref{selfen2}, we get the following estimate
\begin{equation*}
    \int_{{T}^\delta_{ h/R}}|\tilde\eta|^2dx\ge\int_0^{\frac hR}\int_\delta^1\frac{1}{2\pi\rho}
    \left|\int_0^{2\pi}\tilde\eta(\rho,\theta,z) e_\theta\rho\ d\theta \right|^2d\rho dz
\end{equation*} 
\begin{equation}
    \ge\frac{|Q^Tb|^2}{2\pi} \frac{h}{R}|\log\delta|\ge \Psi_0(Q^Tb,t)\frac{|\log\delta|}{2\pi c_1}\frac{h}{R}.\label{aggiunta}
\end{equation}
Therefore we obtain the following estimate of the quadratic part of the energy
\begin{align}
   \liminf_{j\to\infty} \frac{\lambda}{h\delta^{k-1}}\int_{C^k_i}\frac{|\beta_j-Q|^2}{r_j^2}dx&= \liminf_{j\to\infty}\lambda\frac{R}{h}\int_{{T}^\delta_{h/R}}|\tilde\eta_j|^2dx\notag\\
&\ge \lambda\frac{R}{h}\int_{{T}^\delta_{h/R}}|\tilde{\eta}|^2dx\ge |\log\delta|\lambda C \Psi_0(Q^Tb,t).\label{aggiunta2}
\end{align}

\textit{Step 3: Estimate of the nonlinear term.} 
\\
We will show, following the ideas in \cite{s.z.}, that the scaled problem linearises in the limit giving rise to the expected estimate. 
From \eqref{kerre} we have
\[\lambda_j^k:=\frac{r_j}{R\delta^{k-1}}\le\frac{r_j}{R\delta^{\tilde{k}_r}}\le\delta^{-1}\left(\frac{r_j}{R}\right)^{1-s},\]
and then the sequence $\lambda_j^k$ is infinitesimal for $j\to\infty$.
We also consider the sequence $\chi_j$ of characteristic functions
\begin{equation}
    \label{claim4}
    \chi_j:=\begin{cases}1&\text{if $|\tilde{\eta}_j|\le r_j^{(s-1)/2}$}\\ 0&\text{otherwise in ${T}^\delta_{h/R}$.}
    \end{cases}
\end{equation}
By the boundedness of $\tilde{\eta}_j$ in $L^2({T}^\delta_{h/R};\mathbb{R}^{3\times3})$ it follows that $\chi_j\to1$ in measure, so that, by (\ref{add1}), $\chi_j\tilde{\eta}_j\rightharpoonup\tilde{\eta}$ in $L^2({T}^\delta_{h/R};\mathbb{R}^{3\times3})$.
By a Taylor expansion, using assumption (ii) and (iv) on $W$,
we get 
$$
W(I+F)=\frac{1}{2}\mathbb{C}F:F+\sigma(F),
$$ 
where $\mathbb{C}:=\frac{\partial^2W}{\partial F^2}(I)$ and $\sigma(F)/|F|^2\to0$ as $|F|\to0$. Setting $\omega(t):=\sup_{|F|\le t}|\sigma(F)|$, we have
\begin{equation}
    \label{taylorexp}
    W(I+rF)\ge\frac{1}{2}r^2\mathbb{C}F:F-\omega(r|F|),
\end{equation}
with $\omega(t)/t^2\to0$ as $t\to0$.
Thus, by the frame indifference of $W$ and using (\ref{taylorexp}) and \eqref{etaj} we obtain
\begin{multline}
    \frac{1}{h\delta^{k-1}}\int_{C_i^k}\frac{W(\beta_j)}{r_j^2}\ dx=\frac{1}{h\delta^{k-1}}\int_{{C}_i^k}\frac{W(I+r_j\eta_j)}{r_j^2}\ dx\ge \frac{R}{h}\int_{{T}^\delta_{h/R}}\chi_j\frac{W(I+\lambda_j^k\tilde{\eta}_j)}{(\lambda_j^k)^2}\ dx\\
\ge\frac{R}{h}\int_{{T}^\delta_{h/R}}\left(\frac{1}{2}\mathbb{C}(\chi_j\tilde{\eta}_j):(\chi_j\tilde{\eta}_j)-\chi_j\frac{\omega(\lambda_j^k|\tilde{\eta}_j|)}{(\lambda_j^k)^2}\right)\ dx.
    \label{claim5}
\end{multline}
Again the first term in the right-hand side of \eqref{claim5} is lower semicontinuous with respect to the weak $L^2({T}^\delta_{h/R};\mathbb{R}^{3\times3})$-convergence, that is
\begin{equation}
    \label{claim6}
\liminf_{j\to\infty} \frac{R}{h}\int_{{T}^\delta_{h/R}}\frac{1}{2}\mathbb{C}(\chi_j\tilde{\eta}_j):(\chi_j\tilde{\eta}_j)\ dx\ge   \frac{R}{h}\int_{{T}^\delta_{h/R}}\frac{1}{2}\mathbb{C}\tilde{\eta}:\tilde{\eta}\ dx,
\end{equation}
while the last term in (\ref{claim5}) converges to zero as $j\to\infty$. Indeed we can rewrite its integrand as 
\begin{equation*}
    \chi_j\frac{\omega(\lambda_j^k|\tilde{\eta}_j|)}{(\lambda_j^k)^2}=|\tilde{\eta}_j|^2\cdot \chi_j\frac{\omega(\lambda_j^k|\tilde{\eta}_j|)}{(\lambda_j^k|\tilde{\eta}_j|)^2},
\end{equation*}
which is the product of a bounded sequence in $L^1({T}^\delta_{h/R})$ and a sequence converging to zero in $L^\infty({T}^\delta_{h/R})$, since $\lambda_j^k|\tilde{\eta}_j|\le\delta^{-1}R^{s-1}r_j^{(1-s)/2}$ for every $k$, when $\chi_j\ne0$. This combined with \eqref{claim6} yields
\begin{equation*}
 \liminf_{j\to\infty} \frac{1}{h\delta^{k-1}}\int_{C_i^k}\frac{W(\beta_j)}{r_j^2}dx\ge\frac{R}{h}\int_{\tilde{T}^\delta_{h/R}}\frac{1}{2}\mathbb{C}\tilde{\eta}:\tilde{\eta}\ dx.
\end{equation*}
 Finally, since $\tilde\eta$ is admissible for $\Psi(Q^Tb,t,h/R,\delta,1)$, recalling the estimate in Lemma \ref{cgolem} with $M:= \frac hR$, we get
\begin{align}
\liminf_{j\to\infty}&\frac{1}{h\delta^{k-1}}\int_{C_i^k}\frac{W(\beta_j)}{r_j^2}dx\ge \left|\log\delta\right|\Psi(Q^Tb,t,h/R,\delta,1)\notag\\
&\ge\left|\log\delta\right|\Psi_0(Q^Tb,t)\Big(1-\frac{C}{M}-\omega_M(\delta)\Big).\label{claim8}
\end{align}
From \eqref{claim8} and \eqref{aggiunta2} it follows
\begin{align*}
   \liminf_{j\to\infty} \frac{1}{h\delta^{k-1}}\int_{C_i^k}&\frac{W(\beta_j)+\lambda|\beta_j-Q|^2}{r_j^2}dx\\
   &\ge|\log\delta|\Psi_0(Q^Tb,t)\Big(1+\lambda C-\frac{C}{M}-\omega_M(\delta)\Big)\\
  & \ge |\log\delta|\Psi_0(Q^Tb,t)\Big(1-\frac CM-\omega_M(\delta)\Big).
\end{align*}

This concludes the proof of the Claim.

\end{proof}
\begin{lem}[Upper Bound]\label{upperbound} 
There exists a constant $C>0$ such that for every $Q\in SO(3)$, $b\in\mathbb{R}^3$, $t\in S^2$, $h\ge R>0$, it holds
\begin{equation*}
    \limsup_{r\to0}\Psi_\lambda^{nl}(Q,b,t,h,r,R)\le(1+\lambda C)\Psi_0(Q^Tb,t).
\end{equation*}
\end{lem}
\begin{proof}
Analogously to the previous lemma we perform the proof for $t=e_3$ (the general case is similar) so that $T_h^r:=(B'_R\setminus B'_r)\times(0,h)$.\\
We define the sequence
\begin{equation}
    \beta_r:=Q(I+r\eta_{Q^Tb,t})\quad\text{in $T_h^r$},
\end{equation}
where $\eta_{Q^Tb,t}$ is the solution to (\ref{eul2}) with $\mu_{b,t}=Q^Tb\otimes t\mathcal{H}^1\res\mathbb{R}t$.
Then $\beta_r$ is admissible for $\Psi_\lambda^{nl}(Q,b,t,h,r,R)$. We observe that from \eqref{eul3} we immediately get 
\begin{equation}
    \label{upperbound0}
    \int_{T_h^r}|\beta_r-Q|^2dx=\int_{T_h^r}|r\eta_{Q^Tb,t}|^2dx\le C|b|^2hr^2\log\frac{R}{r}.
\end{equation}
Hence by \eqref{selfen2} we have
\begin{align}\label{upper_new}
\lambda\int_{T_h^r}|\beta_r-Q|^2dx\leq\lambda C\Psi_0(Q^Tb,t)hr^2\log\frac{R}{r}.
\end{align}
Next, fix $s\in(0,1)$ and set $T_h^{r^{1-s}}:=(B'_R\setminus B'_{r^{1-s}})\times(0,h)$ and $\tilde{T}_h^r:=(B'_{r^{1-s}}\setminus B'_r)\times(0,h)$. In view of the frame indifference of $W$, we have 
\begin{equation*}
    \frac{1}{hr^2\log\frac{R}{r}}\int_{T_h^r}W(\beta_r)dx= \frac{1}{hr^2\log\frac{R}{r}}\int_{T_h^r}W(I+r\eta_{Q^Tb,t})dx
\end{equation*}
\begin{equation*}
    =\frac{1}{hr^2\log\frac{R}{r}}\int_{T_h^{r^{1-s}}}W(I+r\eta_{Q^Tb,t})dx+\frac{1}{hr^2\log\frac{R}{r}}\int_{\tilde{T}_h^r}W(I+r\eta_{Q^Tb,t})dx=:I_r^1+I_r^2.
\end{equation*}
We now estimate $I_r^1$ and $I_r^2$. Regarding $I_r^1$, by a Taylor expansion of $W$ near the identity we get
\begin{equation*}
    I_r^1=\frac{1}{h\log\frac{R}{r}}\int_{T_h^{r^{1-s}}}\frac{1}{2}\mathbb{C}\eta_{Q^Tb,t}:\eta_{Q^Tb,t}\ dx+\frac{1}{h}\int_{T_h^{r^{1-s}}}\frac{\sigma(r\eta_{Q^Tb,t})}{r^2\log\frac{R}{r}}\ dx,
\end{equation*}
where $\sigma(F)/|F|^2\to0$ as $|F|\to0$.
By (\ref{eul2}) and \eqref{selfen1} we deduce 
\begin{equation*}
    \int_{T_h^{r^{1-s}}}\frac{1}{2}\mathbb{C}\eta_{Q^Tb,t}:\eta_{Q^Tb,t}\ dx=\int_0^h\int_{r^{1-s}}^R\frac{1}{\rho^2}\Psi_0(Q^Tb,t)\rho \ d\rho dz=h\log\frac{R}{r^{1-s}}\Psi_0(Q^Tb,t),
\end{equation*}
and then
\begin{equation}
    \label{upperbound1}
    \limsup_{r\to0}\frac{1}{h\log\frac{R}{r}}\int_{T_h^{r^{1-s}}}\frac{1}{2}\mathbb{C}\eta_{Q^Tb,t}:\eta_{Q^Tb,t}\ dx
    \le\Psi_0(Q^Tb,t).
\end{equation}
Moreover setting $\chi_r:=\chi_{T_h^{r^{1-s}}}$ and again $\omega(t):=\sup_{|F|\le t}|\sigma(F)|$, we have
\begin{equation}\label{upperbound2}
    \lim_{r\to0}\left|\frac{1}{h}\int_{T_h^{r^{1-s}}}\frac{\sigma(r\eta_{Q^Tb,t})}{r^2\log\frac{R}{r}}dx\right|  \le  \lim_{r\to0}\frac{1}{h}\int_{B'_R\times(0,h)}\chi_r\frac{\omega(r|\eta_{Q^Tb,t}|)}{|r\eta_{Q^Tb,t}|^2}\cdot\frac{|r\eta_{Q^Tb,t}|^2}{r^2\log\frac{R}{r}}dx=0.
\end{equation}
In fact the above integrand is the product of a sequence converging to zero in $L^\infty(B'_R\times(0,h))$ and a bounded sequence in $L^1(B'_R\times(0,h))$ by \eqref{upperbound0}. Thus combining (\ref{upperbound1}) and (\ref{upperbound2}), we infer
\begin{equation}
    \label{upperbound3}
    \limsup_{r\to0}I_r^1\le \Psi_0(Q^Tb,t).
\end{equation}
Finally the growth assumption on $W$ and the definition of $\eta_{Q^Tb,t}$ gives
\begin{equation*}
    I_r^2\le\frac{C}{hr^2\log\frac{R}{r}}\int_{\tilde{T}_h^r}|r\eta_{Q^Tb,t}|^2dx\le C\frac{\log r^{-s}}{\log\frac{R}{r}},
\end{equation*}
then, as $s<1$, we get
\begin{equation}\label{upperbound4}
    \limsup_{r\to0}I_r^2\le Cs.
\end{equation}
In view of (\ref{upper_new}), (\ref{upperbound3}) and \eqref{upperbound4} we conclude
\begin{equation*}
    \limsup_{r\to0}\Psi_\lambda^{nl}(Q,b,t,h,r,R)\le(1+\lambda C) \Psi_0(Q^Tb,t)+Cs,
\end{equation*}
hence the thesis follows by the arbitrariness of $s\in(0,1)$.
\end{proof}
\begin{cor}
\label{cor1}
For any $Q\in SO(3)$, $b\in\mathbb{R}^3$, $t\in S^2$, $R>0$, one has
\begin{equation}
  \lim_{\lambda\to0}  \lim_{h\to\infty}\lim_{r\to0}\Psi_\lambda^{nl}(Q,b,t,h,r,R)=\Psi_0(Q^Tb,t).\label{upper+lower}
  \end{equation}
\end{cor}
\begin{proof}
The thesis is an immediate consequence of Lemma \ref{lowerbound} and Lemma \ref{upperbound}.
\end{proof}

Property \eqref{upper+lower} cannot be directly applied in the analysis of the $\Gamma$-limit in Section \ref{sec4}, in fact we need an uniform bound from below (which does not depend on the parameters) for the cell formula, in order to deal with the relaxation process. The uniform bound is proved in the following lemma.
\begin{lem}\label{lemma4}
For every $\lambda>0$, $M\ge1$, $K>0$, there exist $C>0$ and $\omega_{M,K}\colon(0,\infty)\to(0,\infty)$ with
\begin{equation*}
    \lim_{r\to0}\omega_{M,K}(r)=0,
\end{equation*}
such that
\begin{equation}
    \Psi_\lambda^{nl}(Q,b,t,h,r,R)\ge\Psi_0(Q^Tb,t)-\frac{CK^2}{M}-\omega_{M,K}\left(\frac{r}{R}\right),\label{unif1}
\end{equation}
for all $Q\in SO(3)$,  $b\in\bar{B}_K(0)\subset\mathbb{R}^3$, $t\in S^2$, $r, R$, $h>0$ with $MR\le h$. Furthermore, there exists $c_*>0$ (which does not depend on the parameters), such that, for all $\lambda$, $Q$, $b$, $t$, $h$, $r$, $R$ with $2r\le R\le h$,

\begin{equation}
\Psi_\lambda^{nl}(Q,b,t,h,r,R)\ge\left(1-\frac{R}{h}\right) c_*|b|^2.
    \label{unif2}
\end{equation}
\end{lem}
\begin{oss}
	As a consequence of the above estimate we also deduce an uniform bound from below of the elastic energy which will be needed for the compactness argument. Namely if $\beta\in L^2(T^r_h;\R^{3\times3})$ is a test function for $\Psi^{nl}_\lambda(Q,b,t,h,r,R)$ then from \eqref{unif2} we get
\begin{equation}\label{consequence1}
\frac{1}{hr^2\log\frac{R}{r}}\int_{T^r_h}W(\beta)dx\ge\lim_{\lambda\to0}\Psi^{nl}_\lambda(Q,b,t,h,r,R)\ge\left(1-\frac Rh\right)c_*|b|^2.
    \end{equation}
\end{oss}

For the proof of Lemma \ref{lemma4} we follow the approach proposed in \cite[Lemma 5.7]{C.G.O.} which requires a preliminary result:
\begin{lem}\label{equicontinuity}
Let $\lambda>0$, $Q\in SO(3)$, $H:=\{(b,t,h)\in\mathbb{R}^3\times S^2\times(1,\infty)\}$, $K\subset H$ compact. 
The family of functions $\Psi_\lambda^{nl}(Q,\cdot,\cdot,\cdot,r,1)$, $r\in(0,1)$, is equicontinuous on $K$.
\end{lem}
\begin{proof} Let $Q\in SO(3)$, $\lambda>0$ be fixed.
It suffices to prove equicontinuity of $\Psi_\lambda^{nl}$ with respect to $r$ separately in each variable $b$, $t$ and $h$.
\\
\noindent\textit{Step 1. Continuity in $b$.} We show that there exists a constant $c>0$ (not depending on $r$) such that for all $b,b'\in\mathbb{R}^3$, it holds
\begin{equation*}
\Psi_\lambda^{nl}(Q,b',t,h,r,1)-\Psi_\lambda^{nl}(Q,b,t,h,r,1)\le c|b'-b|(1+c|b'-b|)
\end{equation*}
for all $t\in S^2,h>0, r\in(0,1)$.\\

Let $b,b'\in\mathbb{R}^3$, consider $\beta$ test function for $\Psi_\lambda^{nl}(Q,b,t,h,r,1)$ and define 
\[\beta':=\beta+r\eta_{b'-b,t},\]
with $\eta_{b'-b,t}$ the solution of \eqref{eul1} for $\mu_{b'-b,t}$. Then for every $\delta>0$ we have 
\begin{align*}
    \int_{T_h^r}|\beta'-Q|^2dx&\le (1+\delta)\int_{T_h^r}|\beta-Q|^2dx+ \left(1+\frac{1}{\delta}\right)\int_{T_h^r}|\beta'-\beta|^2dx\\
 &=(1+\delta)\int_{T_h^r}|\beta-Q|^2dx+
    \left(1+\frac{1}{\delta}\right)\int_{T_h^r}|r\eta_{b'-b,t}|^2dx.
\end{align*}
Moreover by \eqref{jacob} we get
\begin{align*}
    \int_{T_h^r}W(\beta')dx-\int_{T_h^r}W(\beta)dx&=\int_{T_h^r}\int_0^1 \frac{\partial W}{\partial F} (\beta+s(r\eta_{b'-b,t})):(r\eta_{b'-b,t})\ dsdx\\
& \le c\int_{T_h^r}\int_0^1\dist\bigl(\beta+s(r\eta_{b'-b,t}),SO(3)\bigr)\left|r\eta_{b'-b,t}\right|dsdx\\
& \le c\int_{T_h^r}\int_0^1\bigl[\dist(\beta,SO(3))+s\left|r\eta_{b'-b,t}\right|\bigr]\left|r\eta_{b'-b,t}\right|dsdx\\
& \le c\int_{T_h^r}\dist(\beta,SO(3))\left|r\eta_{b'-b,t}\right|dx + c\int_{T_h^r}\frac{1}{2}\left|r\eta_{b'-b,t}\right|^2dx.
\end{align*}
Again for every $\delta>0$ it holds
\begin{equation*}
    \int_{T_h^r}\dist(\beta,SO(3))\left|r\eta_{b'-b,t}\right|dx\le \frac{\delta}{2}\int_{T_h^r}\dist^2(\beta,SO(3))\ dx+\frac{1}{2\delta}\int_{T_h^r}\left|r\eta_{b'-b,t}\right|^2dx.
\end{equation*}
Then setting $\delta=|b'-b|$, after rescaling and using the arbitrariness of the test function $\beta$ we finally get
\begin{equation*}
    \Psi_\lambda^{nl}(Q,b',t,h,r,1)-\Psi_\lambda^{nl}(Q,b,t,h,r,1)\le c|b'-b|(\Psi_\lambda^{nl}(Q,b,t,h,r,1)+1)+c|b'-b|^2.
\end{equation*}
We conclude the proof of Step 1
observing that by the growth condition on $W$
\begin{align}
    \Psi_\lambda^{nl}(Q,b,t,h,r,1)&\le\frac{1}{hr^2\log\frac{1}{r}}\int_{T_h^r}W(Q+r\eta_{b,t})+\lambda|r\eta_{b,t}|^2dx\nonumber\\
   & \le \frac{c+\lambda}{hr^2\log\frac{1}{r}}\int_{T_h^r}|r\eta_{b,t}|^2dx\le c(1+\lambda)|b|^2.\label{stimadasu}
\end{align}
\noindent\textit{Step 2. Continuity in $t$.} We now show that there is a constant $c>0$ such that for all $t$, $t'\in S^2$,
\begin{equation*}
\Psi_\lambda^{nl}(Q,b,t',h,r,1)-\Psi_\lambda^{nl}(Q,b,t,h,r,1)\le c(|t'-t|+|t'-t|^2)
\end{equation*}
for all $b\in\R^3,h>0,r\in(0,1)$.

We choose $t$, $t'\in S^2$ and fix as above $\beta$ to be an admissible strain in the definition of $\Psi_\lambda^{nl}(Q,b,t,h,r,1)$. The function 
\begin{equation*}
\beta'(x):=\beta(S^Tx)S^T-Q(S^T-I),\quad S:=Q_{t'}Q_t^T\in SO(3),
\end{equation*}
is admissible for $\Psi_\lambda^{nl}(Q,b,t',h,r,1)$ and 
\begin{equation}\label{conttt}
|S^T-I|\le c|t'-t|\quad\forall t',t\in S^2.
\end{equation}
If we let $\tilde\beta(x):=\beta'(Sx)$ then a change of variable gives
\begin{equation*}
\int_{ST}W(\beta')+\lambda|\beta'-Q|^2dx=\int_TW(\tilde\beta)+\lambda|\tilde\beta-Q|^2dx,
\end{equation*}
with $T:=Q_t((B'_1\setminus B'_r)\times(0,h))$. For every $\delta>0$, adding and subtracting $\beta$ to $\tilde\beta-Q$ yields the inequality
\begin{equation*}
	\int_T|\tilde\beta-Q|^2\le\left(1+\frac{1}{\delta}\right)c|t'-t|^2
	\int_T|\beta-Q|^2dx+(1+\delta)\int_T|\beta-Q|^2dx,
\end{equation*}
so that choosing $\delta=|t'-t|$ we find
\begin{equation}\label{ldue}
\lambda\int_T(|\tilde\beta-Q|^2-|\beta-Q|^2)dx\le \lambda c(|t'-t|+|t'-t|^2)\int_T|\beta-Q|^2dx.
\end{equation}
Moreover
\begin{align*}
\int_TW(\tilde\beta)dx&-\int_TW(\beta)dx=\\
&=\int_T\int_0^1\frac{\partial W}{\partial F}(\beta+s(\beta-Q) (S^T-I)):((\beta-Q)(S^T-I))dx\\
&\le c|t'-t|\int_T\int_0^1\dist(\beta+s(\beta-Q) (S^T-I),SO(3))\cdot|\beta-Q|dsdx\\
&\le c|t'-t|\int_T\int_0^1\bigl[\dist(\beta,SO(3))+sc|t'-t||\beta-Q|\bigr]\cdot|\beta-Q|dsdx\\
&\le c|t'-t|\int_T\dist(\beta,SO(3))\cdot|\beta-Q|dx+c|t'-t|^2\int_T|\beta-Q|^2dx.
\end{align*}
Now using Young's inequality we find
\begin{equation*}
\int_T\dist(\beta,SO(3))\cdot|\beta-Q|dx\le \frac12\int_T\dist^2(\beta,SO(3))dx+\frac12\int_T|\beta-Q|^2dx.
\end{equation*}
Therefore by the growth estimate on $W$ we get
\begin{equation*}
\int_TW(\tilde\beta)-W(\beta)dx\le c(|t'-t|-|t'-t|^2)\int_TW(\beta)+|\beta-Q|^2dx.
\end{equation*}
Combining this with \eqref{ldue}, we have
\begin{equation*}
\Psi_\lambda^{nl}(Q,b,t',h,r,1)-\Psi_\lambda^{nl}(Q,b,t,h,r,1)\le
c(|t'-t|+|t'-t|^2)\Psi^{nl}_{\lambda+1}(Q,b,t,h,r,1)
\end{equation*}
which again gives the conclusion using \eqref{stimadasu}.\medspace

\noindent\textit{Step 3. Continuity in $h$.} 
We show that there exists a constant $c>0$ such that for $h, h'>1$,
 \begin{equation*}
|\Psi_\lambda^{nl}(Q,b,t,h',r,1)-\Psi_\lambda^{nl}(Q,b,t,h,r,1)|\le c|h'-h|
 \end{equation*}
 for all $b\in \R^3,t\in S^2$ with $(b,t,h)\in K$, and for all $r\in(0,1)$.
 
Assume that $h'>h$ and notice that any test function for $\Psi_\lambda^{nl}(Q,b,t,h',r,1)$ can be restricted to get a test function for $\Psi_\lambda^{nl}(Q,b,t,h,r,1)$ and in particular
\begin{equation}\label{importante}
    \Psi_\lambda^{nl}(Q,b,t,h,r,1)\le\frac{h'}{h}\Psi_\lambda^{nl}(Q,b,t,h',r,1),
\end{equation}
or equivalently using \eqref{stimadasu}
\begin{align*}
   \Psi_\lambda^{nl}(Q,b,t,h,r,1)-\Psi_\lambda^{nl}(Q,b,t,h',r,1)&\le \frac{h'-h}{h}\Psi_\lambda^{nl}(Q,b,t,h',r,1)\\&\le c(1+\lambda)|b|^2\frac{h'-h}{h}.
\end{align*}
It remains to show the other inequality. As usual we assume $t=e_3$. 
Let $\beta$ be admissible for $\Psi_\lambda^{nl}(Q,b,t,h,r,1)$ and let $u\in W^{1,2}(T_h^r;\mathbb{R}^3)$ be such that $\nabla u=\beta-r\eta_{b,t}$. Then we estimate
\begin{align*}
    \int_{T_h^r}|\nabla u-Q|^2dx&
    \le 2\int_{T_h^r}|\beta-Q|^2dx+2\int_{T_h^r}|r\eta_{b,t}|^2dx\\
   & \le 2\int_{T_h^r}|\beta-Q|^2dx+c|b|^2hr^2\log\frac{1}{r}.
\end{align*}
We denote by $w\in W^{1,2}(T_h^r,\mathbb{R}^3)$  the function such that $\nabla w=\nabla u-Q$ and 
we choose $y_3\in(0,h)$ such that
\begin{align}\label{stimadiw}
    \int_{B'_1\setminus B'_r}\left|\nabla w(x',y_3)\right|^2dx'&\le\frac{1}{h}\int_{T_h^r}\left|\nabla w(x',x_3)\right|^2dx \notag\\
 & \le \frac2h\int_{T_h^r}|\beta-Q|^2dx+ c|b|^2r^2\log\frac{1}{r}.
\end{align}
Next we define the extension $\tilde{w}\in W^{1,2}(T_{h'}^r;\mathbb{R}^3)$ with $T_{h'}^r:=(B'_1\setminus B'_r)\times(0,h')$
\begin{equation*}
    \tilde{w}(x',x_3):=\begin{cases}w(x',x_3)&\text{if $x_3\in(0,y_3)$}\\
    w(x',y_3)&\text{if $x_3\in(y_3,y_3+h'-h)$}\\
    w(x',x_3-(h'-h))&\text{if $x_3\in(y_3+h'-h,h')$}.
    \end{cases}
\end{equation*}
The function $\tilde{\beta}:=r\eta_{b,t}+\nabla\tilde{w}+Q$ is admissible for $\Psi_\lambda^{nl}(Q,b,t,h',r,1)$ and satisfies
\begin{align*}
    \int_{T_{h'}^r}|\tilde{\beta}-Q|^2dx&=\int_{T_{h}^r}|\beta-Q|^2dx +(h'-h)\int_{B'_1\setminus B'_r}|r\eta_{b,t}+\nabla \tilde w(x',y_3)|^2dx'\\
   & \le \int_{T_{h}^r}|\beta-Q|^2dx +(h'-h)\left( \frac{2}{h}\int_{T_h^r}|\beta-Q|^2dx+ c|b|^2r^2\log\frac{1}{r}\right),
\end{align*}
where we have used \eqref{stimadiw}.
Furthermore using the growth assumption on $W$ we find
\begin{align*}
    \int_{T_{h'}^r}W(\tilde{\beta})dx &=\int_{T_h^r}W(\beta)dx+(h'-h)\int_{B'_1\setminus B'_r}W(r\eta_{b,t}+\nabla \tilde w(x',y_3)+Q)dx'\\
& \le \int_{T_{h}^r}W(\beta)dx+ c(h'-h)\int_{B'_1\setminus B'_r}|r\eta_{b,t}+\nabla \tilde w(x',y_3)|^2dx'\\
&\le\int_{T_{h}^r}W(\beta)dx+ 
    c(h'-h)\left( \frac{c}{h}\int_{T_h^r}|\beta-Q|^2dx+ c|b|^2r^2\log\frac{1}{r}\right).
\end{align*}
Then we conclude
\begin{align*}
    \Psi_\lambda^{nl}(Q,b,t,h',r,1)-\Psi_\lambda^{nl}(Q,b,t,h,r,1)&\le c\frac{h'-h}{h}\left(|b|^2+\Psi_{\lambda+1}^{nl}(Q,b,t,h,r,1)\right)\\&\le c(1+\lambda)|b|^2\frac{h'-h}{h}.
\end{align*}

\end{proof}

\begin{proof}[Proof of Lemma \ref{lemma4}] By property \eqref{scalingg} we can reduce to $R=1$. In addition we can assume $h\in[M,2M]$. In fact for any $N\in\mathbb{N}$, we can subdivide $T_h^r$ into $N$ cylinders $T_i$, $i=1,...,N$, of length $h/N$, and for any $\beta$ test function for $\Psi_\lambda^{nl}(Q,b,t,h,r,1)$ we have
\begin{equation*}
 \int_{T_h^r}W(\beta)+\lambda|\beta-Q|^2dx=\sum_{i=1}^N\int_{T_i}W(\beta)+\lambda|\beta-Q|^2dx\ge N\min_i\int_{T_i}W(\beta)+\lambda|\beta-Q|^2dx
\end{equation*}
and therefore
\begin{equation}
    \Psi_\lambda^{nl}(Q,b,t,h,r,1)\ge \Psi_\lambda^{nl}(Q,b,t,\frac{h}{N},r,1).
    \label{unif3}
\end{equation}
Consider the compact set $H_{M,K}:=\bar{B}_K(0)\times S^2\times[M,2M]$ and recall that  by the previous  lemma $\Psi_\lambda^{nl}(Q,\cdot,\cdot,\cdot,r,1)$ is equicontinuous on $H_{M,K}$. Therefore  for any $j\in\mathbb{N}$ there exists $\delta_j>0$ such that
\begin{equation}\label{equicont}
\left|\Psi_\lambda^{nl}(Q,b,t,h,r,1)-\Psi_\lambda^{nl}(Q,b',t',h',r,1)\right|\le1/j,
\end{equation}
 if $|b'-b|+|t'-t|+|h'-h|\le\delta_j$, $(b,t,h)$, $(b',t',h')\in H_{M,K}$, $r\in(0,1)$. 
 We cover $H_{M,K}$ with finitely many balls of radius $\delta_j$. By Lemma \ref{lowerbound}, at every center $(b_l,t_l,h_l)$ of these balls, we have 
\begin{equation*}
    \liminf_{r\to0}\Psi_\lambda^{nl}(Q,b_l,t_l,h_l,r,1)\ge\Psi_0(Q^Tb_l,t_l)-\frac{C}{M}|b_l|^2.
\end{equation*}
Therefore for every $l$ there exists $r_j(l)\in(0,r_{j-1}(l))$ such that, for all $r<r_j(l)$ it holds
\begin{equation}\label{compact}
    \Psi_\lambda^{nl}(Q,b_l,t_l,h_l,r,1)\ge\Psi_0(Q^Tb_l,t_l)-\frac{C}{M}|b_l|^2-\frac{1}{j}.
\end{equation}
Define $r_j:=\min_lr_j(l)$ and for any $(b,t,h)\in H_{M,K}$ let $(b_l,t_l,h_l)$ be the center of the ball of radius $\delta_j$ that contains $(b,t,h)$. Then for every $r<r_j$ we get
\begin{align*}
    \Psi_\lambda^{nl}(Q,b,t,h,r,1)-\Psi_0(Q^Tb,t))&=(\Psi_\lambda^{nl}(Q,b,t,h,r,1)-\Psi_\lambda^{nl}(Q,b_l,t_l,h_l,r,1))\\
    &\,+(\Psi_\lambda^{nl}(Q,b_l,t_l,h_l,r,1)-\Psi_0(Q^Tb,t)\\
   & \ge  -\frac{3}{j}-\frac{c}{M}K^2,
\end{align*}
where we have used \eqref{compact}, \eqref{equicont}, and the continuity of $\Psi_0$.
Then, set $\omega_{M,K}(r):=3/j$ for $r\in(r_{j+1},r_j]$ and (\ref{unif1}) is proven.\\

It remains to  prove (\ref{unif2}). We assume $t=e_3$.
Using first  \eqref{importante} and then \eqref{unif3} for $N=\lfloor h\rfloor$ we get
\begin{equation}\label{importantebis}
    \Psi_\lambda^{nl}(Q,b,t,h,r,1)\ge\frac{\lfloor h\rfloor}{h}\Psi_\lambda^{nl}(Q,b,t,\lfloor h\rfloor,r,1)\ge\frac{\lfloor h\rfloor}{h}\Psi_\lambda^{nl}(Q,b,t,1,r,1). 
\end{equation}
Let $\beta$ be admissible for $\Psi_\lambda^{nl}(Q,b,t, 1,r,1)$, and denote $T_1^r:=(B'_1\setminus B'_r)\times(0,1)$. Then, by the growth condition on $W$ and the rigidity estimate on cylinders, Corollary \ref{rigiditytris}, there exist $c>0$ (independent of the parameters) and $\tilde{Q}\in SO(3)$ such that 
$$
 \int_{T_{1}^r}W(\beta)dx\ge c\int_{T_{1}^r}\left|\beta-\tilde{Q}\right|^2dx,
$$
and by Jensen inequality and using the condition on the curl of $\beta$ we obtain
 \begin{align*}
    \int_{T_{1}^r}W(\beta)dx\ge & c\  \int_0^{1}\int_r^1\frac{1}{2\pi\rho}\left|\int_0^{2\pi}(\beta -\tilde{Q})\cdot e_\theta\rho\ d\theta\right|^2d\rho dz\\
 =&c\int_0^{1}\int_r^1\frac{1}{2\pi\rho}\left|\int_0^{2\pi}\beta\cdot e_\theta\rho\ d\theta\right|^2d\rho dz\\
 =&c\int_0^{1}\int_r^1\frac{r^2}{2\pi\rho}|b|^2d\rho dz=c|b|^2r^2\log\frac{1}{r}.
\end{align*}
Analogously we obtain
\begin{equation*}
 \int_{T_{1}^r}\lambda|\beta-Q|^2dx\ge\lambda c|b|^2r^2\log\frac{1}{r}.
\end{equation*}
 Thus from \eqref{importantebis} we conclude
\begin{equation*}
\Psi_\lambda^{nl}(Q,b,t,h,r,1)\ge\frac{\lfloor h\rfloor}{h}(1+\lambda)c|b|^2\ge \Big(1-\frac{1}{h}\Big)c|b|^2.
\end{equation*}

\end{proof}

\section{An uniform estimate of the quadratic energy}\label{sec3}

This section is devoted to the analysis of the asymptotic behaviour of the quadratic energy associated to a sequence of dilute dislocations. Indeed, as already pointed out, our approach exploits the fact that the  energy asymptotically linearises near a suitable rotation. 
We will provide in Proposition \ref{stimaL2}
an uniform $L^2$ estimate of the strains corresponding to a sequence of dilute dislocations $\mu_\varepsilon\in{\mathcal{M}}_\mathcal{B}^{h_\varepsilon,\alpha_\varepsilon}(\bar\Om)$. This result will be crucial for the proof of compactness in Section \ref{sec4}.\\

It is indeed known (see \cite{BrB}) that for any measure $\mu\in \mathcal M(\R^3)$ there exists a unique  $\eta\in L^{3/2}(\R^3;\R^{3\times3})$ which is a distributional solution of
\begin{align}\label{divsystem}
\begin{cases}
\Div \eta=0&\text{ in }\R^3\\
\curl \eta=\mu&\text{ in }\R^3.
\end{cases}
\end{align}
In general we cannot expect a better summability for $\eta$, in particular it might not belong to $L^p$ if $p>3/2$ (see \cite[Lemma 4.1]{C.G.O.}). Nevertheless if we assume that the measure is polyhedral it is possible to show, by an explicit computation, that $\eta$ is in $L^p$ for any $p<2$. 
In particular one can show that if $\mu\in \mathcal M_{\mathcal B}(\R^3)$ is a dislocation measure of the form
\begin{align}\label{polymeas}
    \mu:=\sum_ib_i\otimes t_i\mathcal H^1\res \gamma_i,
\end{align}
where $\gamma_i$ are straight segments, then
there is a constant $C>0$, independent of $\mu$, such that the following estimates hold
\begin{align}
    |\eta(x)|\leq \frac{C|\mu|(\R^3)}{\dist(x,\supp\mu)^2},\label{est_1}\\
    |\eta(x)|\leq C\sum_i\frac{|b_i|}{\dist(x,\gamma_i)}.\label{est_2}
\end{align}
This is for instance proved in \cite[Lemma 4.1]{C.G.O.} in a more general context.

Note that in these estimates $\mu$ needs to be defined in the whole of $\R^3$ and divergence free. Therefore in order to use them we will extend any measure in $\mathcal{M}^{h,\alpha}_\mathcal{B}(\bar\Om)$ to a measure in $\mathcal{M}_\mathcal{B}(\R^3)$. In what follows we will first prove that similar estimates hold true for a larger class of measures, that we will call deformed polyhedral, namely measures of the form \eqref{polymeas} where $\gamma_i$ are segments deformed under suitably regular maps. Second, we will extend a dilute measure $\mu$ in $\bar\Omega$ to a deformed polyhedral one in the whole of $\R^3$.\\
Precisely, a dislocation measure $\mu\in \mathcal M_{\mathcal B}(\R^3)$ is said to be a {\it deformed polyhedral measure }if it is of the form \eqref{polymeas}
where  now $\gamma_i:=\Phi_i(L_i)$ for some $L_i$ straight segments and $\Phi_i:\R^3\rightarrow\R^3$ uniformly bi-Lipschitz functions, i.e., there is a positive constant $\ell>1$ such that, for all $i$ 
\begin{align}\label{BL}
    \frac{1}{\ell}|x-y|\leq |\Phi_i(x)-\Phi_i(y)|\leq \ell|x-y|,\;\;\;\;\;\text{for all }x,y\in \R^3.
\end{align}

 \begin{lem}\label{lemma_deform}
Let $\mu\in \mathcal M_{\mathcal B}(\R^3)$ be a deformed polyhedral dislocation measure. Let $\eta\in L^{\frac32}(\R^3;\R^{3\times3})$ be the solution to \eqref{divsystem}. Then there is a constant $C>0$ such that, for all $x\notin \supp\mu$, \eqref{est_1} and \eqref{est_2} hold.
\end{lem}

\begin{proof} The proof follows the argument in Lemma 4.1 in \cite{C.G.O.}, by means of a change of variable and exploiting the uniform bi-lipschitz condition in \eqref{BL}. For the reader convenience we sketch the main points.\\
From \eqref{divsystem}, since $\curl\curl \eta= \nabla\Div \eta-\Delta \eta$,  we obtain that $\eta=(-\Delta)^{-1}(\curl \mu)$, and therefore for all $x\notin \supp \mu$
\begin{align}
	\eta_{ij}(x)=\int_{\R^3}K_{jk}(x-y)d\mu_{ik}(y),
\end{align}
where $K_{lm}(x)=- \varepsilon_{lmk}\frac{x_k}{4\pi|x|^3}$. 
Then we have
\begin{align}
	|\eta(x)|\leq C\sum_i|b_i|\int_{\gamma_i}\frac{1}{|x-y|^2}d\mathcal H^1(y).\label{6.10}
\end{align}	
	
Estimate \eqref{est_1} is  easily achieved. To obtain \eqref{est_2} we use the property  of $\Phi_i$, \eqref{BL}, and with a change of variables we write
\begin{equation}
    \int_{\gamma_i}\frac{1}{|x-y|^2}d\mathcal H^1(y)\leq \ell\int_{L_i}\frac{1}{|x'-y'|^2}d\mathcal H^1(y')\leq C\frac{\ell}{\dist(x',L_i)},
\end{equation}
where $x'=\Phi_i^{-1}(x)$, $y'=\Phi_i^{-1}(y)$, and the last inequality can be obtained by a direct estimate (see \cite{C.G.O.}).
Again by assumption \eqref{BL} we have $|x'-y'|\geq\frac{1}{\ell}|x-y|$ and hence 
\begin{align}
    \dist(x',L_i)\geq \ell\dist(x,\gamma_i),
\end{align}
which in turn implies
$$    \int_{\gamma_i}\frac{1}{|x-y|^2}d\mathcal H^1(y)\leq C\frac{1}{\dist(x,\gamma_i)},$$
for some constant $C>0$ independent of $i$. Therefore, together with \eqref{6.10}, we readily conclude.
\end{proof}

We then state and prove an extension lemma for a polyhedral measure $\mu\in \mathcal M_{\mathcal B}(\Om)$, to a deformed polyhedral measure $\tilde \mu\in \mathcal M_{\mathcal B}(\R^3)$,    
 which is a refinement of Lemma 2.3 proved in \cite{C.G.M.}. We recall that we are working with $\Om$ a $C^2$ domain.

\begin{lem}\label{extension_2}
Let $\mu\in \mathcal M_{\mathcal B}(\Om)$ be a polyhedral dislocation measure, $\mu:=\sum_{i\in I}b_i\otimes t_i\mathcal H^1\res \gamma_i$, with $\gamma_i$ straight segments. Then there is a measure $\tilde \mu\in \mathcal M_{\mathcal B}(\R^3)$ that satisfies the following properties:
\begin{enumerate}[(i)]
    \item $\tilde \mu$ extends $\mu$, that is $\tilde \mu\res\Om=\mu$;
    \item $\tilde \mu$ is deformed polyhedral, $$    \tilde\mu:=\sum_{i\in \tilde I}\tilde b_i\otimes \tilde t_i\mathcal H^1\res \tilde\gamma_i;$$

    \item there is a constant $C>0$ depending only on the domain $\Om$ such that the following estimates hold
    \begin{align}\label{ext1}
    &|\tilde \mu|(\R^3)\leq C|\mu|(\Om),\\
    \label{ext2}
    &\sum_{i\in\tilde I}|\tilde b_i|^2\mathcal H^1(\tilde\gamma_i)\leq C\sum_{i\in  I}| b_i|^2\mathcal H^1(\gamma_i),\\
    &\sharp(\tilde I)\le \sharp(I)+C|\mu|(\Om).\label{ext3}
    \end{align}
\end{enumerate}
\end{lem}

 \begin{proof}The proof of Lemma  \ref{extension_2} can be straightforwardly achieved by following the lines of the proof of Lemma 2.3 in \cite{C.G.M.}, with the addition of minimal changes. We discuss it without much detail.
We can first define the extension $\tilde \mu$ in an outer neighbourhood of $\partial\Om$ by a reflection argument. If the boundary of $\Om$ is of class $C^2$ we can use the reflection $\Phi$ through $\partial\Om$ that is a bi-Lipschitz transformation which send the set 
$$\partial\Om^i_s:=\{x\in \Om:\dist(x,\partial\Om)=s\},$$
onto
$$\partial\Om^e_{s}:=\{x\in \R^3\setminus\Om:\dist(x,\partial\Om)=s\},$$
for any $s\leq s_0$, where $s_0>0$ is small enough and  depends only on the geometry of $\Om$.
Since by slicing it holds
\begin{equation*}
    \sum_{i\in I}|b_i|\int_{s_0/2}^{s_0}\mathcal{H}^0(\gamma_i\cap\partial\Om_s)ds\le|\mu|(\Om),
\end{equation*}
and
\begin{equation*}
    \sum_{i\in I}|b_i|^2\int_{s_0/2}^{s_0}\mathcal{H}^0(\gamma_i\cap\partial\Om_s)ds\le C\sum_{i\in I}|b_i|^2\mathcal{H}^1(\gamma_i),
\end{equation*}
we can find $\bar s\in(s_0/2,s_0)$ such that the slice of the measure $\mu$ on $\partial\Om_{\bar s}^i$ satisfies
\begin{align}
   &   \sum_{i\in I}|b_i|\mathcal{H}^0(\gamma_i\cap\partial\Om^i_{\bar s})\leq C|\mu|(\Om),\\
    &\sum_{i\in I}|b_i|^2\mathcal{H}^0(\gamma_i\cap\partial\Om^i_{\bar s})\leq C\sum_{i\in  I}|b_i|^2\mathcal H^1(\gamma_i),
\end{align}
for some constant $C>0$ independent of $\mu$.
In particular $\partial\Om^i_{\bar s}$ intersect the support of $\mu$ in a finite number of points, i.e.,
\begin{equation*}
    \mathcal{H}^0(\partial\Om^i_{\bar s}\cap(\supp\mu))\le|\mu|(\Om).
\end{equation*}
Thus we set $\widehat \mu:=\Phi_\sharp(\mu\res(\partial \Om)^i_{\bar s})$, where $(\partial \Om)^i_{\bar s}=\cup_{s\in(0,\bar s)}\partial \Om^i_s$. 
Notice that $\widehat \mu$ is of the form
\begin{equation}\label{muhat}
    \widehat \mu=\sum_{j\in \widehat I}b_j\otimes \hat t_j\mathcal H^1\res\widehat\gamma_j,
\end{equation}
where $\widehat\gamma_j$ is the image by $\Phi$ of some straight segment $\gamma_j$ belonging to the support of $\mu$ and $\hat t_j$ the unit tangent to $\widehat\gamma_j$.

We then extend $\widehat \mu$ outside $\Om$ by connecting the endpoints of $\widehat\gamma_j$ to each other using piecewise affine curves and associating to them a suitable multiplicity (for details we refer to Step 3 of \cite[Lemma 2.3]{C.G.M.}). The final extension is obtained as
\begin{equation}
\tilde \mu:=\widehat \mu+\mu.    
\end{equation}
Property \eqref{ext3} easily follows by the construction.

\end{proof}

\begin{oss} \label{ossex}
\begin{enumerate}[(a)]
    \item The support of $\tilde \mu$ is given by curves consisting of straight segments with the only exception of the part of $\supp\tilde \mu$ which intersect  the set 
$(\partial \Om)^e_{\bar s}=\cup_{s\in(0,\bar s)}\partial \Om^e_s$ where the curves $\gamma_i$ of the  support of $\tilde\mu$ are obtained by reflecting by $\Phi$ the segments of $\mu$ contained in $(\partial\Om)_{\bar s}^i$. We denote with  $I^e_{\bar s}$ the set indices corresponding to the curves $\gamma_i\subseteq (\partial \Om)^e_{\bar s}$.
\item The bi-Lipschitz map $\Phi$ defined in the proof of Lemma \ref{extension_2}  is not a global map in the whole of $\R^3$ (as it is required in the hypothesis of Lemma \ref{lemma_deform}). However with a covering argument and using the regularity of $\partial\Omega$ the map $\Phi$ can be extended locally in order to fulfil the assumptions in the lemma. Precisely by the fact that $\partial \Omega$ is $C^1$  we can choose $s_0$ such that   for every curve $\gamma_i\subseteq (\partial \Om)^e_{\bar s}$, with $i\in I^e_{\bar s}$,  we can construct a $\Phi_i$ satisfying \eqref{BL} which coincides with $\Phi$ in a neighbourhood of $\gamma_i$ of diameter $s_0$. Therefore,  if $\tilde\eta\in L^{3/2}(\R^3;\R^{3\times3})$ is the unique solution to \eqref{divsystem} with $\mu$ replaced by $\tilde\mu$, then $\tilde\eta$ still satisfies estimates \eqref{est_1} and \eqref{est_2};

\item Since the piecewise straight lines built outside $\Om\cup(\partial \Om)^e_{\bar s}$ are arbitrary, we can assume they consist of segments which have minimal length $h>0$. This will be useful when we extend a dilute measure $\mu$ in $\Om$;
\item From the proof of Lemma \ref{extension_2} we obtain that the number of points in $\supp(\gamma)\cap \partial \Om_{\bar s}$ weighted with the norm of their Burgers vectors (and their squares) is controlled, namely
\begin{equation}
 \sum_{i\in I}|b_i|\mathcal{H}^0(\gamma_i\cap\partial\Om^i_{\bar s})\leq C|\mu|(\Om),\qquad
\sum_{i\in I}|b_i|^2\mathcal{H}^0(\gamma_i\cap\partial\Om^i_{\bar s})\leq C\sum_{i\in  I}|b_i|^2\mathcal H^1(\gamma_i).
\label{numeroburgers}
\end{equation}
\end{enumerate}
\end{oss}
We set $\varphi_\eps(x)=\eps^{-3}\varphi(\frac{x}{\eps})$ a mollifier with $\varphi_\varepsilon\le c\frac{ \chi_{B_{\eps}(0)}}{|B_{\varepsilon}(0)|}$ for $c>0$.
\begin{prop} \label{stimaL2}
	Let $\mu_\varepsilon:=\sum_{i=1}^{M_\varepsilon}b_\varepsilon^i\otimes t_\varepsilon^i\mathcal{H}^1\res\gamma_\varepsilon^i$ be a sequence in ${\mathcal{M}}_\mathcal{B}^{h_\varepsilon,\alpha_\varepsilon}(\bar\Om)$ such that 
	\begin{equation}\label{eq0}
	\sum_{i=1}^{M_\varepsilon}|b_\varepsilon^i|^2\mathcal{H}^1(\gamma_\varepsilon^i)\le C<\infty.
	\end{equation}
	Then there exists a sequence $\tilde\eta_\eps\in L^{3/2}(\R^3;\R^{3\times3})$ with the following properties
		\begin{enumerate}[(i)]
		\item $\curl\tilde{\eta}_\varepsilon={\mu}_\varepsilon$ distributionally in $\Om$;
		\item there exists a constant $c>0$ (independent of $\varepsilon$) such that $$\int_{\Omega_\eps(\mu_\eps)}|\tilde{\eta}_\varepsilon|^2dx\le c|\log\varepsilon|,$$
		where $\Om_\eps(\mu_\eps):=\{x\in\Om:\dist(x,\supp\mu_\eps)>\eps\}$.
	\end{enumerate}
Furthermore 
 there exist a sequence $\hat\eta_\varepsilon\in L^{3/2}(\R^3;\R^{3\times3})\cap L^2(\Om;\R^{3\times3})$ and a constant $c>0$ (independent of $\eps$)  such that for every $\Om'\subset\subset\Om$ 
	\begin{enumerate}[(i')]
		\item $\curl\hat{\eta}_\varepsilon=\mu_\varepsilon*\varphi_\eps$ distributionally in $\Om'$;
		\item  $\int_\Omega|\hat{\eta}_\varepsilon|^2dx\le c|\log\varepsilon|.$
	\end{enumerate}
\end{prop}
Before proving the proposition we fix some notation and give a preliminary result.
If $\gamma$ is a segment,  for every $\rho>0$ and $\delta\geq0$, we denote by $T_{\rho,\delta}(\gamma)$ the cylinder with vertical axis $\gamma$, radius $\rho$, and length $\mathcal{H}^1(\gamma)-2\delta$, namely
\begin{equation}\label{cilindri_ext}
    T_{\rho,\delta}(\gamma):=A(B'_{\rho}\times S_\delta),
\end{equation}
where $B'_\rho$ is the ball in $\R^2$ of radius $\rho$ centered at the origin, $S_\delta\subset\R$ is a segment of length $\mathcal{H}^1(\gamma)-2\delta$ and $A$ an affine transformation that maps $S_\delta$ in $\gamma$ and the midpoint of $S_\delta$ into the midpoint of $\gamma$. 
To shorten the notation, if $\gamma$ is fixed, we will simply write $T_{\rho,\delta}$. 
\begin{lem}\label{Stima_Palla}
	Let  $\Phi\colon\R^3\to\R^3$ be a bi-Lipschitz map and $\Omega\subseteq\R^3$. There exists a constant $C>0$ such that 
given a segment $\gamma$, a number $\delta\geq 0$, and positive parameters $\rho>\varepsilon$, it holds
\begin{equation}\label{stimapalla1}
    \int_{\Om\setminus ( U\cup V)}\frac{1}{\dist^2(x,\tilde\gamma)}dx\leq C\left( \mathcal{H}^1(\tilde\gamma)\log\frac{C}{\rho}+\delta\log\frac{C}{\varepsilon}+1\right),
\end{equation}
where $\tilde\gamma=\Phi(\gamma)$, $ U=\Phi(T_{\rho,\delta})$, $V=\Phi(T_{\varepsilon,0})$. 
\end{lem}
\begin{proof}
Let $\tilde\gamma$, $U$ and $V$ be as in the statement. 
We observe that, for every $X\in\R^3$, $Y\in\gamma$ it holds
\begin{equation}\label{stima_palla_1}
    \dist(X,\gamma)\le|X-Y|\le \ell|\Phi(X)-\Phi(Y)|,
\end{equation}
where $\ell>0$ is the Lipschitz constant associated to $\Phi^{-1}$ and $\Phi(Y)\in\tilde\gamma$. If now we take the infimum over all $Y\in\gamma$ in \eqref{stima_palla_1}, we get
\begin{equation}\label{stima_palla_2}
    \dist(X,\gamma)\le \ell\dist(\Phi(X),\tilde\gamma)\quad\text{for all $X\in\R^3$}.
\end{equation}
\begin{figure}
	\centering
	\includegraphics[width=6cm
	]{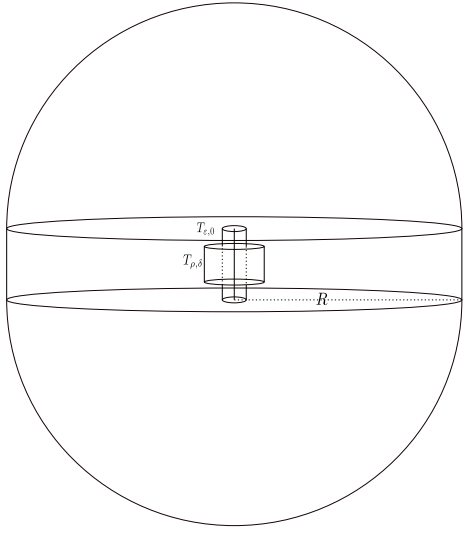}
	\caption{Large domain of diameter $R+\mathcal{H}^1(\gamma)$ containing $T_{\rho,\delta}$ and $T_{\varepsilon,0}$.}
	\label{bigcyl}
\end{figure}
By a change of variable and property \eqref{stima_palla_2}
we have that
\begin{align*}
    \int_{\Om\setminus( U\cup V)}\frac{1}{\dist^2(x,\tilde\gamma)}dx&\le C\int_{\Phi^{-1}(\Om)\setminus(T_{\rho,\delta}\cup T_{\varepsilon,0})}\frac{1}{\dist^2(\Phi(X),\tilde\gamma)}dX\\
   & \le \frac{C}{\ell^2}\int_{\Phi^{-1}(\Om)\setminus( T_{\rho,\delta}\cup T_{\varepsilon,0})}\frac{1}{\dist^2(X,\gamma)}dX.
\end{align*}
Next, we set $R=3\diam(\Phi^{-1}(\Om)))\le3C\diam(\Om)$ and assume $\dist(\gamma,\Phi^{-1}(\Om))\leq R/2$.
Up to a change of coordinates we can assume that $\gamma$ is centered at the origin, namely $$\gamma=\{(0,0)\}\times(-\mathcal{H}^1(\gamma)/2,\mathcal{H}^1(\gamma)/2),$$
and $T_{\rho,\delta}$ $T_{\varepsilon,0}$ are contained in the large domain of diameter $R+\mathcal{H}^1(\gamma)$ as represented in Figure \ref{bigcyl}.
Thus, setting $X=(X',X_3)$, we obtain
\begin{multline}\label{eq6}
     \int_{\Phi^{-1}(\Om)\setminus (T_{\rho,\delta}\cup T_{\varepsilon,0})}\frac{1}{\dist^2(X,\gamma)}dX\\\le |S_\delta|\int_{B'_R\setminus B'_{\rho}}\frac{1}{|X'|^2}dX'+2\delta\int_{B'_R\setminus B'_\varepsilon}\frac{1}{|X'|^2}dX'+\int_{B_R}\frac{1}{|X|^2}dX
   \\ \le 2\pi\mathcal{H}^1(\gamma)\log\frac{R}{\rho}+4\pi\delta\log\frac{R}{\varepsilon}+4\pi R.
\end{multline}
If instead $\dist(\gamma,\Phi^{-1}(\Omega))>R/2$ then we estimate as follows
\begin{equation}\label{eq7}
    \int_{\Phi^{-1}(\Omega)\setminus(T_{\rho,\delta}\cup T_{\varepsilon,0})}\frac{1}{\dist^2(X,\gamma)}dX\le |\Omega|\frac{4}{R^2}\le C.
    \end{equation}
\end{proof}
\begin{proof}[Proof of Proposition \ref{stimaL2}]

Let $\mu_\varepsilon\in{\mathcal{M}}_\mathcal{B}^{h_\varepsilon,\alpha_\varepsilon}(\bar\Om)$ be as in the statement 
\begin{equation*}
    \mu_\varepsilon=\sum_{i=1}^{M_\varepsilon}b_\varepsilon^i\otimes t_\varepsilon^i\mathcal{H}^1\res\gamma_\varepsilon^i,
\end{equation*}
and consider the sequence of deformed polyhedral extended measures $\tilde\mu_\varepsilon\in\mathcal{M}_\mathcal{B}(\R^3)$ given by Lemma \ref{extension_2} so that 
\begin{equation*}
    \tilde\mu_\varepsilon=\sum_{i=1}^{N_\varepsilon}b_\varepsilon^i\otimes t_\varepsilon^i\mathcal{H}^1\res\gamma_\varepsilon^i\quad\text{for $N_\varepsilon\ge M_\varepsilon$}.
\end{equation*} 
In particular by \eqref{ext1} and \eqref{ext2} we have
\begin{equation}\label{EQ}
    |\tilde\mu_\varepsilon|(\R^3)\le C,\quad
    \sum_{i=1}^{N_\varepsilon}|b_\varepsilon^i|^2\mathcal{H}^1(\gamma_\varepsilon^i)\le C,
\end{equation}
and therefore, from \eqref{eq0} and by the definition of ${\mathcal{M}}_\mathcal{B}^{h_\varepsilon,\alpha_\varepsilon}(\bar\Om)$ we infer that
\begin{equation}\label{eq3}
   M_\varepsilon h_\varepsilon\le \sum_{i=1}^{M_\varepsilon}|b_\varepsilon^i|h_\varepsilon
   \le \sum_{i=1}^{M_\varepsilon}|b_\varepsilon^i|^2 h_\varepsilon\le C.
\end{equation}
Analogously by \eqref{EQ} and recalling Remark \ref{ossex} (c), (d),
\begin{equation}\label{thelasttry}
N_\varepsilon h_\varepsilon\le\sum_{i=1}^{N_\varepsilon}|b_\varepsilon^i|h_\varepsilon\le \sum_{i=1}^{N_\varepsilon}|b_\varepsilon^i|^2h_\varepsilon\le C.
\end{equation}
Now let $\eta_\varepsilon\in L^{3/2}(\mathbb{R}^3,\mathbb{R}^{3\times3})$ be the distributional solution to \eqref{divsystem} with $\mu$ replaced by $\tilde\mu_\varepsilon$, 
that, thanks to Lemma \ref{lemma_deform}, satisfies 
\begin{equation}
    \label{eq2}
    |\eta_\varepsilon(x)|\le C\sum_{i=1}^{N_\varepsilon}\frac{|b_\varepsilon^i|}{\dist(x,\gamma_\varepsilon^i)}\quad\text{for $x\notin\supp\tilde\mu_\varepsilon$.}
\end{equation}
To construct $\tilde\eta_\eps$ (and consequently $\hat\eta_\eps$) we will modify $\eta_\varepsilon$ close to the support of $\mu_\varepsilon$.
 We start by fixing two parameters $\rho_\eps$ and $\delta_\eps$ and to denote  by $U^i_\varepsilon$, $V^i_\varepsilon$ (for $i=1,...,N_\varepsilon$), the cylinders defined as follows:
 \begin{itemize}
 	\item If $\gamma^i_\varepsilon$ is a straight segment then
 	\begin{equation*}
 	U^i_\varepsilon:=T_{\rho_\varepsilon,\delta_\varepsilon}(\gamma^i_\varepsilon)\quad V^i_\varepsilon:=T_{\varepsilon,0}(\gamma^i_\varepsilon);
 	\end{equation*}
 	\item If $\gamma^i_\varepsilon$ is obtained by reflecting some $\gamma^j_\varepsilon$ ($i\ne j$), then $U^i_\varepsilon$ and $V^i_\varepsilon$ are the reflections of $U^j_\varepsilon$ and $V^j_\varepsilon$ (see Figure \ref{cilstor}).
 \end{itemize}
	\begin{figure}
	\centering
	\includegraphics{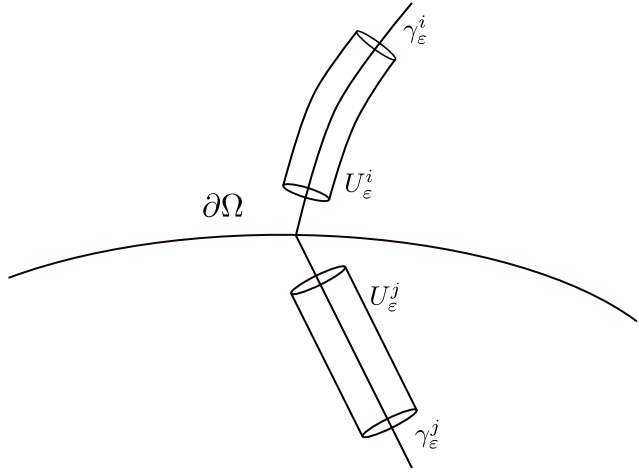}
	\caption{Deformed cylinder $U_\varepsilon^i$ obtained by reflection.}\label{cilstor}
\end{figure}
By the diluteness assumption, if 
\begin{equation}\label{rhodelta}
h_\eps \alpha_\eps\gg \alpha_\eps\delta_\eps\gg\rho_\eps,
\end{equation} then the cylinders $U^i_\varepsilon$ 
for $i\in\{1,...,M_\varepsilon\}$ are contained inside $\Om$, while all the  $U^i_\varepsilon$ with $i\in \{M_\varepsilon+1,\dots,N_\varepsilon\}$ lie outside $\Om$. Further 
  \begin{equation*}
      U_\varepsilon^i\cap U_\varepsilon^k=\emptyset,\quad U_\varepsilon^i\cap\gamma_\varepsilon^k=\emptyset
      \quad\text{for $i,k=1,...,N_\varepsilon$, $i\ne k$. }
  \end{equation*}
  For the convenience of the reader we divide the rest of the proof in 3 steps.\\
\noindent \textit{Step 1. Estimate of the $L^2$ norm of $\eta_\varepsilon$ in $D^\varepsilon:=\tilde\Omega_\varepsilon(\mu_\varepsilon)\setminus \cup_iU_\varepsilon^i$, with $\tilde\Omega_\varepsilon(\mu_\varepsilon):=\{x\in(\Om)_\eps:\dist(x,\supp\mu_\eps)>\eps\}$.}\\
Estimate \eqref{eq2} gives
 \begin{align}
     \int_{D^\varepsilon}|\eta_\varepsilon|^2dx&\le C\int_{D^\varepsilon}\left(\sum_{i=1}^{N_\varepsilon}\frac{|b_\varepsilon^i|}{\dist(x,\gamma_\varepsilon^i)}\right)^2dx\notag\\&\le CN_\varepsilon\sum_{i=1}^{N_\varepsilon}\int_{D_\varepsilon}\frac{|b_\varepsilon^i|^2}{\dist^2(x,\gamma_\varepsilon^i)}dx\notag\\
     &\le CN_\varepsilon\sum_{i=1}^{N_\varepsilon} \int_{(\Omega)_\eps\setminus ( U^i_\varepsilon\cup V^i_\varepsilon)}\frac{|b_\varepsilon^i|^2}{\dist^2(x,\gamma_\varepsilon^i)}dx.\label{eq4}
 \end{align}
We then use Lemma \ref{Stima_Palla} replacing $\Omega$ with $\Omega'\supset\supset\Omega$ and get
\begin{align}
\int_{(\Omega)_\eps\setminus ( U^i_\varepsilon\cup V^i_\varepsilon)}\frac{1}{\dist^2(x,\gamma_\varepsilon^i)}dx&\leq \int_{\Omega'\setminus ( U^i_\varepsilon\cup V^i_\varepsilon)}\frac{1}{\dist^2(x,\gamma_\varepsilon^i)}dx\notag\\
&\leq C\big(\mathcal{H}^1(\gamma_\varepsilon^i)\log\frac{C}{\rho_\varepsilon}+4\pi\delta_\varepsilon\log\frac{C}{\varepsilon}+1\big),
\label{eq5}
\end{align}
for some positive constant depending only on the domain $\Om'$.
Then from \eqref{thelasttry}, \eqref{eq4}, and \eqref{eq5}, it follows that
\begin{align}\label{eq8}
   \int_{D_\varepsilon}|\eta_\varepsilon|^2dx&\le CN_\varepsilon \sum_{i=1}^{N_\varepsilon}|b_\varepsilon^i|^2\left(\mathcal{H}^1(\gamma_\varepsilon^i)\log\frac{C}{\rho_\varepsilon}+\delta_\varepsilon\log\frac{C}{\varepsilon}+1\right)\notag\\
    &\le C\left(\frac{1}{h_\varepsilon}\log\frac{C}{\rho_\varepsilon}+\frac{\delta_\varepsilon}{h_\varepsilon^2}\log\frac{C}{\varepsilon}+\frac{1}{h_\varepsilon^2}\right).
\end{align}
\noindent\textit{Step 2. Construction of $\tilde\eta_\eps$.}\\
We observe that by \eqref{thelasttry} and \eqref{eq2} it follows 
\begin{equation*}\label{eeq1}
    |\eta_\varepsilon(x)|\le \frac{C}{ h_\varepsilon}\frac{|b^i_\varepsilon|}{\dist(x,\gamma^i_\varepsilon)}\quad\forall x\in U_\varepsilon^i,\,i\in\{1,...,M_\varepsilon\}
\end{equation*}
(recall \eqref{normal}).
Further it holds 
\begin{equation*}\label{eeq2}
    \curl\eta_\varepsilon= b_\varepsilon^i\otimes t_\varepsilon^i\mathcal{H}^1\res\gamma^i_\varepsilon\quad\text{in $U^i_\varepsilon$},\,i\in\{1,...,M_\varepsilon\}
\end{equation*}
and then applying Lemma \ref{estimatefromabove} with $c^*=C/h_\varepsilon$, $R=\rho_\eps$, and $r=\varepsilon$ we find for every $i\in\{1,...,M_\varepsilon\}$ a function $\tilde\eta^i_\varepsilon\in L^1(U^i_\varepsilon,\R^{3\times3})$ such that 
$$ \tilde\eta_\varepsilon^i=\eta_\varepsilon\quad\text{in a neighbourhood of $\partial U^i_\varepsilon$},$$
$$\curl\tilde\eta_\varepsilon^i=\curl\eta_\varepsilon\quad\text{in $U_\varepsilon^i$,}$$
and  
\begin{equation}\label{Eq3}
    \int_{U_\varepsilon^i\setminus V^i_\varepsilon}|\tilde\eta_\varepsilon^i|^2dx\le C|b_\varepsilon^i|^2\mathcal{H}^1(\gamma_\varepsilon^i)\log\frac{\rho_\varepsilon}{\varepsilon}+\frac{C}{h_\varepsilon^2}|b_\varepsilon^i|^2\left(\mathcal{H}^1(\gamma_\varepsilon^i)\left(\log\frac{\rho_\varepsilon}{\varepsilon}\right)^{1/2}+\frac{(\mathcal{H}^1(\gamma_\varepsilon^i))^3}{\rho_\varepsilon^2}\right).
\end{equation}

Now define $\tilde{\eta}_\varepsilon$ as follows
\begin{equation}
    \label{eq17}
  \tilde{\eta}_\varepsilon(x):=  \begin{cases}
 \tilde\eta_\varepsilon^i(x)&\text{if $x\in U_\varepsilon^i$ for $i=1,...,M_\varepsilon$}\\
  \eta_\varepsilon(x)&\text{otherwise in $\R^3$.}
    \end{cases}
\end{equation}
Then by estimate \eqref{eq8} in Step 1 and \eqref{Eq3}  we have that
\begin{multline*}
    \int_{\tilde\Omega_\varepsilon(\mu_\varepsilon)}|\tilde{\eta}_\varepsilon|^2dx=\int_{D_\varepsilon}|\eta_\varepsilon|^2dx+\sum_{i=1}^{M_\varepsilon}\int_{U^i_\varepsilon\setminus V_\varepsilon^i}|\tilde\eta_\varepsilon^i|^2dx\\
    \le C\left(\frac{1}{h_\varepsilon}\log\frac{C}{\rho_\varepsilon}+\frac{\delta_\varepsilon}{h_\varepsilon^2}\log\frac{C}{\varepsilon}+\frac{1}{h_\varepsilon^2}+\log\frac{\rho_\varepsilon}{\varepsilon}+\frac{1}{h_\varepsilon^2}\left(\log\frac{\rho_\varepsilon}{\varepsilon}\right)^{1/2}+\frac{1}{h_\varepsilon^2\rho_\varepsilon^2}\right)\\
    \le
       C\left(\frac{\delta_\varepsilon}{h_\varepsilon^2}\log\frac{C}{\varepsilon}+\log\frac{\rho_\eps}{\eps}+\frac{1}{h_\varepsilon^2}\left(\log\frac{\rho_\varepsilon}{\varepsilon}\right)^{1/2}+\frac{1}{h_\varepsilon^2\rho_\varepsilon^2}\right).
       \end{multline*}
Now choosing $\delta_\eps=h_\eps^2$ and $\rho_\eps=(\alpha_\eps h_\eps)^2$, which is compatible with \eqref{rhodelta}, we immediately obtain
\begin{equation}
    \label{eq17bis}
    \int_{\tilde\Omega_\varepsilon(\mu_\varepsilon)}|\tilde{\eta}_\varepsilon|^2dx\le c|\log\varepsilon|.
\end{equation}
\noindent\textit{Step 3. Construction of $\hat\eta_\eps$.}\\
Take the function $\eta_\eps*\varphi_\eps$, then it holds
\begin{equation*}
	\curl(\eta_\eps*\varphi_\eps)=\tilde\mu_\eps*\varphi_\eps \quad\text{in}\quad\R^3,
\end{equation*}
and
\begin{equation*}\label{eeq1moll}
|\eta_\varepsilon*\varphi_\eps(x)|\le \frac{C}{ h_\varepsilon}\frac{|b^i_\varepsilon|}{\dist(x,\gamma^i_\varepsilon)+\eps}\quad\forall x\in U_\varepsilon^i,\,i\in\{1,...,M_\varepsilon\}.
\end{equation*}
Thus we apply Lemma \ref{estimatefromabovemoll} with $c^*=C/h_\eps$, $R=\rho_\eps$, and $r=\varepsilon$ and get for every $i\in\{1,...,M_\varepsilon\}$ a function $\hat\eta^i_\varepsilon\in L^1(U^i_\varepsilon,\R^{3\times3})$ such that 
$$ \hat\eta_\varepsilon^i=\eta_\varepsilon*\varphi_\eps\quad\text{in a neighbourhood of $\partial U^i_\varepsilon$},$$
$$\curl\hat\eta_\varepsilon^i=\curl(\eta_\varepsilon*\varphi_\eps)\quad\text{in $U_\varepsilon^i$,}$$
and  
\begin{equation}\label{Eq3moll}
\int_{U_\varepsilon^i}|\hat\eta_\varepsilon^i|^2dx\le C|b_\varepsilon^i|^2\mathcal{H}^1(\gamma_\varepsilon^i)\log\frac{\rho_\varepsilon}{\varepsilon}+\frac{C}{h_\varepsilon^2}|b_\varepsilon^i|^2\left(\mathcal{H}^1(\gamma_\varepsilon^i)\left(\log\frac{\rho_\varepsilon}{\varepsilon}\right)^{1/2}+\frac{(\mathcal{H}^1(\gamma_\varepsilon^i))^3}{\rho_\varepsilon^2}\right).
\end{equation} 
Next we define $\hat\eta_\eps$ in the following way
\begin{equation}
\label{eq17moll}
\hat{\eta}_\varepsilon(x):=  \begin{cases}
\hat\eta_\varepsilon^i(x)&\text{if $x\in U_\varepsilon^i$ for $i=1,...,M_\varepsilon$}\\
\eta_\varepsilon*\varphi_\eps(x)&\text{otherwise in $\R^3$.}
\end{cases}
\end{equation}
In particular we obtain that 
\begin{equation}\label{eq18moll}
\int_\Omega|\hat{\eta}_\varepsilon|^2dx=   
\int_{\Omega\setminus\cup_{i=1}^{M_\eps}U^i_\eps}|\eta_\varepsilon*\varphi_\eps|^2dx+    \sum_{i=1}^{M_\eps}\int_{U^i_\eps}|\hat{\eta}_\varepsilon^i|^2dx.
\end{equation}
The estimate of second term on the right-hand side of \eqref{eq18moll} easily follows by \eqref{Eq3moll} arguing as in step 2. To estimate the first term we write
\begin{equation*}
\int_{\Omega\setminus\cup_{i=1}^{M_\eps}U^i_\eps}|\eta_\varepsilon*\varphi_\eps|^2dx = \int_{W_\eps^1}|\eta_\varepsilon*\varphi_\eps|^2dx+\int_{W_\eps^2}|\eta_\varepsilon*\varphi_\eps|^2dx,
\end{equation*}
where 
$$ W_\eps^1:=\Omega_{2\eps}(\mu_\eps)\setminus(\cup_{i=1}^{M_\eps}U^i_\eps),\quad W_\eps^2:=\big(\Omega\setminus(\cup_{i=1}^{M_\eps}U^i_\eps)\big)\setminus W_\eps^1. $$
By definition it holds
\begin{equation*}
\int_{W_\eps^1}|\eta_\varepsilon*\varphi_\eps|^2dx\le \int_{(W^1_\eps)_\eps}|\eta_\eps|^2dx\le \int_{\tilde{\Omega}_\varepsilon(\mu_\varepsilon)}|\tilde{\eta}_\varepsilon|^2dx\le C|\log\varepsilon|,
\end{equation*}
where the second inequality follows since $\eta_\eps=\tilde\eta_\eps$ in a neighbourhood of $\partial U^i_\eps$ of thickness greater than $\eps$.
For the remaining term we use that from \eqref{eq2}
\begin{equation*}
	|\eta_\eps*\varphi_\eps|(x)\le C \sum_{i=1}^{N_\eps}\frac{|b^i_\eps|}{\dist(x,\gamma^i_\eps)+\eps}\le \frac{C}{h_\eps} \sum_{i=1}^{N_\eps}\frac{1}{\dist(x,\gamma^i_\eps)+\eps},
\end{equation*}
and get
\begin{align*}
	\int_{W_\eps^2}|\eta_\varepsilon*\varphi_\eps|^2dx&\le \frac{C}{h_\eps^3}\sum_{i=1}^{N_\eps}\int_{W^2_\eps}\frac{1}{(\dist(x,\gamma_\eps^i)+\eps)^2}dx\\
	&\le \frac{C}{h_\eps^4}\int_{W^2_\eps}\frac{1}{(\dist(x,\gamma_\eps)+\eps)^2}dx\\
	&\le \frac{C}{h_\eps^4}\sum_{i=1}^{N_\eps}\int_{(\gamma^i_\eps)_{2\eps}}\frac{1}{(\dist(x,\gamma^i_\eps)+\eps)^2}dx\\
	&\le \frac{C}{h_\eps^4}\sum_{i=1}^{N_\eps}(\mathcal{H}^1(\gamma^i_\eps)+\eps)\le \frac{C}{h_\eps^4},
\end{align*}
where $\gamma_\eps=\cup_i\gamma_\eps^i$. By \eqref{diluteness2} we conclude the proof.

\end{proof}

\black

\section{Proof of Compactness and $\Gamma$-Limit}\label{sec4}
We finally pass to the proof of the main results, namely Theorems \ref{compactness} and \ref{gammalimit}.

\begin{proof}[Proof of Theorem \ref{compactness}]

\noindent \textit{Compactness of $\mu_j$.}
Let $(\mu_j,\beta_j)\in{\mathcal{M}}_\mathcal{B}^{h_{\varepsilon_j},\alpha_{\varepsilon_j}}(\bar\Omega)\times\mathcal{AS}_{\varepsilon_j}(\mu_j)$ be as in the statement. Since $\mu_j$ is dilute, we can write $\mu_j=\sum_ib_j^i\otimes t_j^i\mathcal{H}^1\res\gamma_j^i$ where $\gamma_j^i\subset\Omega$ satisfy the conditions of Definition \ref{def_dilute2}. We choose the parameters 
\begin{equation}
    \label{parameters}
    \rho_j:=(\alpha_{\varepsilon_j} h_{\varepsilon_j})^3\quad\delta_j:=(\alpha_{\varepsilon_j} h_{\varepsilon_j})^2
\end{equation}
and define, with a little abuse of notation, the cylinders
\begin{equation}\label{cylinders}
U^i_j:=T_{\rho_j,\delta_j}(\gamma^i_j)\quad V^i_j:=T_{\varepsilon_j,0}(\gamma^i_j)
\end{equation}
according to definition \eqref{cilindri_ext} given in Section \ref{sec3}.

It turns out that, by the choice of $\rho_j$ and $\delta_j$, the cylinders $U_j^i$ are pairwise disjoint and $U_j^i\cap\gamma_j^k=\emptyset$ for all $i\ne k$; 
therefore we have
\begin{equation}
    \mathcal{E}_{\varepsilon_j}(\mu_j,\beta_j)\ge\sum_i\int_{U_j^i\setminus V_j^i}W(\beta_j)dx.\label{comp4}
\end{equation}
Recalling estimate \eqref{consequence1} we find a constant $c>0$ such that \begin{equation}\label{**}
\frac{1}{|S_j^i|\varepsilon_j^2\log\frac{\rho_j}{\varepsilon_j}}\int _{U_j^i\setminus V_j^i}W(\beta_j)dx\ge \biggl(1-\frac{\rho_j}{|S_j^i|}\biggr)c|b_j^i|^2,
\end{equation}
where $|S_j^i|=\mathcal{H}^1(\gamma_j^i)-2\delta_j$.
This and \eqref{comp4} give
\begin{equation}\label{***}
\begin{split}
    \mathcal{F}_{\varepsilon_j}(\mu_j,\beta_j)&\ge c\left(1-\frac{\rho_j}{|S_j^i|}\right)\frac{\log(\rho_j/\varepsilon_j)}{|\log\varepsilon_j|}\sum_i|b_j^i|^2|S_j^i|\\&\ge c\left(1-\frac{\rho_j}{|S_j^i|}\right)(1-2h_{\varepsilon_j}\alpha_{\varepsilon_j}^2)\frac{\log(\rho_j/\varepsilon_j)}{|\log\varepsilon_j|}\sum_i|b_j^i|^2\mathcal{H}^1(\gamma_j^i).
    \end{split}
\end{equation}
By \eqref{diluteness1}, \eqref{diluteness2} and the definition of $\rho_\eps$ we get
\begin{equation}
   C\ge \sum_i|b_j^i|^2\mathcal{H}^1(\gamma_j^i),\label{comp5}
\end{equation}
and in particular, since $|b_j^i|\ge 1$, we get
\begin{equation}
   C\ge \sum_i|b_j^i|\mathcal{H}^1(\gamma_j^i)=|\mu_j|(\Omega).\label{comp6}
\end{equation}
By \eqref{comp6} and \cite[Theorem 2.5]{C.G.M.}, up to subsequences, we derive
\begin{equation*}
    \mu_{\varepsilon_j}\stackrel{*}{\rightharpoonup}\mu\in\mathcal{M}_\mathcal{B}(\Omega).
\end{equation*}

\noindent\textit{Compactness of $\beta_j$.} 
For every $j$ let $\Om_j\subset\subset\Om$ be as in the statement, namely
$$ \Om_j:=\Om\setminus\{x\in\Om:\dist(x,\partial\Om)\le\varepsilon_j\}.$$
Notice that, since $\Om$ is of class $C^2$, any $\Om_j$ can be obtained by a bi-Lipschitz transformation of $\Om$ with Lipschitz constant $L_j\le L$, moreover the characteristic function $\chi_{\Om_j}$ converges to $1$ in measure. 
Thanks to Proposition \ref{stimaL2} there exist a constant $c>0$, not depending on $\varepsilon_j$, and a sequence \[\hat{\eta}_j\in L^{3/2}(\mathbb{R}^3;\mathbb{R}^{3\times3})\cap L^2(\Om;\mathbb{R}^{3\times3}),\] 
such that 
\begin{equation}\label{comp7}
    \int_\Om|\hat\eta_j|^2dx\le c|\log\varepsilon_j|,
\end{equation}
and
\begin{equation}\label{stima-utile-j}
\curl\hat\eta_j=\mu_j*\varphi_{\varepsilon_j}\quad\text{in $\Om_j$.}
\end{equation} 
In particular
\begin{equation*}
    \curl(\beta_j-\varepsilon_j\hat\eta_j)=0\quad\text{in $\Om_j$.}
  \end{equation*}  
 Therefore there exists $u_j\in W^{1,2}(\Om_j;\R^3)$ such that
\begin{equation}\label{betadec}
    \beta_j=\varepsilon_j\hat{\eta}_j+\nabla u_j\quad\text{in $\Omega_j$.}
\end{equation}
Using the rigidity estimate on $\Omega_j$ we find a constant $C_j>0$ and a sequence $\{Q_j\}\subset SO(3)$ such that
\begin{equation*}
    \int_{\Omega_j}|\nabla u_j-Q_j|^2dx\le C_j\int_{\Omega_j}\dist^2(\nabla u_j,SO(3))dx.
\end{equation*}
Now Theorem 5 in \cite{F.J.M.} and the hypothesis on $\Omega_j$ imply that $C_j\le C$, for some constant $C>0$ independent of $j$.
This together with the growth conditions on $W$ gives
\begin{align*}
C\varepsilon_j^2|\log\varepsilon_j|&\ge\int_{\Omega}\dist^2(\beta_j,SO(3))dx\\&\ge \frac{1}{2}\int_{\Omega_j}\dist^2(\nabla u_j,SO(3))dx-\int_{\Omega_j}\varepsilon_j^2|\hat{\eta}_j|^2dx\\
& \ge \frac{1}{2C}\int_{\Omega_j}|\nabla u_j-Q_j|^2dx-\int_{\Omega_j}\varepsilon_j^2|\hat{\eta}_j|^2dx\\&\ge\frac{1}{4C}\int_{\Omega_j}|\beta_j-Q_j|^2dx-\frac{1}{2C}\int_{\Omega_j}\varepsilon_j^2|\hat{\eta}_j|^2dx,
\end{align*}
and by \eqref{comp7} we conclude
\begin{equation}\label{equazionediprima}
    \int_{\Omega_j}|\beta_j-Q_j|^2dx\le C\varepsilon_j^2|\log\varepsilon_j|.
\end{equation}
Hence, there exists $\beta\in L^2(\Omega,\mathbb{R}^{3\times3})$ with $\curl\beta=0$ and $Q\in SO(3)$, such that, up to subsequences, it holds
\begin{equation*}
    \frac{(Q_j)^T\beta_j-I}{\varepsilon_j\sqrt{|\log\varepsilon_j|}}\chi_{\Om_j}\rightharpoonup\beta\quad\text{in $L^2(\Omega,\mathbb{R}^{3\times3})$  and}\quad Q_j\to Q\in SO(3).
\end{equation*}

 \end{proof}

\begin{oss}
	The crucial point in order to obtain the compactness of the $\beta$'s is the decomposition \eqref{betadec} which is guaranteed by \eqref{stima-utile-j}. In the case in which we fix the extension measure in the definition of admissible configurations $\mathcal{AS}^*_{\varepsilon_j}(\mu_j)$ for the functionals $\mathcal{F}_{\varepsilon_j}(\mu_j,\cdot)$, see Remark \ref{compactness2},
	thanks to Proposition \ref{stimaL2} we obtain the decomposition in the whole of $\Omega$. Eventually we can proceed as above and obtain  \eqref{equazionediprima} in $\Om$.
 \end{oss}

\begin{prop}[Lower Bound]\label{lowerboundstep1}
For any sequence $\varepsilon_j\to0$ and for any  $(\mu_{j},\beta_{j})\in{\mathcal{M}}_\mathcal{B}^{h_{\varepsilon_j},\alpha_{\varepsilon_j}}(\bar\Omega)\times\mathcal{AS}_{\varepsilon_j}(\mu_j)$ converging to $(\mu,\beta,Q)\in\mathcal{M}_\mathcal{B}(\Omega)\times L^2(\Omega;\mathbb{R}^{3\times3})\times SO(3)$ in the sense of Definition \ref{def_conv} with $\curl\beta=0$, we have
\begin{equation}\label{tesiliminf}
    \mathcal{F}_0(\mu,\beta,Q)\le\liminf_{j\to\infty}\mathcal{F}_{\varepsilon_j}(\mu_j,\beta_j).
\end{equation}
\end{prop}
\begin{proof}
Let $(\mu_{j},\beta_{j})$ be a sequence with equibounded energy that converges to $(\mu,\beta,Q)$ with $\curl\beta=0$ as in the statement. Then by Definition \ref{def_conv} there exists a sequence $\{Q_j\}\subset SO(3)$ such that, up to subsequence, $Q_j\to Q$ and
\begin{equation}
    \label{lowbound0} 
    \mu_{j}\stackrel{*}{\rightharpoonup}\mu\quad\text{in}\quad \mathcal{M}_\mathcal{B}(\Omega),
\end{equation}
\begin{equation}
    \frac{Q_j^T\beta_{j}-I}{\varepsilon_j\sqrt{|\log\varepsilon_j|}}\chi_{\Om_j}\rightharpoonup\beta\quad\text{in}\quad L^2(\Omega;\mathbb{R}^{3\times3}),\label{lowbound1}
\end{equation}
where  $\Om_j:=\Om\setminus\{x\in\Om:\dist(x,\partial\Om)\le\varepsilon_j\}$.
Let $\rho_{j}>0$ be as in \eqref{parameters} and define the sets
\begin{equation}
    \Omega_j':=\{x\in\Omega:\dist(x,\supp\mu_j)\ge2\rho_{j}\}\quad\text{and}\quad\Omega_j'':=\Omega\setminus\Omega_j'.
\end{equation}
Then 
\begin{equation*}
    \mathcal{F}_{\varepsilon_j}(\mu_j,\beta_j)=\frac{1}{\varepsilon_j^2|\log\varepsilon_j|}\int_{\Omega_j'}W(\beta_j)dx+\frac{1}{\varepsilon_j^2|\log\varepsilon_j|}\int_{\Omega_j''}W(\beta_j)dx=:F_j'+F_j''.
\end{equation*}

\noindent \textit{Step 1: Lower bound for $F'_j$.} We can perform a Taylor expansion of $W$ near the identity as in \eqref{taylorexp} that yields
\begin{equation}\label{texp}
    W(I+F)\ge\frac{1}{2}\mathbb{C}F:F-\omega(|F|),
\end{equation}
with $\omega(t)/t^2\to0$ as $t\to0$. 
We then set
\begin{equation*}
    G_j:=\frac{Q_j^T\beta_j-I}{\varepsilon_j\sqrt{|\log\varepsilon_j|}},\quad \tilde G_j:=G_j\chi_{\Om_j},
\end{equation*}
and  
\begin{equation*}
    \chi_j:=\begin{cases}
    1&\text{if $|\tilde G_j|\le\varepsilon_j^{-1/2}$}\\
    0&\text{otherwise in $\Omega$}
    \end{cases},\quad
\tilde{\chi}_j:=\chi_j\cdot \chi_{\Omega_j'}.
\end{equation*}
Then using \eqref{texp} we have that
\begin{align*}
    F_j'=\frac{1}{\varepsilon_j^2|\log\varepsilon_j|}\int_{\Omega_j'}W(Q_j^T\beta_j)dx&\ge\frac{1}{\varepsilon_j^2|\log\varepsilon_j|}\int_{\Omega}W(I+\varepsilon_j\sqrt{|\log\varepsilon_j|}G_j)\tilde{\chi}_j\cdot\chi_{\Om_j}dx\\&
    \ge\int_{\Omega}\left(\frac{1}{2}\mathbb{C}\tilde{\chi}_j\tilde G_j:\tilde{\chi}_j\tilde G_j-\tilde{\chi}_j\cdot\chi_{\Om_j}\frac{\omega(\varepsilon_j\sqrt{|\log\varepsilon_j|}|G_j|)}{\varepsilon_j^2|\log\varepsilon_j|}\right)dx.
\end{align*}
Now \eqref{lowbound1} implies that $(\tilde G_j)$ is bounded in $L^2(\Omega;\mathbb{R}^{3\times3})$, then $\tilde{\chi}_j\to1$ in $\Om$ in measure and $\tilde{\chi}_j\tilde G_j\rightharpoonup\beta$ in $L^2(\Omega;\mathbb{R}^{3\times3})$. Therefore by lower semicontinuity it follows that
\begin{equation*}
    \liminf_{j\to\infty}\int_{\Omega}\frac{1}{2}\mathbb{C}\tilde{\chi}_j\tilde G_j:\tilde{\chi}_j\tilde G_jdx\ge\int_{\Omega}\frac{1}{2}\mathbb{C}\beta:\beta dx.
\end{equation*}
On the other hand we have that
\begin{equation*}
  |G_j|^2\chi_{\Om_j}\cdot  \tilde{\chi}_j\frac{\omega(\varepsilon_j\sqrt{|\log\varepsilon_j|}|G_j|)}{\varepsilon_j^2|\log\varepsilon_j||G_j|^2}=|\tilde G_j|^2\cdot\tilde{\chi}_j\frac{\omega(\varepsilon_j\sqrt{|\log\varepsilon_j|}|\tilde G_j|)}{\varepsilon_j^2|\log\varepsilon_j||\tilde G_j|^2}
\end{equation*}
is the product of a bounded sequence in $L^1(\Omega)$ and a sequence tending to zero in $L^\infty(\Omega)$ since $\varepsilon_j\sqrt{|\log\varepsilon_j|}|\tilde G_j|\le\varepsilon_j^{1/2}\sqrt{|\log\varepsilon_j|}$ whenever $\tilde{\chi}_j\ne0$.
Then 
\begin{equation*}
    \lim_{j\to\infty}\int_{\Omega}\tilde{\chi}_j\cdot\chi_{\Om_j}\frac{\omega(\varepsilon_j\sqrt{|\log\varepsilon_j}|G_j|)}{\varepsilon_j^2|\log\varepsilon_j|}dx=0,
\end{equation*}
which implies
\begin{equation}
    \label{term1}\liminf_{j\to\infty}F_j'\ge\int_{\Omega}\frac{1}{2}\mathbb{C}\beta:\beta dx.
\end{equation}

\noindent\textit{Step 2: Lower bound for $F''_j$.} Notice that by \eqref{lowbound1} there exists a constant $C>0$ such that 
\begin{equation}
    \label{lowbound2}   \int_{\Omega}|\beta_{j}-Q_j|^2\chi_{\Om_j}dx\le C\varepsilon_j^2|\log\varepsilon_j|.
\end{equation}
Without loss of generality we can assume that $|\mu|(\Omega)\ge\tilde{C}>0$
 (otherwise there is nothing to prove) and then by lower semicontinuity we have $|\mu_{j}|(\Omega)\ge\tilde{C}$. 

Moreover diluteness of $\mu_{j}$ corresponds to
\begin{equation*}
    \mu_{j}=\sum_ib_{j}^i\otimes t_{j}^i\mathcal{H}^1\res\gamma_{j}^i,
\end{equation*}
with $\gamma^i_j$ satisfying conditions of Definition \ref{def_dilute2}.
Consider the cylinders $U_j^i$, $V_j^i$ defined as in the proof of Theorem 
\ref{compactness}, thus, recalling \eqref{**}, we have that

\begin{multline*}
    \frac{1}{\varepsilon_j^2|\log\varepsilon_j|}\int_{\Omega_j''} W(\beta_{j})dx\ge \frac{1}{\varepsilon_j^2|\log\varepsilon_j|}\sum_i\int_{U_j^i\setminus V^i_j}W(\beta_j)dx\\
    \ge(1+ o(1)) C\sum_i|b_{j}^i|\mathcal{H}^1(\gamma_{j}^i)\ge(1+o(1))C|\mu_{j}|(\Omega)\ge (1+o(1))C,
\end{multline*}
where $o(1)\to0$ as $j\to\infty$.
Let $\nu>0$ and set $\tilde{\lambda}:=\nu/(1-\nu)$, then using  the previous estimates and \eqref{lowbound2} it holds 
\begin{align}
  F_j''&=\frac{1-\nu}{\varepsilon_j^2|\log\varepsilon_j|}\left(\int_{\Omega_j''} W(\beta_{j})dx+\tilde{\lambda}\int_{\Omega_j''} W(\beta_{j})dx\right)\notag\\
&\ge \frac{1-\nu}{\varepsilon_j^2|\log\varepsilon_j|}\left(\int_{\Omega_j''} W(\beta_{j})dx+\tilde{\lambda} (1+o(1))C\varepsilon_j^2|\log\varepsilon_j|\right)\notag\\
&\ge \frac{1-\nu}{\varepsilon_j^2|\log\varepsilon_j|}\left(\int_{\Omega_j''} W(\beta_{j})dx+\tilde{\lambda} C\int_{\Omega_j''}|\beta_{j}-Q_j|^2\chi_{\Om_j}dx\right).\label{*_*}
\end{align}
 By \eqref{nonlinearcell} we get for any $\lambda>0$
\begin{align*}
    \int_{\Omega_j''\cap\Om_j} W(\beta_{j})+\lambda|\beta_{j}-Q_j|^2dx&\ge\sum_{i\in I(j)}\int_{T_j^i\setminus V^i_j}W(\beta_{j})+\lambda|\beta_{j}-Q_j|^2dx\\
&\ge \sum_{i\in I(j)}|S_{j}^i|\varepsilon_j^2\log\frac{\rho_{j}}{\varepsilon_j}\Psi_\lambda^{nl}(Q_j,b_j^i,t_j^i,|S_{j}^i|,\varepsilon_j,\rho_{j})\\
& =\sum_{i\in I(j)}|S_{j}^i|\varepsilon_j^2\log\frac{\rho_{j}}{\varepsilon_j}\Psi_\lambda^{nl}(I,Q_j^Tb_j^i,t_j^i,|S_{j}^i|,\varepsilon_j,\rho_{j}),\end{align*}
    where
    $$I(j):=\{i:\gamma_j^i\subset\Om_j\}.$$
Using this into \eqref{*_*} with $\lambda:=\tilde{\lambda}C$, by \eqref{parameters}, we infer 
\begin{equation}
F_j''  \ge (1-\nu)(1-2h_{\varepsilon_j}\alpha_{\varepsilon_j}^2)\frac{\log(\rho_{j}/\varepsilon_j)}{|\log\varepsilon_j|}\sum_{i\in I(j)}\mathcal{H}^1(\gamma_j^i)\Psi_\lambda^{nl}(I,Q_j^Tb_j^i,t_j^i,|S_{j}^i|,\varepsilon_j,\rho_{j}).\label{lowbound3}
\end{equation}

To conclude the proof fix $M>1$, $K>0$, and denote
\begin{equation*}
    I^1_j:=\{i\in I(j):|b_j^i|\le K\}\quad\text{and}\quad I^2_j:=\{i\in I(j):|b_j^i|>K\}.
\end{equation*}
For sufficiently big $j$ we have $|S_{j}^i|\ge \frac12h_{\eps_j}\ge M\rho_{j}$, for all $i\in I(j)$; thus if $\omega_{M,K}$ is the function given by Lemma \ref{lemma4}, from estimate \eqref{unif1} it follows that
\begin{align}
    \sum_{i\in I_j^1}\mathcal{H}^1(\gamma_j^i)\Psi_\lambda^{nl}&(I,Q_j^Tb_j^i,t_j^i,|S_{j}^i|,\varepsilon_j,\rho_{j})\ge\notag\\
   & \ge\sum_{i\in I_j^1}\mathcal{H}^1(\gamma_j^i)\left(\Psi_0(Q_j^Tb_j^i,t_j^i)-\frac{cK^2}{M}-\omega_{M,K}\left(\frac{\varepsilon_j}{\rho_{j}}\right)\right)\notag\\
   &\ge
   \sum_{i\in I_j^1}\mathcal{H}^1(\gamma_j^i)\Psi_0(Q_j^Tb_j^i,t_j^i)-C\left(\frac{cK^2}{M}+\omega_{M,K}\left(\frac{\varepsilon_j}{\rho_{j}}\right)\right).\label{lowbound4}
  \end{align}  
  Moreover from \eqref{unif2}, using that $|b_j^i|\geq K$  if $i\in I_j^2$, we get
 \begin{align}
         \sum_{i\in I_j^2}\mathcal{H}^1(\gamma_j^i)\Psi_\lambda^{nl}(I,Q_j^Tb_j^i,t_j^i,|S_{j}^i|,\varepsilon_j,\rho_{j})
         &\ge\left(1-\frac{\rho_j}{|S^i_{j}|}\right)c_*\sum_{i\in I_j^2}\mathcal{H}^1(\gamma_j^i)|b_j^i|^2\notag\\
        & \ge \left(1-\frac{\rho_j}{|S^i_{j}|}\right)\sum_{i\in I_j^2}\mathcal{H}^1(\gamma_j^i)c_*K|b_j^i|.
         \label{lowbound5}
         \end{align}  
        If now we choose $K$ such that $c_*K\ge\tilde{c}_1$, where $\tilde{c}_1$ satisfies
 \eqref{selfen4}, and recalling that $\Psi_0(Q_j^Tb_j^i,t_j^i)\ge\tilde{\Psi}_0(Q_j^Tb_j^i,t_j^i)$ we have from \eqref{lowbound4} and \eqref{lowbound5}
    \begin{multline*}
       \sum_{i\in I(j)}\mathcal{H}^1(\gamma_j^i)\Psi_\lambda^{nl}(I,Q_j^Tb_j^i,t_j^i,|S_{j}^i|,\varepsilon_j,\rho_{j})\ge\\\ge \left(1-\frac{\rho_j}{|S^i_{j}|}\right)\sum_{i\in I(j)}\mathcal{H}^1(\gamma_j^i)\tilde{\Psi}_0(Q_j^Tb_j^i,t_j^i)-C\left(\frac{cK^2}{M}+\omega_{M,K}\left(\frac{\varepsilon_j}{\rho_{j}}\right)\right),
    \end{multline*}
    thus for every $\tilde\Om\subset\subset\Om$
    \begin{align}
        \liminf_{j\to\infty}F''_j&\ge\liminf_{j\to\infty}\frac{1}{\varepsilon_j^2|\log\varepsilon_j|}\int_{\Omega''_j\cap\Om_j} W(\beta_j)dx\notag\\
        &\ge(1-\nu)\liminf_{j\to\infty}\sum_{i\in I(j)}\mathcal{H}^1(\gamma_j^i)\tilde{\Psi}_0(Q_j^Tb_j^i,t_j^i)-\frac{CK^2}{M}\notag\\
        &\ge(1-\nu)\int_{\gamma\cap \tilde \Omega}\tilde{\Psi}_0(Q^Tb(x),t(x))d\mathcal{H}^1(x) -\frac{CK^2}{M}\label{term2}
    \end{align}
    where the last inequality follows by the $\mathcal{H}^1$-ellipticity of $\tilde\Psi_0$ and then by the lower semicontinuity of the associated line tension energy,  from the weak$*$ convergence of $Q_j^T\mu_j$ to $Q^T\mu$, with $\mu=b\otimes t\mathcal{H}^1\res\gamma\in\mathcal{M}_\mathcal{B}(\Omega)$. 
    
    \noindent\textit{Conclusions.} From Step 1 and Step 2, specifically from estimates
     \eqref{term1} and \eqref{term2}, taking the limit as  $M\to\infty$  and then as $\nu\to0$ we obtain 
    \begin{equation*}
        \liminf_{j\to\infty}\mathcal{F}_{\varepsilon_j}(\mu_j,\beta_j)\ge\int_{\Omega}\frac{1}{2}\mathbb{C}\beta:\beta\ dx+\int_{\gamma\cap\tilde\Omega}\tilde{\Psi}_0(Q^Tb,t)d\mathcal{H}^1
    \end{equation*}
   Since $\tilde\Om\subset\subset\Omega$ is arbitrary this concludes the proof.
    \end{proof}
   For the upper bound we split the proof into three propositions exploiting the fact that all measures in $\mathcal{M}_\mathcal{B}(\Omega)$ can be approximated by dilute measures. 
   \begin{prop}\label{preliminary_recovery}
Let $\nu\in\mathcal{M}_\mathcal{B}(\R^3)$ be polyhedral, and fix $r>0$. Then for any $\varepsilon>0$, there exists $\theta_\varepsilon^\nu\in L^1(\Om;\R^{3\times3})$ such that $\curl\theta_\varepsilon^\nu=\nu$ in $\Om$ and 
\begin{equation}\label{gamma-sup-linear}
    \limsup_{\varepsilon\to0}\frac{1}{|\log\varepsilon|}\int_{\Om_\eps(\nu)}\mathbb{C}\theta_\varepsilon^\nu:\theta_\varepsilon^\nu dx\le \int_{\gamma\cap(\Omega)_r}\Psi_0(b,t)d\mathcal{H}^1,
\end{equation}
where  $(\Om)_r:=\{x\in\R^3:\dist(x,\Om)<r\}$ and the sequence $\theta_\eps^\nu$ satisfies
\begin{equation}\label{recovery-linear}
|\theta_\eps^\nu(x)|\leq \frac{C}{\dist(x,\gamma)} \qquad \forall x\in \R^3.
\end{equation}
Moreover the sequence
$\hat\theta_\eps^\nu=\theta_\varepsilon^\nu *\varphi_\varepsilon\in L^2(\Om;\R^{3\times3})$, with $\curl\hat\theta_\varepsilon^\nu=\nu*\varphi_\varepsilon$ in $\Om$ satisfies

\begin{equation}\label{gamma-sup-linear-moll}
\limsup_{\varepsilon\to0}\frac{1}{|\log\varepsilon|}\int_\Om\mathbb{C}\hat\theta_\varepsilon^\nu:\hat\theta_\varepsilon^\nu dx\le \int_{\gamma\cap(\Omega)_r}\Psi_0(b,t)d\mathcal{H}^1
\end{equation}
and 
\begin{equation}\label{recovery-linear-moll}
|\hat\theta_\eps^\nu(x)|\leq \frac{C}{\dist(x,\gamma)+\eps}\qquad \forall x\in \R^3.
\end{equation}
\end{prop}
\begin{proof}
The proof of this statement is given in \cite[Proposition 6.7]{C.G.O.}, where the $\Gamma$-limsup estimate for the linear problem is obtained. The explicit estimates \eqref{recovery-linear} and \eqref{recovery-linear-moll} can be deduced from the construction of the recovery sequence. Indeed in \cite{C.G.O.} the latter is obtained essentially  by glueing the solution in the whole space given in \eqref{divsystem} with $\mu=\nu$ together with the cell problem solution for each single segment in the support of $\nu$. This is rigorously done by using Lemma \ref{estimatefromabove} (Lemma 5.10 and 5.11 in \cite{C.G.O.}).
 In particular this gives the estimate \eqref{recovery-linear} as a combination of \eqref{est_2} and \eqref{eul3}. 
 Finally \eqref{recovery-linear-moll} can be easily obtained by \eqref{recovery-linear}.
\end{proof}

   \begin{prop}
    \label{upperboundstep1}
 Let $r>0$ and  $(\mu,\beta,Q)\in\mathcal{M}_{\mathcal{B}}(\R^3)\times L^\infty(\Omega;\mathbb{R}^{3\times3})\times SO(3)$ with $\mu=\sum_ib^i\otimes t^i\mathcal{H}^1\res\gamma^i$ polyhedral and $\curl\beta=0$.\\
  Then setting $\beta_\eps=Q+\eps\sqrt{|\log\eps|}\beta+\eps\hat\theta_\eps^\nu$ with $\nu=Q^T\mu=\sum_iQ^Tb^i\otimes t^i\mathcal{H}^1\res\gamma^i$ and $\hat\theta_\varepsilon^\nu$ given by Proposition \ref{preliminary_recovery}, we have
  $\beta_\varepsilon\in\mathcal{AS}_\varepsilon(\mu\res\Om)$ 
\begin{equation}\label{limsup}
  \frac{Q^T\beta_\eps-I}{\eps\sqrt{|\log\eps|}}=\beta+\frac{\hat\theta_\varepsilon^\nu}{\sqrt{|\log\varepsilon|}}\weak \beta\qquad \hbox{in}\quad L^2(\Omega;\R^{3\times 3}),
\end{equation}

 and
    \begin{equation*}
        \limsup_{\varepsilon\to0}\mathcal{F}_\varepsilon(\mu,\beta_\varepsilon)\le \int_{\Omega}\frac{1}{2}\mathbb{C}\beta:\beta\ dx+\int_{\gamma\cap(\Omega)_r}\Psi_0(Q^Tb,t)d\mathcal{H}^1.
    \end{equation*}\end{prop}
\begin{proof}
Let $(\mu,\beta,Q)$ and $\beta_\eps$ be as in the statement. 
Clearly $\beta_\varepsilon\in\mathcal{AS}_\varepsilon(\mu\res\Om)$, furthermore it holds
\begin{equation*}
    \frac{Q^T\beta_\varepsilon-I}{\varepsilon\sqrt{|\log\varepsilon|}}=\beta+\frac{\hat\theta_\varepsilon^\nu}{\sqrt{|\log\varepsilon|}}\rightharpoonup\beta\quad\text{in $L^2(\Omega;\mathbb{R}^{3\times3})$}.
\end{equation*}
Indeed from \eqref{recovery-linear-moll} we obtain that $\hat\theta_\varepsilon^\nu/\sqrt{|\log\varepsilon|}$ is bounded in $L^2(\Omega;\mathbb{R}^{3\times3})$ and  converges to zero strongly in $L^1(\Omega;\mathbb{R}^{3\times3})$.
Then $(\mu\res\Om,\beta_\varepsilon)$ converges to $(\mu\res\Om,\beta,Q)$ in the sense of Definition \ref{def_conv} and also satisfies \eqref{limsup}.\\
We define $\Om_{\eps^\alpha}(\mu):=\{x\in\Om:\dist(x,\supp\mu)>\eps^\alpha\}$ and $(\gamma)_{\eps^\alpha}:=\Om\setminus\Om_{\eps^\alpha}(\mu)$  for $\alpha\in(0,1)$.
	Then using the frame indifference and the Taylor expansion we get
\begin{equation*}
\begin{split}
&\frac{1}{\varepsilon^2|\log\varepsilon|}\int_{\Om_{\eps^\alpha}(\mu)} W(I+\varepsilon\sqrt{|\log\varepsilon|}\beta+\varepsilon\hat\theta_\varepsilon^\nu)dx
\\
&=\int_{\Om_{\eps^\alpha}(\mu)}\frac{1}{2}\mathbb{C}\beta:\beta dx+\frac{1}{|\log\varepsilon|}\int_{\Om_{\eps^\alpha}(\mu)}\frac{1}{2}\mathbb{C}\hat\theta_\varepsilon^\nu:\hat\theta_\varepsilon^\nu dx    +\frac{1}{\sqrt{|\log\varepsilon|}}\int_{\Om_{\eps^\alpha}(\mu)}\mathbb{C}\beta:\hat\theta_\varepsilon^\nu dx\\
&\qquad+ \int_{\Om_{\eps^\alpha}(\mu)}\frac{\sigma(\varepsilon\sqrt{|\log\varepsilon|}\beta+\varepsilon\hat\theta_\varepsilon^\nu)}{\varepsilon^2|\log\varepsilon|}dx,
\end{split}
\end{equation*}
where $\sigma(F)/|F|^2\to0$ as $|F|\to0$.
By Proposition \ref{preliminary_recovery} we deduce
\begin{equation}\label{alpha-limit}
    \limsup_{\varepsilon\to0}\frac{1}{|\log\varepsilon|}\int_\Omega\frac{1}{2}\mathbb{C}\hat\theta_\varepsilon^\nu:\hat\theta_\varepsilon^\nu dx \le\int_{\gamma\cap \Omega_r}\Psi_0(Q^Tb,t)d\mathcal{H}^1.
\end{equation}
Recalling that $\beta\in L^\infty(\Omega;\mathbb{R}^{3\times3})$ and, from \eqref{recovery-linear-moll}, $\hat\theta_\varepsilon^\nu/\sqrt{|\log\varepsilon|}$ converges to zero strongly in $L^1(\Omega;\mathbb{R}^{3\times3})$ we have
\begin{equation*}
    \lim_{\varepsilon\to0}\frac{1}{\sqrt{|\log\varepsilon|}}\int_{\Om_{\eps^\alpha}(\mu)}\mathbb{C}\beta:\hat\theta_\varepsilon^\nu dx=0.
\end{equation*}
Finally, setting $\omega(t):=\sup_{|F|\le t}|\sigma(F)|$ we find
\begin{multline*}
    \lim_{\varepsilon\to0}\left|     \int_{\Om_{\eps^\alpha}(\mu)}\frac{\sigma(\varepsilon\sqrt{|\log\varepsilon|}\beta+\varepsilon\hat\theta_\varepsilon^\nu)}{\varepsilon^2|\log\varepsilon|}dx\right|\\
    \le\lim_{\varepsilon\to0}\int_\Omega\chi_{\Om_{\eps^\alpha}(\mu)}\frac{\omega(|\varepsilon\sqrt{|\log\varepsilon|}\beta+\varepsilon\hat\theta_\varepsilon^\nu|)}{|\varepsilon\sqrt{|\log\varepsilon|}\beta+\varepsilon\hat\theta_\varepsilon^\nu|^2}\cdot \frac{|\varepsilon\sqrt{|\log\varepsilon|}\beta+\varepsilon\hat\theta_\varepsilon^\nu|^2}{\varepsilon^2|\log\varepsilon|}dx=0.
\end{multline*}
Indeed by \eqref{recovery-linear-moll}, in $\Omega_{\eps^\alpha}(\mu)$ we have $|\hat\theta_\varepsilon^\nu|\leq C \eps^{-\alpha}$ and then  the integrand
\begin{equation*}
    \frac{\omega(|\varepsilon\sqrt{|\log\varepsilon|}\beta+\varepsilon\hat\theta_\varepsilon^\nu|)}{|\varepsilon\sqrt{|\log\varepsilon|}\beta+\varepsilon\hat\theta_\varepsilon^\nu|^2}\cdot \frac{|\varepsilon\sqrt{|\log\varepsilon|}\beta+\varepsilon\hat\theta_\varepsilon^\nu|^2}
    {\varepsilon^2|\log\varepsilon|}
\end{equation*}
is the product of a sequence converging to zero in $L^\infty(\Omega)$ and a bounded sequence in $L^1(\Omega)$.

Finally it remains to estimate the energy in $(\gamma)_{\eps^{\alpha}}$. Using the frame indifference and the estimate from above for $W$ we get
\begin{align*}
\frac{1}{\eps^2|\log\eps|}\int_{(\gamma)_{\eps^{\alpha}}}W(\beta_\eps)dx&\le 	\frac{1}{\eps^2|\log\eps|}\int_{(\gamma)_{\eps^{\alpha}}}|\eps\sqrt{|\log\eps|}\beta+\eps\hat\theta_\eps^\nu|^2dx\\
&\le 2\int_{(\gamma)_{\eps^{\alpha}}}|\beta|^2dx+\frac{2}{|\log\eps|}\int_{(\gamma)_{\eps^{\alpha}}}|\hat\theta_\eps^\nu|^2dx,
\end{align*}
where the first term of the right hand side tends to zero as $\eps\to 0$, while
from \eqref{recovery-linear-moll} we have
$$
\lim_{\eps\to 0}\frac{1}{|\log\eps|}\int_{(\gamma)_{\eps^{\alpha}}}|\hat\theta_\eps^\nu|^2dx\le C(1-\alpha).
$$
Thus we conclude that
\begin{align*}
 \limsup_{\varepsilon\to0}\mathcal{F}_\varepsilon(\mu,\beta_\varepsilon)&\leq  \int_{\Omega}\frac{1}{2}\mathbb{C}\beta:\beta dx + \limsup_{\varepsilon\to0}\frac{1}{|\log\varepsilon|}\int_\Omega\frac{1}{2}\mathbb{C}\hat\theta_\varepsilon^\nu:\hat\theta_\varepsilon^\nu dx\\&\,\,\,\,\, +\frac{2}{|\log\eps|}\int_{(\gamma)_{\eps^{\alpha}}}|\hat\theta_\eps^\nu|^2dx\\
 &\le \int_{\Omega}\frac{1}{2}\mathbb{C}\beta:\beta dx+\int_{\gamma\cap (\Omega)_r}\Psi_0(Q^Tb,t)d\mathcal{H}^1+ C(1-\alpha).
\end{align*}
which concludes the proof taking the limit as $\alpha\to 1$.\end{proof}
\begin{prop}[Upper Bound] Let $(\mu,\beta,Q)\in\mathcal{M}_\mathcal{B}(\Omega)\times L^2(\Omega;\mathbb{R}^{3\times3})\times SO(3)$ with $\curl\beta=0$. Then, for every sequence $\varepsilon_k\to0$ there exists a sequence $(\tilde\mu_k,\beta_k)\in\mathcal{M}_\mathcal{B}(\mathbb{R}^3)\times\mathcal{AS}_{\varepsilon_k}(\tilde\mu_k\res\Omega)$ such that $\tilde\mu_k\res\Omega$ is $(h_{\varepsilon_k},\alpha_{\varepsilon_k})$-dilute in $\bar\Om$, $\tilde\mu_k\weakstar\mu$ in $\Omega$,
	$$
	\frac{Q^T\beta_k-I}{\eps_k\sqrt{|\log\eps_k|}}\weak \beta\quad \hbox{in}\quad L^2(\Omega;\R^{3\times 3}),
	$$
 and
\begin{equation*}
    \limsup_{k\to\infty}\mathcal{F}_{\eps_k}(\tilde\mu_k\res\Omega,\beta_k)\le\mathcal{F}_0(\mu,\beta,Q).
\end{equation*}
\label{upperboundstep2}
\end{prop}
\begin{proof}
	By a standard density argument we can assume that $\beta\in L^\infty(\Omega;\mathbb{R}^{3\times3})$.
The proof is based on a diagonal argument using Proposition \ref{upperboundstep1} and it is analogous to the linear case (\cite[Proposition 6.8]{C.G.O.}). \\
Let $Q^T\mu=Q^Tb\otimes t\mathcal{H}^1\res\gamma$,
by \cite[Theorem 3.1]{C.G.M.} 
\begin{equation*}
    \int_\gamma\tilde{\Psi}_0(Q^Tb,t)d\mathcal H^1,
\end{equation*}
with $\tilde{\Psi}_0$ defined in \eqref{relaxedenergy}, is the lower semicontinuous envelope of
\begin{equation*}
    \int_\gamma{\Psi}_0(Q^Tb,t)d\mathcal H^1.
\end{equation*}
Then, we can find  a sequence $\nu_j=b_j\otimes t_j\mathcal{H}^1\res\gamma_j\in\mathcal{M}_\mathcal{B}(\Omega)$ 
converging weak$^*$ to $\mu$ such that
\begin{equation}
\limsup_{j\to\infty}\int_{\gamma_j}\Psi_0(Q^T b_j,t_j)d\mathcal H^1\leq \int_\gamma \tilde{\Psi}_0(Q^Tb,t)d\mathcal H^1 .\label{ultima}
\end{equation}

Now we denote by $\mathcal{F}$ the functional
\begin{equation}\label{unrelaxed}
\mathcal{F}(\mu,\beta,Q, r):=\int_{\Omega}\frac{1}{2}\mathbb{C}\beta:\beta dx+\int_{\gamma\cap(\Omega)_r}\Psi_0(Q^Tb,t)d\mathcal{H}^1,
\end{equation}
for $r>0$.
Then for all $j$ we apply  \cite[Lemma 6.4]{C.G.O.}
and find a polyhedral measure  $\mu_j:=b_j\otimes t_j\mathcal{H}^1\res\gamma_j\in\mathcal{M}_\mathcal{B}(\R^3)$ such that  
\begin{equation}\label{limiteenergie}
\mathcal{F}(\mu_j,\beta,Q, \textstyle{\frac1j})\le (1+c\textstyle{\frac1j}) \mathcal{F}(\nu_j,\beta,Q,0)+C\textstyle{\frac1j},
\end{equation} 
and $\mu_j$ is close to $\nu_j$ in the following sense: there exists a bi-Lipschitz map $f^j:\R^3\to\R^3$, with 
\begin{equation}\label{approx-bilip-2}
|f^j(x)-x|+|Df^j(x)-Id|<\frac1j\qquad \forall x\in \R^3,
\end{equation}
such that
\begin{equation}\label{approx-bilip}
|\mu_j-f^j_\sharp\nu_j|(\Omega)<\frac{1}{j}.
\end{equation}
In particular $\mu_j\weakstar \mu$ in $\Omega$ as $j\to\infty$.
%
Furthermore since the restriction $\mu_j\res\Omega$ is polyhedral it is not restrictive to assume that the segments of the support of $\mu_j$ intersect the boundary of $\Omega$ with an angle at most $\alpha_j>0$ (otherwise a small modification of the support of $\mu_j$ for segments that are tangent to $\partial\Omega$ will reduce to the latter case with arbitrary small errors in the line tension energy and therefore in \eqref{limiteenergie}). Thus $\mu_j$ is $(h_\eps,\alpha_\eps)$-dilute in $\bar\Om$ according to Definition \ref{def_dilute2}, for sufficiently small $\eps$. 
From Proposition \ref{upperboundstep1} applied to $(\mu_j,\beta, Q)$, for every $j$ there is a sequence $\beta_k^j\in\mathcal{AS}_{\varepsilon_k}(\mu_j\res\Omega)$ that satisfies 
\begin{equation}\label{l2conv}
\frac{Q^T\beta_k^j-I}{\eps_k\sqrt{|\log\eps_k|}}=\beta+\frac{\hat\theta_{\varepsilon_k}^{\mu_j}}{\sqrt{|\log{\varepsilon_k}|}}\weak \beta\qquad \hbox{in}\ L^2(\Omega;\R^{3\times 3}),
\end{equation}
and
\begin{equation}
    \limsup_{k\to\infty}\mathcal{F}_{\varepsilon_k}(\mu_j\res\Om,\beta_k^j)\le
    \mathcal{F}(\mu_j,\beta,Q, \textstyle{\frac1j}).
\end{equation} 
The function $\hat\theta_{\eps_k}^{\mu_j}$ is given by Proposition \ref{preliminary_recovery} and by \eqref{gamma-sup-linear-moll} satisfies
\begin{equation*}
\left\|\frac{\hat\theta_{\eps_k}^{\mu_j}}{\sqrt{|\log\eps|}}\right\|_{L^2(\Om;\R^{3\times3})}\le M.
\end{equation*}
By \eqref{ultima} we finally obtain
\begin{equation}\label{limsupp}
\limsup_{j\to\infty} \limsup_{k\to\infty}\mathcal{F}_{\varepsilon_k}(\mu_j\res\Om,\beta_k^j)\le \mathcal{F}_0(\mu,\beta,Q).
\end{equation}

In order to construct a diagonal sequence which satisfies the thesis, we follow the same idea of \cite{C.G.O.}, and we notice that the following properties are satisfied for $k$ large enough:

\begin{enumerate}[(1)]
	\item The measures $\mu_j$ are $(h_{\eps_k},\alpha_{\eps_k})$-dilute in $\bar\Om$;
	\item $\mathcal{F}_{\varepsilon_k}(\mu_j\res\Om,\beta_k^j)\le\mathcal{F}(\mu_j,\beta,Q, \textstyle{\frac1j})+\frac1j$;
	\item $d(\hat\theta_{\eps_k}^{\mu_j}/\sqrt{|\log\eps_k|},0)\le\frac1j$ where $d$ denotes the distance that metrizes the weak convergence in $L^2\cap\{f:\|f\|_{L^2(\Om;\R^{3\times3})}\le M\}$.
\end{enumerate}

For every $j$ we define an increasing sequence of indices $m(j)$ as follows

\begin{equation}
m(j):=\min\big\{m\geq m(j-1)\,:\ \mu_j\quad \text{satisfies (1)-(3)}\quad\forall k\ge m\big\}.
\label{diagonalseq}
\end{equation}
Now for every $k>0$ we define $\tilde\mu_k:=\mu_j$ and $\beta_k:=\beta_k^j$  if $k\in [m(j),m(j+1) )\cap\N$. By \eqref{approx-bilip} we have
$$
|\tilde\mu_k-f^j_\sharp\nu_j|(\Omega)<\frac{1}{j}\quad \forall k\in [m(j),m(j+1) ).
$$
and then from \eqref{approx-bilip-2} and the fact that $\nu_j$ weak$^*$ converges to $\mu$, we  conclude that $\tilde\mu_k$ weak$^*$ converges to $\mu$. In addition by (3) and (2) we also have
$$
\frac{Q^T\beta_k-I}{\eps_k\sqrt{|\log\eps_k|}}\weak\beta\qquad \hbox{in}\ L^2(\Omega;\R^{3\times 3}),
$$
\begin{equation*}
 \limsup_{k\to\infty}\mathcal{F}_{\varepsilon_k}(\tilde\mu_k\res\Om,\beta_k)\le \mathcal{F}_0(\mu,\beta,Q),
\end{equation*}
and the proof is concluded.

\end{proof}
\begin{proof}[Proof of Theorem \ref{gammalimit}]
The thesis is a direct consequence of Propositions \ref{lowerboundstep1} and \ref{upperboundstep2}.
\end{proof}

\section*{Acknowledgments}
The present paper  benefits from the support of the GNAMPA (Gruppo Nazionale per l'Analisi Matematica, 
la Probabilit\`a e le loro Applicazioni) of INdAM 
(Istituto Nazionale di Alta Matematica). 


\addcontentsline{toc}{section}{Bibliografia}
\begin{thebibliography}{99}




\bibitem{extension}Acerbi, E., Chiadò Piat, V., Dal Maso, G., Percivale, D.,: \emph{An extension theorem from connected sets, and homogenization in general periodic domains}. Nonlinear Anal., 18 (1992), no. 5, 481-496.

\bibitem{ABO} Alberti, G., Baldo S., Orlandi., G. \emph{Variational convergence for functionals of Ginzburg-Landau type.} Indiana Univ. Math. J. \textbf{54}, no. 5, 1411-1472 (2005).
\bibitem{ACP} Alicandro, R., Cicalese, M., Ponsiglione, M., \emph{Variational equivalence between Ginzburg-Landau, $XY$-spin systems and screw dislocations energies}, Indiana Univ. Math. J. \textbf{60}(1), 171-208 (2011).

\bibitem{ADLGP} Alicandro, R., De Luca, L., Garroni, A., Ponsiglione, M., \emph{Metastability and dynamics of discrete topological singularities in two dimensions: a $\Gamma$-convergence approach}, Arch. Ration. Mech. Anal. \textbf{214}(1), 269-330 (2014).

\bibitem{A.B.P.} Anzellotti, G., Baldo, S., Percivale, D.,: \emph{Dimension reduction in variational problems, asymptotic development in $\Gamma$-convergence and thin structures in elasticity}. Asymptotic Anal. 9(1), 61-100 (1994). 

\bibitem{BBS} Bacon, D. J., Barnett, D. M., Scattergood, R. O., \emph{Anisotropic continuum theory of elastic defects}, Progr. Mater. Sci. \textbf{23} 51-262, (1979).

\bibitem{BS} Barnet, D., Swanger, L., \emph{The elastic energy of a straight dislocation in an infinite anisotropic elastic medium}, Physica Status Solidi \textbf{23}, 51-262 (1979).

\bibitem{B.B.H.} Bethuel, F., Brezis, H., Hélein, F., \emph{Ginzburg-Landau Vortices.} Birkh\"auser, Basel, (1994).

\bibitem{B.B.O.}  Bethuel, F., Brezis, H., Orlandi, G., \emph{Asymptotics for the Ginzburg-Landau equation in arbitrary dimensions.} J. Funct. Anal., \textbf{186}, 432-520 (2001).

\bibitem{Bott} {Bott, R., Tu, L.W.,} Differential forms in algebraic topology, Springer Berlin (1982).

\bibitem{BrB} Bourgain, J., Brezis H., \emph{New estimate for the laplacian, the div-curl, and related Hodge systems}, C. R. Acad. Paris, Ser. I 338, 539-543 (2004)

\bibitem{CG} Cacace, S., Garroni, A., \emph{A multiphase transition model for dislocations with interfacial microstructure.} Interfaces Free Bound. \textbf{11}, no. 2, 291-316 (2009).

\bibitem{CL} Cermelli, P., Leoni, G., \emph{Renormalized energy and forces on dislocations} SIAM J. Math. Anal. \textbf{37}(4), 1131-1160 (2005).

\bibitem{C.G.M.} Conti, O., Garroni, A., Massaccesi, A.,: \emph{Modeling of dislocations and relaxation of functionals on 1-currents with discrete multiplicity}. Calc. Val. PDE (2014).

\bibitem{C.G.MU.1} Conti, S., Garroni, A., M\"uller, S., \emph{Singular Kernels, multiscale decomposition of microstructure, and dislocation models.} Arch. Rational Mech. Anal., Vol. \textbf{199}, no. 3, 779-819 (2011).

\bibitem{C.G.MU.2} Conti, S., Garroni, A., M\"uller, S., \emph{Dislocation microstructures and strain-gradient plasticity with one active slip plane.} J. Mech. Phys. Solids Vol. \textbf{93}, 240-251 (2016)

\bibitem{C.G.O.} Conti, S., Garroni, A., Ortiz, M.,:
\emph{The line-tension approximation as the dilute limit of linear-elastic dislocations}. Arch. Rational Mech. Anal. 218 (2015) 699-755.

\bibitem{DLGP} De Luca, L., Garroni, A., Ponsiglione, M., \emph{$\Gamma$-convergence analysis of systems of edge dislocations: the self energy regime}, Arch. Ration. Mech. Anal. \textbf{206}(3), 885-910 (2012).

\bibitem{FPP1} Fanzon, S., Palombaro, M., Ponsiglione, M., \emph{A variational model for dislocations at semi-coherent interfaces}, J. Nonlinear Sci. \textbf{27}(5), 1435-1461 (2017).

\bibitem{FPP2} Fanzon, S., Palombaro, M., Ponsiglione, M., \emph{Derivation of Linearised Polycrystals from a 2D system of edge dislocations}, Preprint  arXiv:1805.04484 (2018).

\bibitem{F.J.M.02} Friesecke, G., James, R.D., M\"uller, S.,: \emph{A Theorem on geometric rigidity and the derivation of nonlinear plate theory from three-dimensional elasticity}. Comm. Pure Appl. Math. 55 (2002) 1461-1506.

\bibitem{F.J.M.} Friesecke, G., James, R.D., M\"uller, S.,: \emph{A hierarchy of plate models derived from nonlinear elasticity by gamma-convergence}. Arch. Rational Mech. Anal. 180 (2006) 183-236.


\bibitem{GLP} Garroni, A., Leoni, G., Ponsiglione, M., \emph{Gradient theory for plasticity via homogenization of discrete dislocations} J. Eur. Math. Soc. \textbf{12}(5), 1231-1266 (2010).

\bibitem{GM} Garroni, A., M\"uller, S., \emph{A variational model for dislocations in the line tension limit.} Arch. Ration. Mech. Anal. \textbf{181}, no. 3, 535-578 (2006).

\bibitem{JJ} Ginster, J., \emph{Plasticity as the $\Gamma$-limit of a two-dimesional dislocation energy: the critical regime without the assumption of well-separateness } Arch. Rational Mech. Anal. 233 (2019) 1253-1288

\bibitem{HL} Hirth, J., and Lothe, J., Theory of dislocations. Wiley, 2ed., New-York, (1982).
\bibitem{Hu} Hudson, T., \emph{An existence result for Discrete Dislocation Dynamics in three dimensions.} arXiv:1806.00304.

\bibitem{HuO1} Hudson, T., Ortner, C.,  \emph{Existence and stability of a screw dislocation under anti-plane deformation}. Arch. Ration. Mech. Anal, 213 (2014) no. 3, 887-929.

\bibitem{HuO2} Hudson, T., Ortner, C.,  \emph{Analysis of stable screw dislocation configurations in an anti-plane lattice model.} SIAM J. Math. Anal. 47 (2015) no. 1, 291-320.

\bibitem{Je} Jerrard, R. L., Soner, H. M. \emph{The Jacobian and the Ginzburg-Landau energy.} Calc. Var. Partial Differential Equations \textbf{14}, no. 2, 151-191 (2002).

\bibitem{K.C.O.} Koslowsky, M., Cuiti\~no, A. M., Ortiz, M., \emph{A phase-field theory of dislocation dynamics, strain hardening and hysteresis in ductile single crystal.} J. Mech. Phys. Solids \textbf{50}, 2597-2635 (2002).

\bibitem{K.O.} Koslowsky, M., Ortiz, M., \emph{A multi-phase field model of planar dislocation networks.} Model. Simul. Mat. Sci. Eng. \textbf{12}, 1087-1097 (2004).

\bibitem{LL0} Lauteri, G., Luckhaus, S., \emph{An Energy estimate for Dislocation Configurations and the Emergence of Cosserat-Type Structures in Metal Plasticity.} Preprint, (2017). URL: https//arxiv.org/abs/1608.06155.

\bibitem{LL} Lauteri, G., Luckhaus, S., \emph{Geometric rigidity estimates for incompatible fields in dimension $\ge3$}. Submitted.

\bibitem{LMSZ} Lucardesi, I., Morandotti, M., Scala, R., Zucco, D., \emph{Confinement of dislocations inside a crystal with a prescribed external strain}, Riv. Mat. Univ. Parma, \textbf{9}(2),  283-327  (2018).

\bibitem{MPS} Mora, M.G., Peletier, M., Scardia, L., \emph{
Convergence of interaction-driven evolutions of dislocations with Wasserstein dissipation and slip-plane confinement},
SIAM J. Math. Anal. \textbf{49}(5), 4149-4205 (2017).

\bibitem{MRS}Mora, M.G., Rondi, L., Scardia, L., \emph{
The equilibrium measure for a nonlocal dislocation energy}, Com. Pure Appl. Anal., \textbf{72}(1), 136-158 (2019).


\bibitem{MP1} M\"uller, S., Palombaro, M., \emph{Existence of minimizers for a polyconvex energy in a crystal with dislocations}, Calc. Var. and Partial Differential Equations \textbf{31}(4) 473-482 (2008).

\bibitem{MP2} M\"uller, S., Palombaro, M., \emph{Derivation of a rod theory for biphase materials with dislocations at the interface}, Calc. Var. and Partial Differential Equations \textbf{48}  315-335 (2013).

\bibitem{MSZ} M\"uller, S., Scardia, L., Zeppieri, C. I., \emph{Geometric rigidity for incompatible fields and an application to strain-gradient plasticity}, Indiana Univ. Math. J. \textbf{63}, 1365-1396 (2014).

\bibitem{Pons} Ponsiglione, M., \emph{Elastic energy stored in a crystal induced by screw dislocations: from discrete to continuous}, SIAM J. Math. Anal. \textbf{39}(2), 449-469 (2007).

\bibitem{SS} Sandier, E., Serfaty, S., \emph{Vortices in the Magnetic Ginzburg-Landau Model.} Birkh\"auser (2007).

\bibitem{SS1} Sandier, E., Serfaty, S., \emph{From the  Ginzburg-Landau Model to Vortex Lattice Problems.} Com. Math. Phys., 313 (2012), 635-743

\bibitem{SV1} Scala, R., Van Goethem, N., \emph{Analytic and geometric properties of dislocation singularities}, To appear Proc. A Royal Soc. Edinburgh, DOI:10.1017/prm.2018.57 (2019).

\bibitem{SV2} Scala, R., Van Goethem, N., \emph{A variational approach to single crystals with dislocations}, SIAM J. Math. Anal. \textbf{51}(1), 489-531 (2019).

\bibitem{s.z.} Scardia, L., Zeppieri, C.,: \emph{Line-tension model for plasticity as the $\Gamma$-limit of a nonlinear dislocation energy}. SIAM J. Math. Anal. \textbf{44}, 2372-2400 
(2012).

\end{thebibliography}
\end{document}